\documentclass[11pt]{a_a_win}

\usepackage[T2A]{fontenc} 
\usepackage[cp1251]{inputenc} 
\usepackage[english,russian]{babel}
\usepackage{amsfonts,mathrsfs,amscd} 
\usepackage{amssymb}
 \usepackage{amsmath} 
\usepackage{enumerate} 
\usepackage{verbatim}
\usepackage[mathscr]{eucal} 
\usepackage[dvips]{graphicx} 

\usepackage[pdftex,unicode,colorlinks,linkcolor=blue,citecolor=red,bookmarksopen,pdfhighlight=/N]{hyperref}

\usepackage[hyper]{amsbib}
\usepackage{esint}

\theoremstyle{plain} 
\newtheorem{theorem}{Теорема}

\newtheorem*{thAa}{Теорема A1}
\newtheorem*{thAb}{Теорема A2} 
\newtheorem*{thAc}{Теорема A3}

\newtheorem*{Thopi}{Теорема о повторных интегралах}
\newtheorem{lemma}{Лемма}[section]
\newtheorem{propos}{Предложение}[section] 
\newtheorem{corollary}{Следствие}[section]

\theoremstyle{definition}
\newtheorem{definition}{Определение}[section] 
\newtheorem{remark}{Замечание}[section] 
\newtheorem{example}{Пример}[section]

\theoremstyle{plain} 
\newtoks\thehProclaim 
\newtheorem*{Proclaim}{\the\thehProclaim}

\theoremstyle{definition} 
\newtoks{\thehRemark} \newtheorem*{Remark}{\the\thehRemark}

\renewcommand{\leq}{\leqslant} 
\renewcommand{\geq}{\geqslant}
\newcommand{\mc}{\mathcal}

\newcommand{\rad}{\text{\tiny\rm rad}}
\newcommand{\bal}{\rm {bal}}
\newcommand{\RR}{\mathbb{R}} 
\newcommand{\CC}{\mathbb{C}} 
\newcommand{\NN}{\mathbb{N}} 
\newcommand{\ZZ}{\mathbb{Z}}
\newcommand{\DD}{\mathbb{D}} 
\newcommand{\e}{\varepsilon}
\newcommand{\const}{{\rm const}}

\DeclareMathOperator{\clos}{clos} 

\DeclareMathOperator{\Har}{har} 
\DeclareMathOperator{\Hol}{Hol} 
\DeclareMathOperator{\Poi}{P} 

\DeclareMathOperator{\Arg}{Arg}
\DeclareMathOperator{\var}{var} 

\DeclareMathOperator{\Exp}{Exp} 
\DeclareMathOperator{\Zero}{Zero} 
\DeclareMathOperator{\sbh}{sbh}

\DeclareMathOperator{\supp}{supp} 

\DeclareMathOperator{\type}{type} 
\DeclareMathOperator{\ord}{ord}

\DeclareMathOperator{\up}{\text{\rm \tiny up}}
 \DeclareMathOperator{\lw}{\text{\rm \tiny low}}
 
\DeclareMathOperator{\loc}{loc}

\DeclareMathOperator{\dd}{\,{\mathrm d\!}}
\renewcommand{\Re}{{\rm Re \,}}
\renewcommand{\Im}{{\rm Im \,}}

\begin{document} 
\title[Выметание на систему лучей]{Выметание мер и
субгармонических функций на систему лучей. I. Классический случай}
	
	\author{Б.\,Н.~Хабибуллин, А.\,В.~Хасанова} 
	
\address{факультет математики и ИТ\\ Башкирский государственный университет\\ 450074, г. Уфа\\ ул. Заки Валиди, 32\\ Башкортостан\\
Россия}
	
	\email{Khabib-Bulat@mail.ru}
	
	\subjclass[2010]{Primary 30D15; Secondary 30D35, 41A30, 31A05}
	
\keywords{целая функция, последовательность нулей, субгармоническая функция, мера Рисса, выметание}	
\begin{abstract} Развиваются классические выметания мер и субгармонических функций на систему лучей $S$ с общим началом на комплексной плоскости $\CC$. Это позволяет для субгармонической функции $v$ на $\CC$ строить также субгармоническую на $\CC$ функцию, гармоническую вне $S$ и совпадающую с $v$ на $S$.  Приводятся применения к исследованию   взаимосвязи роста целой функции на $S$ с распределением её  нулей, условий вполне регулярного роста целых и субгармонических функций на системе лучей, к вопросам неполноты экспоненциальных  систем в пространствах голоморфных функций в невыпуклых неограниченных открытых множествах, сужающихся вблизи бесконечности. 
Настоящая первая часть работы содержит и необходимый подготовительный материал для построения нового типа выметания конечного рода  на $S$ во второй части.
	
Библиография: 35 названий 
\end{abstract}
	
\thanks{Работа поддержана Российским фондом фундаментальных исследований, проект №~16--01--00024a}	
\date{18 августа 2018 г.}
	
\maketitle

\tableofcontents

\section{Введение}\label{s10}
\subsection{Истоки и предмет исследования}\label{111}

В контрасте со многими классическими математическими результатами и методами,  авторство которых часто сложно  установить, можно с уверенностью сказать, что первооткрывателем методов  выметания\footnote{balayage или sweeping (out) method} меры
и потенциала на границу области, или из области, был Анри Пуанкаре (конец XIX века \cite{P1}--\cite[гл.~VII]{P2}). В первой половине XX века его метод классического выметания был существенно развит Шарлем Валле-Пуссеном \cite{V-P}. Дальнейший  список
авторов и работ, причастных к развитию техники выметания, можно частично извлечь из монографий  \cite{L} и  \cite{BH}. В последней книге наибольшее внимание уделено абстрактной форме выметания. Несколько иная форма абстрактного выметания разрабатывалась и широко использовалась нами в  \cite{Kha01}--\cite{Khab011} применительно к конкретным задачам теории  функций. Выметание на замкнутую систему лучей с общим началом на комплексной плоскости $\CC$ рассматривалось и использовалось  первым из соавторов в  \cite{Kha88}--\cite{Khsur}. Специфика таких систем лучей $S$, рассматриваемых и как точечные подмножества в $\CC$, --- некомпактность и неограниченность $S$ в $\CC$; бесконечносвязное, в общем случае, дополнение $\CC\setminus S$; возможный быстрый рост вблизи бесконечности характеристик выметаемых на систему лучей мер и субгармонических функций. На основе выметания на систему лучей в \cite{Kha88}--\cite{Khsur} был получен ряд новых результатов: {\it критерий вполне регулярного роста целой и/или субгармонической функции на системе лучей\/} в терминах распределения её нулей/масс; полный перенос теоремы Рубела\,--\,Мальявена \cite{MR}--\cite[22]{RC}  о минимально возможной ширине индикаторной диаграммы целой функции экспоненциального типа с {\it положительными нулями\/} на функции с {\it произвольными комплексными нулями\/} с существенными усилениями и обобщениями;
 геометрический {\it критерий полноты экспоненциальной системы\/} в пространстве голоморфных функций {\it в  выпуклой неограниченной области\/} в естественных топологиях и многие другие. Для недоступных классическому выметанию ситуаций использовалось и новое на тот период глобальное выметание рода $1$ \cite[\S~3]{kh91AA}, \cite[гл.~II]{KhDD92}, когда выметание меры даёт заряд, т.\,е.
вещественнозначную меру, а выметание субгармонической функции --- уже разность субгармонических функций, или $\delta$-субгармоническая функция. В настоящей первой части работы мы возвращаемся к классическому выметанию на системы лучей с общей начальной точкой. При этом предшествующие результаты существенно совершенствуется, особенно в количественном аспекте, на основе точных элементарных верхних и нижних оценок гармонической меры для полуплоскости (угла) из пунктов \ref{p_3_1}, \ref{s:dal}. При выметании меры и субгармонической функции рассматриваем не только меры и функции конечного порядка роста, но и меры и функции бесконечного порядка роста. В настоящей первой части
работы преследуются две цели: 1) максимально реализовать возможности классического выметания на систему лучей, когда выметание положительной меры и субгармонической функции также соответственно положительная мера и субгармоническая функция; 2) подготовить технику и методы для развития упоминавшегося выше глобального выметания рода $1$ до выметания конечного рода $q=0,1, 2, \dots$,
применимого к любым мерам и субгармоническим функциям конечного порядка роста. Очень сжатое описание схемы построения выметания любого конечного рода было приведено в \cite[заключительная часть \S~2.1, стр.~93--95]{KhDD92}. При этом классическое выметание оказывается выметанием рода $0$. Кроме того, имея ввиду вторую цель, некоторые понятия и вспомогательные результаты даны здесь в форме более
общей, чем необходимо для первой части работы. В следующем п.~\ref{1_2b} мы даём сводку новых и полученных одним из авторов ранее основных результатов о классическом выметании на произвольную систему лучей в модельном случае системы из двух противоположно направленных лучей. Из сводки видны как связи наших результатов с классическими понятиями теории целых функций \cite{RC}--\cite{GO}, так и возможные применения выметания на систему лучей к теории роста целых функций и к аппроксимирующим свойствам экспоненциальных  систем в пространствах  голоморфных функций. Используется терминология, обозначения и понятия из \cite{KhDD92}--\cite{GM}. 

Авторы глубоко признательны рецензенту за ряд ценных замечаний, поправок и уточнений.

\subsection{Выметание на вещественную ось $\RR$}\label{1_2b} 
Классического выметание на произвольную систему лучей с началом в нуле иллюстрирует достаточно простой в формулировках, но уже вполне содержательный модельный случай выметания на систему $S=\{\RR^+, \RR^-\}$ лишь из двух лучей $\RR^+:=\{x\in \RR\colon x\geq 0\}$ и $\RR^-:=-\RR^-:=\{x\in \RR\colon x\leq 0\}$ --- замкнутые соответственно положительная и отрицательная полуоси. Иначе говоря, $S=\RR$, если одновременно рассматривать систему лучей $\{\RR^+, \RR^-\}$ и как точечное подмножество $\RR^-\cup \RR^+=\RR\subset \CC$.
 
\subsubsection{Выметание последовательности точек}\label{b_seq} Всюду в этом п.~\ref{1_2b} последовательность точек ${\sf Z}=\{{\sf z}_k\}_{k=1,2,\dots}$ из $\CC$ не имеет точек сгущения в $\CC$. 
Последовательности точек $\sf Z$ сопоставляем считающую меру $n_{\sf Z}$, равную на каждом подмножестве $S\subset \CC$ числу точек из $\sf Z$, содержащихся в $S$. Для интервала $I\subset \RR$ через $\omega\bigl( {\sf z}, I \bigr)\geq 0$ обозначаем делённый на $\pi$ угол, под которым виден интервал $I$ из точки ${\sf z}\in \CC$. Меру $n_{\sf Z}^{\bal}$ на $\CC$ с носителем $\supp n_{\sf Z}^{\bal}\subset \RR$ и определяющей её возрастающей функцией распределения 
\begin{equation}\label{nuRA0} {(n_{\sf Z}^{\bal})}^{\RR}(x):=\begin{cases}
\sum\limits_{k}\omega\bigl( {\sf z}_k, [0,x] \bigr) \; &\text{ при } x\geq 0,
\\ -\sum\limits_{k}\omega\bigl( {\sf z}_k, [x,0) \bigr)
\; &\text{ при } x<0 , 
\end{cases} 
\end{equation} 
называем {\it выметанием\/} меры $n_{\sf Z}$ или последовательности $\sf Z$ {\it
на\/} $\RR$. Для ${(n_{\sf Z}^{\bal})}^{\RR}(x)$ в \eqref{nuRA0} допускаются и значения $\pm \infty$.

Для меры $\nu$ на $\CC$ при $r\in \RR^+$ через $\nu^{\rad}(r)$ обозначаем $\nu$-меру замкнутого круга с центром в нуле радиуса $r$.
\begin{thAa}[{\rm см. \cite{Kha91} и теоремы \ref{th:Sb}--\ref{th:Lc} ниже}]\label{thA}
Попарно эквивалентны следующие три утверждения: 
\begin{enumerate}[{\rm (i)}] 
\item\label{bR1} 
выметание $\sf Z$ на $\RR$ --- положительная борелевская мера, т.\,е. функция распределения из \eqref{nuRA0} конечна при любом $x\in \RR$;  

\item\label{bR2} значение функции из \eqref{nuRA0} конечно при некотором $x\neq 0$; 

\item\label{bR3} для верхней и
нижней полуплоскостей выполнено условие Бляшке \begin{equation}\label{Adz0} \sum_{{\sf z}_k\neq 0} \Bigl|\Im \frac{1}{{\sf
z}_k}\Bigr|<+\infty. 
\end{equation} 

\end{enumerate}

В условиях \/ \eqref{bR1}--\eqref{bR3} функция распределения из \eqref{nuRA0} удовлетворяет ус\-л\-о\-в\-ию Липшица на каждом
отрезке $[x_1,x_2]\subset \RR\setminus \{0\}$, не содержащем точек последовательности ${\sf Z}$, и \begin{enumerate} \item[{\rm
(iv)}]\label{bR4} если для $g\colon \RR^+\to \RR^+$ имеем $g(r)>r$ при всех $r>0$ , то\footnote{Здесь и далее
ссылки над знаками бинарных отношений означают, что эти соотношения как-то связаны с приведёнными ссылками, например, вытекают из них.} 
\begin{multline*}
{(n_{\sf Z}^{\bal})}^{\RR}(r)-{(n_{\sf Z}^{\bal})}^{\RR}(-r)
\overset{\eqref{nuRA0}}{=} (n_{\sf Z}^{\bal})^{\rad}(r) \\ \leq (n_{\sf Z})^{\rad}\bigl(g(r)\bigr)+\frac{2rg^2(r)}{\pi (g(r)-r)^2}
\sum_{|{\sf z}_k|\geq g(r)} \Bigl|\Im \frac{1}{{\sf z}_k}\Bigr| \quad \text{при любых $r>0$}.
\end{multline*} 
\end{enumerate}

Пусть, в дополнение к \eqref{bR1}--\eqref{bR3}, $\sf Z$ --- последовательность конечной верхней плотности при порядке\/
$p\in \RR^+$, т.\,е.\/ $(n_{\sf Z})^{\rad}(r)=O(r^p)$ при $0\leq r\to +\infty$. Тогда 
\begin{enumerate}
 \item[{\rm (v)}]\label{bR5}
выметание\/ $n_{\sf Z}^{\bal}$ меры $n_{\sf Z}$ --- мера конечной верхней плотности при порядке\/ $p$, т.\,е.\/ $(n_{\sf
Z}^{\bal})^{\rad}(r)=O(r^p)$ при $r\to +\infty$, а при $p\geq 1$ имеет место и точная оценка \begin{equation*}
\limsup_{0\leq r\to+\infty}\frac{(n_{\sf Z}^{\bal})^{\rad}(r)}{r^p} \leq \limsup_{0\leq r\to+\infty}\frac{(n_{\sf
Z})^{\rad}(r)}{r^p}\,; 
\end{equation*} 
\item[{\rm (vi)}]\label{bR6} если $p$ --- натуральное число и для $\sf Z$ выполнено условие
Линделёфа 
\begin{equation*}
\biggl|\sum_{1\leq |{\sf z}_k|<r} \frac{1}{{\sf z}_k^p}\biggr|=O(1)\quad \text{при $r\to
+\infty$,}
\end{equation*} 
то оно выполнено и для  выметания $n_{\sf Z}^{\bal}$, а именно: при $r\to +\infty$
\begin{equation*}
\biggl|\,\int\limits_{1\leq |z|<r} \frac{1}{z^p} \dd n_{\sf Z}^{\bal}(z)\biggr| 
\overset{\eqref{nuRA0}}{=} \left|\biggl(\,\int^{-1}_{-r} +\int_{1}^r \,\biggr)\, \frac{\dd \,(n_{\sf
Z}^{\bal})^{\RR}(x)}{x^p} \right| =O(1); 
\end{equation*} 
\item[{\rm (vii)}]\label{bR7} если
последовательность $\sf Z$ асимптотически отделена от\/ $\RR$, что означает $\liminf\limits_{k\to \infty} {|\Im {\sf z}_k|}/{|{\sf z}_k|}>0$, то для некоторых чисел $b, r_0>0$ при любых $-\infty<t_1<t_2<+\infty$, $t_1t_2>0$, из $|t_1+t_2|\geq r_0$ и $|t_2-t_1|\leq {|t_1+t_2|}/{2}$ следует неравенство 
\begin{equation*}
\bigl|(\nu^{\bal})^{\RR}(t_2)-(\nu^{\bal})^{\RR}(t_1)\bigr| \leq
b|t_2-t_1| |t_1+t_2|^{p-1} . 
\end{equation*} 
\end{enumerate}
\end{thAa} 

\subsubsection{Выметание функции}\label{bfR} В этом пп.~\ref{bfR} всюду $f\colon \CC\to \CC$ --- ненулевая целая функция с
последовательностью нулей $\Zero_f=:{\sf Z}$, перенумерованной с учетом кратности этих нулей, и $v:=\log |f|$. Функция $f$ принадлежит
классу A \cite{Levin56}, или классу Ахиезера, если последовательность $\sf Z$ удовлетворяет условию Бляшке \eqref{Adz0}. Нетрудно 
показать (см. предложение \ref{LKcr} ниже), что {\it функция $f$ принадлежит классу\/ {\rm A,} если и только если} 

\begin{equation*}\label{Akhi} \hspace{-2mm}\sup_{r\geq 1}\left(\int_{1}^r\Bigl(\,\frac{1}{x^2}-\frac{1}{r^2}\Bigr)\log\bigl|f(-x)f(x)\bigr| \dd x
+\frac{1}{r}\int_0^{2\pi}\log \bigl|f(re^{i\theta})\bigr| |\sin \theta| \dd \theta\right)<+\infty, 
\end{equation*} что {\it для целой  функции $f$ экспоненциального типа,\/} т.\,е. при выполнении  условия $\limsup\limits_{z\to\infty} \frac{\log |
f(z)|}{|z|}<+\infty,$ эквивалентно любому из следующих двух условий \cite[гл.~V, \S~1, теоремы 2, 3]{Levin56}
\begin{equation*}\label{Akhi1} 
\sup_{r\geq 1}\int_{1}^r\frac{\log\bigl|f(-x)f(x)\bigr|}{x^2} \dd x <+\infty, \qquad
\sup_{r\geq 1}\, \left|\,\int_{1}^r\frac{\log\bigl|f(-x)f(x)\bigr|}{x^2} \dd x\right| <+\infty. 
\end{equation*} Любая целая функция $f$ порядка $<1$, т.\,е. при $\limsup\limits_{z\to\infty} \frac{\log \log (1+| f(z)|)}{\log |z|}<1,$ принадлежит классу A.
 
{\it Выметание на $\RR$ субгармонической на\/ $\CC$ функции $v=\log |f|$\/} определяем как субгармоническую функцию $v^{\bal}$ на $\CC$, гармоническую на $\CC \setminus \RR$, сужение которой $v^{\bal}\bigm|_{\RR}$ на $\RR$ совпадает с сужением $\log|f|\bigm|_{\RR}$ на $\RR$. Оно, конечно же, не единственно. Содержание теорем \ref{th:bs}--\ref{thmA} иллюстрирует 

\begin{thAb}\label{bal:Af} 
Пусть $f$ --- целая функция класса\/ {\rm A}. Тогда 
\begin{enumerate}[{\rm (i)}] 
	\item\label{fA1} существует выметание  $v^{\bal}$ функции $v=\log |f|$ на\/ $\RR$,
для которого функция распределения на $\RR$ меры Рисса $\nu_{v^{\bal}}$ 
функции $v^{\bal}$ 
с носителем $\supp \nu_{v^{\bal}}\subset \RR$ совпадает с $(n_{\sf Z}^{\bal})^{\RR}$ из \eqref{nuRA0}; 
\item\label{fA2} для целой функции $f$ конечного типа при порядке $p\in \RR^+$, т.\,е. при  $\limsup\limits_{z\to \infty} \frac{\log |f(z)|}{|z|^{p}}<+\infty,$
 выметание $v^{\bal}$ --- субгармоническая функция конечного типа при том же порядке $p$, т.\,е. $\limsup\limits_{z\to \infty} \frac{v^{\bal}(z)}{|z|^p}<+\infty.$ 
\end{enumerate}

Целая функция $f$ класса\/ {\rm A} конечного типа при порядке $p\in \RR^+$ имеет вполне регулярный рост на $\RR$ при порядке $p$, т.\,е. для некоторого $E\subset \RR$ существуют два конечных предела 
\begin{equation}\label{Enl} 
\lim_{\substack{x\to \pm \infty\\x\notin E\subset \RR}} \frac{\log |f(x)|}{|x|^p}\, , \quad \text{где}\quad \lim_{0\leq r\to +\infty}\frac{1}{r}\int\limits_{E\cap [-r,r]} \dd x=0, 
\end{equation} 
если и только если 
\begin{enumerate} 
\item[{\rm (iii)}]\label{fA3} в случае $p<1$ существуют два конечных предела $\lim\limits_{x\to \pm \infty}\frac{{(n_{\sf
Z}^{\bal})}^{\RR}(x)}{|x|^p}$; 
\item[{\rm (iv)}]\label{fA4} в случае $p\geq 1$ в обозначениях $[p]$ и $\{p\}$ соответственно для целой и дробной частей числа $p$ существуют два конечных предела 
\begin{equation}\label{Enln}
 \lim_{\substack{x\to \pm \infty\\x\notin E\subset \RR}}|x|^{1-\{p\}} \fint\limits_{\RR\setminus [-1,1]}\frac{1}{x-t}\frac{(n_{\sf Z}^{\bal})^{\RR}(t)}{t^{[p]+1}} \dd t,
\end{equation} 
где $E$ --- некоторое множество нулевой относительной линейной меры на $\RR$, т.\,е. такое же, как в \eqref{Enl}, а
перечёркнутый интеграл $\fint$ задаёт главное значение интеграла в смысле Коши в точке $x$. 
\end{enumerate}
Пусть целая функция $f$ по-прежнему принадлежит классу\/ {\rm A} и 
\begin{equation}\label{fsr} 
\limsup_{0<r\to +\infty} \frac{1}{2\pi
r}\int_0^{2\pi} \log \bigl|f(re^{i\theta})\bigr| \dd \theta \leq b \quad\text{для некоторого числа $b\geq 0$}; 
\end{equation} 
в частности,
условие \eqref{fsr} выполнено для любой целой функции $f$ экспоненциального типа. Тогда 
\begin{enumerate} 
	\item[{\rm (v)}]\label{fA5} для любого числа $\e\in (0,1]$ найдётся ненулевая целая функция,  
для которой $F:=fg$ --- целая функция экспоненциального типа с индикатором роста
$h_{F}(\theta):=\limsup\limits_{0<r\to +\infty}\frac{\log |F(re^{i\theta})|}{r}$, $\theta \in [0,2\pi]$, при порядке $1$, удовлетворяющим оценкам 
\begin{equation}\label{hF}
 h_F(0)+h_F(\pi)\leq 10^3b\e, \quad h_F(-\pi/2)+h_F(\pi/2)\leq 10^3b\,\frac{1}{\e} \,.
\end{equation} 
\end{enumerate} 
\end{thAb} 

\subsubsection{О неполноте экспоненциальных систем}\label{sec:inc_s} 

Последовательности $\sf Z$ сопоставим экспоненциальную систему 
\begin{equation*} \Exp^{i\sf Z}:=\Bigl\{ z^m\exp(i{\sf
z}_kz)\colon 0\leq m\leq n_{\sf Z}\bigl(\{{\sf z}_k\}\bigr)-1\Bigr\}. 
\end{equation*} 
\begin{thAc}[{см. и ср. с \cite[следствие
3]{Kha01l}}] Пусть открытое множество  $\mathcal O\subset \CC$ содержит луч $l$, параллельный   $\RR$, и вместе с каждой точкой $x+iy_0\in l$ содержит интервал $\bigl\{x+iy\colon -y_-(x)<y-y_0<y_+(x)\bigr\}$, параллельный мнимой оси $i\RR$, а также  $|x|y_-(x)\to +\infty$ или $|x|y_+(x)\to +\infty$ при $x\to +\infty$ или при $x\to -\infty$.
Пусть  $\sf Z\subset \CC$  --- последовательность конечной верхней плотности при порядке $1$. 
Если $\sf Z$ удовлетворяет условию Бляшке \eqref{Adz0},  то система $\Exp^{i\sf Z}$ не полна в пространстве $\Hol(\mathcal O)$ голоморфных функций в $\mathcal O$ в топологии равномерной сходимости на компактах. 
\end{thAc}

\subsection{Основные обозначения, определения и соглашения}\label{s1.1}  К этому подразделу~\ref{s1.1} можно обращаться по мере необходимости.
\subsubsection{Множества, порядок, топология}\label{1_1_1} Как обычно, $\NN=\{1,2,\dots\}$ и $\ZZ$ --- множества {\it натуральных\/} и {\it целых чисел,\/} $\NN_0:=\{0\}\cup \NN$;  $\CC_{\infty}:=\CC\cup\{\infty\}$ --- {\it расширенная комплексная плоскость,\/} или компактификация  Александрова комплексной плоскости $\CC$; $\CC^{\up}:= \{z\in \CC \colon \Im z>0\}$ и $\CC_{\lw}:=-\CC^{\up}=\{z\in \CC\colon -z\in \CC^{\up}\}$ соответственно {\it верхняя и нижняя полуплоскости,\/} а  $ \CC^{\overline \up}:=\CC^{\up}\cup \RR$ и $\CC_{\overline \lw}:=\CC_{\lw}\cup \RR$ --- их замыкания  в $\CC$;  $\RR_{-\infty}:=\{-\infty\}\cup \RR$, $\RR_{+\infty}:=\RR\cup \{+\infty\}$;  $\RR_{\pm\infty}:=\RR_{-\infty}\cup\RR_{+\infty}$ --- {\it расширенная
вещественная прямая\/} с естественными отношением порядка $-\infty \leq x\leq +\infty$ для любого $x\in \RR_{\pm\infty}$ и порядковой
топологией, но $\RR_{\infty}=\RR\cup \{\infty\}\subset \CC_{\infty}$ --- {\it замыкание $\RR$ в $\CC_{\infty}$}. 
На  $S\subset \CC_{\infty}$ индуцируется топология с $\CC_{\infty}$, но при рассмотрении
$S\subset \RR_{\pm\infty}$ как подмножества в $\RR_{\pm\infty}$ уже с $\RR_{\pm\infty}$. Для подмножеств $S\subset \CC_{\infty}$ часто используем обозначение 
$S_*:=S\setminus \{0\}$--- {\it<<проколотое>> в нуле\/} $S$; в частности, $\RR_*=\RR\setminus \{0\}$, $\RR_*^+=\RR^+\setminus \{0\}$.  Для $z\in \CC$ и $S\subset \CC$ через $\bar z$ обозначаем сопряжённое комплексное число. Одним и тем же символом $0$ обозначаем, по контексту, число нуль, нулевой вектор, нулевую функцию, нулевую меру (заряд) и т.\,п.; $\varnothing$ --- {\it пустое множество.\/} Для подмножества $X$ упорядоченного векторного пространства чисел, функций, мер и т.\,п. с
отношением порядка \;$\geq$\; полагаем $X^+:=\{x\in X\colon x \geq 0\}$ --- все положительные элементы из $X$; $x^+:=\max \{0,x\}$.
{\it Положительность\/} понимается, в соответствии с контекстом, как $\geq 0$. Для $S\subset \CC_{\infty}$ через $\clos S$ и $\partial S$ обозначем соответственно {\it замыкание\/} и {\it границу\/} подмножества $S$ в $\CC_{\infty}$. 

Для $r\in \RR_{\pm\infty}$ и $z\in \CC$ полагаем $ D(z,r):=\{z' \in \CC \colon |z'-z|<r\}$ --- {\it открытый круг с центром\/ $z$
радиуса\/ $r$.} Так, $D(z,r)=\varnothing$ при $r\leq 0$ и $D(z,+\infty):= \CC$. В частном случае $D(r):=D(0,r)$; $\DD:=D(1)$ --- {\it
единичный круг.\/} Для $z=\infty$ удобно принять $D(\infty,r):=\{\infty\}\cup \{z\in \CC \colon |z|>1/r\}$ и $D(\infty, +\infty):=\CC_{\infty}\setminus \{0\}$. Кроме того, $\overline{D}(z,r):=\clos D(z,r)$ --- {\it замкнутый круг;} $\overline{D}(z,+\infty):=\CC_{\infty}$, но $\overline{D}(z,0):=\{z\}$; $D_*(z,r):=\bigl(D(z,r)\bigr)_*$, $D_*(r):=D_*(0,r)$, $\overline{D}_*(z,r):=\bigl(\overline{D} (z,r)\bigr)_*$, $\overline{D}_*(r):=\overline{D}_ *(0,r)$.

\subsubsection{Функции}\label{1_1_2}  Функция $f\colon X\to Y$ с упорядоченными $(X,\leq)$ и
$(Y,\leq)$ {\it возрастающая}, если для $x_1,x_2\in X$ из $x_1\leq x_2$ следует $f(x_1)\leq f(x_2)$. Аналогично для убывания. В основном  $Y\subset \RR_{\pm \infty}$. {\it На множествах функций\/} с упорядоченным
множеством значений {\it отношение порядка\/} индуцируется с множества значений как {\it поточечное.}

Для  $X\ni x_0$ и  $f\colon X\setminus \{x_0\}\to \RR_{\pm\infty}$ записи <<$f(x)\leq O(1)$ {\it  при\/} $x\to x_0$>> или $f(x)\underset{x\to x_0}{\leq} O(1)$  означает, что функция $f$ в какой-то  проколотой окрестности точки $x_0$ ограничена сверху некоторым числом из $\RR$. 

Пусть $S\subset \CC_{\infty}$. Классы $\Hol(S)$ и  $\sbh (S)$ состоят из сужений на $S$ функций, соответственно {\it голоморфных\/} и {\it субгармонических\/}  в каком-либо открытом множестве,  включающем в себя $S$;  
$\Har (S):=\sbh (S)\cap \bigl(-\sbh(S)\bigr)$ --- класс {\it гармонических\/} функций на $S$;
$\Hol_*(S):=\Hol(S)\setminus \{0\}$.  Тождественную $-\infty$ на $S$ обозначаем  $\boldsymbol{-\infty}\in \sbh (S)$; $\sbh_*(S):=\sbh (S)\setminus \{\boldsymbol{-\infty}\} $.

Для чи\-с\-ла $r >0$ определим  интегральное среднее на  окружности $\partial \overline D(r)$ от 
про\-и\-з\-в\-о\-л\-ь\-н\-ой фу\-н\-к\-ц\-ии $ v\colon \partial \overline D(r)\to \RR_{\pm\infty}$ как
(см. \eqref{fsr})
\begin{equation}\label{df:MCB}
C(r;v):=\frac{1}{2\pi} \int_{0}^{2\pi}  v(re^{i\theta}) \, d \theta
\end{equation}
--- конечно же, при условии существования интеграла справа \cite[определение 2.6.7]{Rans}, для которого, вообще говоря, допускаются значения $\pm\infty$. 

Число $C$ и постоянную функцию, тождественно равную $C$, не различаем. Через $\const_{a_1, a_2, \dots}$ обозначаем вещественные постоянные, зависящие от $a_1, a_2, \dots$ и, если не
оговорено противное, только от них. 

\subsubsection{Меры и заряды}\label{1_1_3} Далее $\mathcal M (S)$ --- класс всех {\it счетно-аддитивных фу\-н\-к\-ций борелевских
подмножеств борелевского множества $S\subset \CC_{\infty}$ со значениями в $\RR_{\pm\infty}$, конечных на компактах из $S$.\/}
Элементы из $\mathcal M (S)$ называем {\it зарядами,\/} или {\it вещественными мерами,\/} на $S\subset \CC_{\infty}$; $\mathcal M^+ (S):=\bigl(\mathcal M (S)\bigr)^+$ --- подкласс {\it положительных мер.\/} Заряд $\mu \in \mathcal M (S)$ {\it сосредоточен\/} на
$\mu$-измеримом подмножестве $S_0\subset S$, если $\mu (S')=\mu (S'\cap S_0)$ для любого $\mu$-измеримого подмножества $S'\subset S$.
Очевидно, заряд $\mu$ сосредоточен на носителе $\supp \mu\subset S$. Для измеримого по $\mu\in \mathcal M(S)$ подмножества $S_0\subset
S$ через $\mu\bigm|_{S_0}$ обозначаем {\it сужение заряда $\mu$ на $S_0$.\/} Через $\lambda_{\CC}$ и $\lambda_{\RR}$ обозначаем {\it
лебеговы меры\/} соответственно на $\CC$ и $\RR$,  а также их сужения на подмножества из $\CC$. Нижний индекс $\CC$ в $\lambda_{\CC}$ при
этом часто опускаем. Через $\delta_z$ обозначаем {\it меру Дирака в точке $z\in \CC_{\infty}$,\/} т.\,е. вероятностную меру с $\supp \delta_z=\{z\}$. Как обычно, для $\mu\in \mathcal M(\CC)$ через
$\mu^+:=\max\{0,\mu\}$, $\mu^-:=(-\mu)^+$ и $|\mu|:=\mu^++\mu^-$, обозначаем {\it верхнюю, нижнюю\/ {\rm и} полную вариации заряда\/
$\mu=\mu^+ -\mu^-$.} Для $\nu\in \mathcal M (S )$ и $D (z,r)\subset S$ 
\begin{equation}\label{df:nup} 
\nu (z,r):=\nu \bigl(\,\overline D(z,r)\bigr),\quad \nu^{\rad}(r):=\nu(0,r) 
\end{equation} --- {\it считающая функция с
центром\/ $z$} и {\it считающая радиальная функция} заряда $\nu$ соответственно. 
Для $\nu\in \mathcal M(\CC)$ используем
и {\it функцию распределения\/} $\nu^{\RR}$ сужения заряда $\nu$ {\it на\/} $\RR$ (частный случай см. в
\eqref{nuRA0}): 
\begin{equation}\label{nuR} \nu^{\RR}(x):=\begin{cases} -\nu\bigl( [x,0) \bigr) \; &\text{ при } x<0 ,\\ \nu\bigl(
[0,x] \bigr) \; &\text{ при } x\geq 0. \end{cases} 
\end{equation}

Для открытого $\mc O\subset \CC_{\infty}$ борелевскую меру Рисса функции ${u}\in \sbh_* ({\mathcal O})$ обозначаем как $\nu_{u}:=\frac{1}{2\pi}\Delta {u}\in \mathcal M ^+(\mathcal O)$ или
$\mu_u$ и т.\,п., где оператор Лапласа $\Delta$ действует в смысле теории обобщённых функций. 
Для функции  $\boldsymbol{-\infty}$ её мера Рисса по
определению равна $+\infty$ на любом подмножестве из ${\mathcal O}$.

\section{Шкалы роста и исчезания функции}\label{sss:sh} 
\setcounter{equation}{0} 
Основная цель этого \S~\ref{sss:sh} ---
систематизировать в целом известные элементарные факты, изложения и применения которых разбросаны по различным источникам, а также
согласовать определения и понятия. Определения и утверждения этого \S~\ref{sss:sh} не зависят от выбора числа $r_0>0$. 
\subsection{Рост около бесконечности}
Для $f\colon [r_0,+\infty) \to \RR$ определим  
\begin{subequations}\label{se:nu0} 
\begin{align} \ord_{\infty}[f]&:=\limsup_{r\to +\infty} \frac{\log \bigl(1+f^+(r)\bigr)}{\log r}\in
\RR^+\cup \{+\infty\} 
\tag{\ref{se:nu0}o}\label{senu0:a} \\ \intertext{{\it --- порядок роста\/} функции $f$ {\it около\/} точки $+\infty$; для $p\in
\RR^+$} \type^{\infty}_p[f]&:=\limsup_{r\to +\infty} \frac{f^+(r)}{r^p}\in \RR^+\cup \{+\infty\}
 \tag{\ref{se:nu0}t}\label{senu0:c}
\end{align} 
\end{subequations} 
 {\it --- тип роста\/} функции $f$ {\it при порядке\/ $p$ около\/} точки $+\infty$. Очевидно,
\begin{equation}\label{t_o} \text{\it если $\type_p^{\infty}[f]<+\infty$, то $\ord_{\infty}[f]\leq p$}. \end{equation} Обратное к \eqref{t_o}
неверно. Используем еще одну характеризацию $f$ около $+\infty$.
Функция $f$ принадлежит {\it классу сходимости\/} или {\it классу расходимости при порядке роста\/ $p\in \RR^+$ около\/} точки $+\infty$ \cite[определение 4.1]{HK}, если соответственно сходится (конечен) или расходится интеграл
\begin{equation}\label{{se:fcc}a} 
\int_{r_0}^{+\infty}\frac{\bigl|f(t)\bigr|}{t^{p+1}}\dd t.
\end{equation}. 
\begin{propos}\label{pr:fcc} 
	Пусть $f\colon [r_0,+\infty)\to \RR$ --- возрастающая функция. 
	\begin{enumerate}[{\rm (i)}] 
		\item\label{i:i0} Если $\ord_{\infty}[f]< p$, то
функция $f$ принадлежит классу сходимости при порядке роста $p$ около $+\infty$. 

\item\label{i:i1} Из сходимости интеграла \eqref{{se:fcc}a} следует $\type_p^{\infty}[f]=0$ и 
\begin{equation}\label{cc:fint}
\int_{r_0}^{+\infty}\frac{\dd f(t)}{t^{p}}<+\infty. 
\end{equation} 
\item\label{i:i2} Обратно, из \eqref{cc:fint} при $p>0$ следует
$\type_p^{\infty}[f]=0$ и сходимость \eqref{{se:fcc}a}, а при $p=0$ существует $\lim\limits_{r\to +\infty}f(r)<+\infty$ и $\type_0^{\infty}[f]<+\infty$. 
\end{enumerate} 
Кроме того, в условиях \eqref{i:i1} и/или \eqref{i:i2} имеет место  равенство 
\begin{equation}\label{in:poch} \int_{r}^{+\infty}\frac{\dd
f(t)}{t^{p}}=-\frac{f(r)}{r^p}+p\int_{r}^{+\infty}\frac{f(t)}{t^{p+1}}\dd t 
\end{equation} 
\end{propos} 
\begin{proof} Утверждение п.~\eqref{i:i0} очевидно. Если $p=0$ в п.~\eqref{i:i1}, то из сходимости интеграла в \eqref{{se:fcc}a} следует, что существует $\lim\limits_{r\to +\infty}f(r)=0$ и заключения п.~\eqref{i:i1} очевидны. Значит достаточно рассмотреть случай $p>0$. Как в
\cite[лемма 4.3]{HK} при $r\geq r_0$ имеем 
\begin{equation*} f(r) \frac1p\, r^p \leq \int_{r}^{+\infty}\frac{f(t)}{t^{p+1}}\dd
t\underset{r\to +\infty}{\longrightarrow}0, 
\end{equation*} 
что влечет за собой $\type_p^{\infty}[f]=0$. Отсюда, подобно
\cite[(4.2.5)--(4.2.6)]{HK}, интегрируя по частям, для любого $r\geq r_0$ получим \eqref{in:poch}, 
где сходимость последнего интеграла влечет за собой \eqref{cc:fint}.

Обратно, если $p>0$, то при $r_0\leq r<R$ из \eqref{cc:fint} следует 
\begin{equation*} \frac{1}{R^p}\bigl(f(R)-f(r)\bigr)\leq
\int_{r}^{R}\frac{\dd f(t)}{t^{p}}\dd t\underset{r\to +\infty}{\longrightarrow}0, 
\end{equation*} 
откуда $\type_p^{\infty}[f]=0$.
Тогда справедливо  \eqref{in:poch} и сходится \eqref{{se:fcc}a}. Если  $p=0$, то  \eqref{cc:fint} эквивалентно существованию предела $\lim\limits_{r\to +\infty}f(r)<+\infty$. 
\end{proof}

\subsection{ Исчезание около нуля} Для $f\colon (0,r_0] \to \RR$ определим 
\begin{subequations}\label{se:nu00} 
	\begin{align} 
	\ord_{0}[f]&:=\liminf_{0<r\to 0} \frac{\log \bigl|f(r)\bigr|}{\log r}\in \RR^+\cup \{+\infty\}
\tag{\ref{se:nu00}o}\label{senu00:o} 
\\ \intertext{{\it --- порядок
исчезания\/} функции $f$ {\it около\/} точки $0$, а  для $p\in \RR^+$} \type^{0}_p[f]&:=\limsup_{0<r\to 0} \frac{\bigl|f(r)\bigr|}{r^p}\in \RR^+\cup \{+\infty\}
\tag{\ref{se:nu00}t}\label{senu00:c} 
\end{align} \end{subequations} 
{\it --- тип исчезания\/} функции $f$ {\it при порядке\/ $p$ около\/} точки $0$. Очевидно, 
\begin{equation}\label{t_o0} 
\text{\it если $\type_p^{0}[f]<+\infty$, то $\ord_{0}[f]\geq p$}. 
\end{equation} 
Обратное к \eqref{t_o0} неверно. Используем ещё одну  характеризацию исчезания около $0$. Говорим, что функция $f\colon (0, r_0]\to \RR$ принадлежит {\it классу сходимости\/}
или {\it классу расходимости при порядке\/ $p\in \RR^+$ около\/} точки $0$, если соответственно сходится (конечен) или расходится интеграл 
\begin{equation}\label{{se:fcc}0}  
\int_{0}^{r_0}\frac{\bigl|f(t)\bigr|}{t^{p+1}}\dd t.
\end{equation}. 
\begin{propos}\label{pr:fcc0} 
	Пусть $f\colon (0,r_0]\to \RR$ --- возрастающая функция. 
	\begin{enumerate}[{\rm (i)}] \item\label{i:i00} Если $\ord_0[f]\overset{\eqref{senu00:o}}{>} p$, то функция $f$
принадлежит классу сходимости интеграла \eqref{{se:fcc}0} при порядке $p$ около нуля. 
\item\label{i:i10} Из сходимости интеграла \eqref{{se:fcc}0}
следует, что существует 
\begin{equation}\label{ex:limf} \lim_{0<r\to 0} f(r)\log r =0, 
\end{equation} 
а также
\begin{equation}\label{ex:limf0} \int_0^{r_0}\log t \dd f(t)>-\infty, \end{equation} а при $p>0$ еще и $\type_p^{0}[f]\overset{\eqref{senu00:c}}{=}0$, а также сходится интеграл 
\begin{equation}\label{cc:fint0}
 \int_{0}^{r_0}\frac{\dd f(t)}{t^{p}}<+\infty. 
\end{equation} 
\item \label{i:i20}
Обратно, если выполнено \eqref{ex:limf0}, то существуют пределы 
\begin{equation}\label{es:00l} f(0):=\lim_{0<r\to 0} f(r)\in \RR,
\quad \lim_{0<r\to 0}\bigl(f(r)-f(0)\bigr)\log r=0, 
\end{equation} 
сходится интеграл 
\begin{equation}\label{f-f0}
\int_{0}^{r_0}\frac{f(t)-f(0)}{t}\dd t<+\infty, 
\end{equation} 
и $\type_0^{0}[f]\overset{\eqref{senu00:c}}{<}+\infty$. Если же при $p>0$ выполнено
\eqref{cc:fint0}, то, наряду с \eqref{es:00l}, $\type_p^{0}\bigl[f-f(0)\bigr]\overset{\eqref{senu00:c}}{=}0$ и 
\begin{equation}
\int_{0}^{r_0}\frac{f(t)-f(0)}{t^{p+1}}\dd t<+\infty. 
\end{equation} 
\end{enumerate} 
В условиях из \eqref{i:i10} и/или  \eqref{i:i20} при $p>0$  имеет место равенство 
\begin{equation}\label{in:poch0}
\int_{0}^{r_0}\frac{f(t)-f(0)}{t^{p+1}}\dd t=-\frac{f(r_0)-f(0)}{pr_0^p}+\frac1p\int_{0}^{r_0}\frac{\dd f(t)}{t^{p}},
\end{equation} а при $p=0$ --- равенство
 \begin{equation}\label{in:0}
\int_{0}^{r_0}\frac{f(t)-f(0)}{t}\dd t=\bigl(f(r_0)-f(0)\bigr)\log r_0- \int_0^{r_0}\log t \dd f(t). 
\end{equation} \end{propos}
\begin{proof} Утверждение \eqref{i:i00} очевидно.

Из сходимости интеграла \eqref{{se:fcc}0} имеем  $\lim\limits_{0<r\to 0} f(t)=0$. Следовательно, $f$ --- положительная функция и при $0<r<\sqrt r\leq \min \{1,r_0\}$  имеем
\begin{equation*} 0\underset{0<r\to 0}{\longleftarrow}\int_r^{\sqrt r}\frac{f(t)}{t^{p+1}}\dd t \geq
\int_r^{\sqrt r}\frac{f(t)}{t}\dd t\geq -\frac{1}{2}\,f(r)\log r, 
\end{equation*} 
что доказывает существование предела в \eqref{ex:limf}, равного нулю. Тогда 
\begin{multline*} \int_0^{r_0}\log t \dd f(t)=f(r_0)\log r_0-\int_{0}^{r_0}\frac{f(t)}{t}\dd t
\\ \geq f(r_0)\log r_0-\int_{0}^{r_0}\frac{f(t)}{t^{p+1}}\dd t>-\infty, 
\end{multline*} 
что доказывает \eqref{ex:limf0} и \eqref{in:0} с
$f(0)=0$ . При $p>0$ имеем
\begin{equation*}
 0\underset{0<r\to 0}{\longleftarrow}\int_{r}^{2r}
\frac{f(t)}{t^{p+1}}\dd t\geq  f(r)\,\frac{1}{p} \left(\frac{1}{r^p}-\frac{1}{(2r)^p}\right)=\frac{1}{p} \left(1-\frac{1}{2^p}\right)\frac{f(r)}{r^p},
\end{equation*} 
откуда $\type_p^{0}[f]\overset{\eqref{senu00:c}}{=}0$. Тогда можно интегрировать по частям, что дает \eqref{in:poch0} с $f(0)=0$, откуда имеет место \eqref{cc:fint0}.

Обратно, пусть выполнено \eqref{ex:limf0}. Тогда существует число $C\in \RR^+$, с которым для любых $r\in (0, r_1)$ при фиксированном $r_1<\min \{1,r_0\}$ 
\begin{equation}\label{Clfr1}
 -C\leq \int_r^{r_1}\log t \dd f(t)\leq \log r_1 \int_r^{r_1}\dd f(t)=\bigl(f(r_1)-f(r)\bigr)\log r_1, 
\end{equation} 
откуда следует существование первого предела в \eqref{es:00l} и возможность
доопределить значение $f(0)\in \RR$. В частности, $\type_0^{0}[f]\overset{\eqref{senu00:c}}{<}+\infty$. Кроме того, устремляя в \eqref{Clfr1} сначала $r$ к нулю, а затем и $r_1$, получаем 
\begin{equation}\label{r-rf}
 0=\lim_{0<r_1\to 0}\int_0^{r_1}\log t \dd f(t)\leq \liminf_{0<r_1\to 0}
\,\bigl(f(r_1)-f(0)\bigr)\log r_1. 
\end{equation} 
При $0<r<\sqrt r<r_1=\min \{1,r_0\}$ из 
\begin{multline*} 0\underset{0<r\to 0}{\longleftarrow}\int_{r}^{\sqrt r} \log t\dd f(t)= \int_{r}^{\sqrt r} \log t\dd \bigl(f(t)-f(0)\bigr)
\\
\geq \log r\int_{r}^{\sqrt r} \dd \bigl(f(t)-f(0)\bigr)  \geq \log r\int_{0}^{\sqrt r} \dd \bigl(f(t)-f(0)\bigr)\\
=\bigl(f(\sqrt r)-f(0)\bigr)\log r =2\bigl(f(\sqrt r)-f(0)\bigr)\log \sqrt r 
\end{multline*} 
следует  $$ 0\geq \limsup_{0<r_1\to 0} \,\bigl(f(\sqrt r)-f(0)\bigr)\log \sqrt r,$$ 
что вместе с \eqref{r-rf} дает второе предельное равенство в \eqref{es:00l} и  можно интегрировать по частям:
\begin{equation*} \int_{0}^{r_0}\frac{f(t)-f(0)}{t}\dd t=\bigl(f(r_0)-f(0)\bigr)\log r_0- \int_0^{r_0}\log t \dd
f(t)\overset{\eqref{ex:limf0}}{<}+\infty, \end{equation*} что дает \eqref{f-f0} и равенство \eqref{in:0}. Пусть теперь $p>0$. Тогда
автоматически выполнено \eqref{ex:limf0}, а следовательно и \eqref{es:00l}. При этом из \eqref{cc:fint0} имеем 
\begin{equation*}
0\underset{0<r\to 0}{\longleftarrow} \int_{0}^{r}\frac{\dd \,\bigl(f(t)-f(0)\bigr)}{t^p} \geq \frac{f(r)-f(0)}{r^p}, \end{equation*}
что доказывает равенства $\type_p^{0}\bigl[f-f(0)\bigr]\overset{\eqref{senu00:c}}{=}0$ и 
\begin{equation*}
\int_{0}^{r_0}\frac{f(t)-f(0)}{t^{p+1}}\dd t=-\frac{f(r_0)-f(0)}{pr_0^p}+\frac1p\int_{0}^{r_0}\frac{\dd f(t)}{t^{p}} \overset{\eqref{cc:fint0}}{<}+\infty, 
\end{equation*} 
где последнее равенство совпадает с \eqref{in:poch0}. 
\end{proof}

\begin{remark}
Для функции $v\colon \CC\to \RR_{\pm \infty}$ все характеристики роста (соответственно исчезания) и понятия принадлежности классу (ра-)\-с\-х\-о\-д\-и\-м\-о\-с\-ти около $\infty$
(соответственно нуля) из \S~\ref{sss:sh} полностью определяются одноимёнными характеристиками и понятиями для функции $f(r):=\sup \bigl\{v(z)\colon z\in D(r)\bigr\}$. Для целой функции $g\in \Hol(\CC)$ эти характеристики и понятия те же, что и для $v:=\log |g|\colon \CC\to \RR_{-\infty}$.
\end{remark}

 \section{Гармоническая мера}\label{gm}
 \setcounter{equation}{0} 
\subsection{Гармоническая мера и интеграл Пуассона}\label{p_3_1}
Через $\mathcal B (S)$ обозначаем $\sigma$-алгебру борелевских подмножеств в  борелевском множестве $S\subset\CC_{\infty}$, а при $S\subset \CC$ используем обозначение $\mathcal B_{\rm b} (S)\subset \mathcal B (S)$ для подкласса ограниченных в $\CC$ элементов  из $\mathcal B (S)$. Для подобласти $D\subset \CC_{\infty}$ с неполярной границей $\partial D\neq \varnothing$ в обозначениях из
\cite[определение 4.3.1]{Rans}, несколько отличных от обозначений из \cite[теорема 3.10]{HK}--\cite{GM}, через
\begin{equation}\label{om_0} \omega_D\colon D\times \mathcal B (\partial D) \to \RR^+ \end{equation} обозначаем функцию, называемую
{\it гармонической мерой для\/ $D$}, которая однозначно определяется требованиями \begin{enumerate}[{\rm (a)}] \item\label{hm:a} $
\omega_D(z, B)\in \mathcal{M}^+(\partial D)$ для каждой точки $z\in D$ и множества $B\in \mathcal B (\partial D)$; \item\label{hm:b}
если $f\in C_{\RR}(\partial D)$, то {\it интеграл Пуассона\/} $\mathcal{P}_D f$ функции $f$ на $D$ 
\begin{equation}\label{df:PDf} 
\mathcal P_Df\colon z\mapsto \int_{\partial D} f(w)\dd \omega_D(z,w), \quad z\in D, 
\end{equation}
совпадает с ассоциированной с функцией $f$ {\it функцией Перрона\/} 
\begin{equation*} 
z\mapsto \sup \Bigl\{u(z)\colon u\in \sbh(D), \limsup_{D\ni z'\to z}u(z')\leq f(z) \text{ для всех $z\in \partial D$}\Bigr\}, \quad z\in D. \end{equation*}
\end{enumerate} Очевидно, $\omega_D(z,\cdot)$ --- {\it вероятностная мера,\/} в чем можно убедится на примере $f=1$ на $\partial D$.
Для произвольного подмножества $B\subset \CC_{\infty}$ при $B\cap \partial D\in \mathcal B (\partial D) $ и точки $z\in D$ полагаем по
определению \begin{equation}\label{df:o} \omega_D(z, B):=:\omega(z, B; D):=\omega_D(z, B\cap \partial D) \end{equation} --- {\it
гармоническая мера множества $B\subset \CC_{\infty}$ в точке $z\in D$ для\/} $D$.

При включениях $\partial D\subset S \subset \CC_{\infty}$ для непрерывной функции $f\colon S\to \RR_{\pm \infty}$ с
непрерывным сужением $f\bigm|_{\partial D}$ {\it интеграл Пуассона\/} $\mathcal P_Df$ определяется как 
\begin{equation}\label{df:IP} \mathcal P_Df:=\mathcal P_D\bigl(f\bigm|_{\partial D}\bigr). \end{equation} Определение \eqref{df:IP},
допуская в \eqref{df:IP} и значения $-\infty$, можно распространить на полунепрерывные сверху сужения $f\bigm|_{\partial D}$ со значениями в $\RR_{-\infty}$, поскольку такая функция --- предел убывающей последовательности непрерывных функций на $\partial D$.  Более того, далее в многих ситуациях можно допускать неопределённость функции $f$ в конечном числе точек на границе $\partial D$, в роли которых будут выступать, как правило, две точки $0, \infty \in \CC_{\infty}$.  В частном случае, когда
$D=\CC^{\up}$, определения \eqref{om_0}--\eqref{df:IP} можно переписать через классические {\it ядро Пуассона\/}
\begin{equation}\label{df:kP} \Poi_{\CC^{\up}}(t,z):=\frac{1}{\pi}\Im \frac{1}{t-z} =\frac{1}{\pi}\,\frac{\Im z}{\bigl(t-\Re
z\bigr)^2+(\Im z)^2}\, , \quad (t,z)\in \RR\times \CC^{\up}, 
\end{equation}
где нижний индекс $\CC^{\up}$ в обозначении ядра
Пуассона в левой части далее опускаем, а также, для всех $z\in \CC^{\up}$, через {\it интеграл Пуассона\/} 
\begin{subequations}\label{se:om} \begin{align}
\omega_{\CC^{\up}}(z, B)&= \int_{B\cap \RR}\Poi (t,z) \dd t =\int_{B\cap \RR_{\pm\infty}} \Poi (t,z) \dd
\lambda_{\RR}(t), \tag{\ref{se:om}a}\label{seom:a} \\ (\mathcal P_{\CC^{\up}}f)(z)&=\int_{-\infty}^{+\infty} f(t) \Poi (t,z)\dd t
=\int_{\RR_{\pm\infty}}\Poi (t,z)\,f(t)\dd \lambda_{\RR}(t). 
 \tag{\ref{se:om}b}\label{seom:b} 
 \end{align}
\end{subequations} 
В случае $D=\CC^{\up}$ в обозначениях левых частей из  \eqref{se:om} нижний индекс ${\CC^{\up}}$, 
также, как правило, опускаем, т.\,е. пишем просто $\omega(z, B)$ и $(\mathcal P f)(z)$. Определение \eqref{seom:a} для $B\cap \RR_{\pm\infty}\in \mathcal B(\CC_{})$ можно распространить на $z=x\in {\RR}$, полагая $\omega(x,B)=\delta_x(B)$ ---
мера Дирака в точке $x$. Для функций $f$ с полунепрерывным сверху сужением $f\bigm|_{\RR}$ и со значениями в $\RR_{-\infty}$ при
\begin{equation}\label{J2+} J_2^+[f]:=\int_{-\infty}^{+\infty} \frac{f^+(t)}{1+t^2}\dd t <+\infty \end{equation} будем
полагать $(\mathcal P f)(x)=f(x)\in \RR_{-\infty}$ для $x\in \RR$. 

Гармоническая мера для верхней полуплоскости имеет известный простой геометрический смысл. Для каждого отрезка $[t_1,t_2]\subset \RR$
--- это угол, под которым виден этот отрезок из точки $z\in \CC^{\up}$, делённый на $\pi$ \cite[гл.~I, 3]{Garnett}, что согласованно с пп. \ref{b_seq}. Аналитически это записывается как\begin{subequations}\label{multo} 
	\begin{align} \omega\bigl(z, [t_1,t_2]\bigr)= \frac{1}{\pi}(\arctg d_2 -\arctg d_1),
\tag{\ref{multo}a}\label{mult}\\ d_2:= \frac{t_2-\Re z}{\Im z}, \quad d_1:=\frac{t_1-\Re z}{\Im z}\,. \tag{\ref{multo}b}\label{d1d2}
\end{align} 
\end{subequations}

\subsection{Представления и оценки гармонической меры}\label{s:dal} Приведенные ниже утверждения данного п.~\ref{p_3_1} и следующего
п.~\ref{s:dal} часто формулируются в более общей форме, чем необходимо для наших целей. Причины этого состоят как в исключительной важности и пользе понятия гармонической меры и её количественных характеристик в теориях функций комплексного переменного и потенциала, так и в заготовке материалов для исследований в направлениях, отличных от затронутых в настоящей работе. Всюду в
доказательствах мы предпочитаем использовать геометрический смысл \eqref{multo} гармонической меры, хотя все утверждения этого п.~\ref{s:dal} могут быть установлены и через её представление 
\eqref{seom:a}  интегралом Пуассона. Именно последний подход был использован в \cite{Kha89}--\cite[гл.I, II]{KhDD92}. 

\begin{propos}\label{pr:gm} Пусть $-\infty<t_1<t_2<+\infty$. Если точка $z\in \CC^{\up}$ лежит вне замкнутого полукруга из верхней полуплоскости с ди\-а\-м\-е\-т\-р\-ом-от\-р\-е\-з\-к\-ом $[t_1,t_2]\subset \RR\subset \CC$, что аналитически записывается как
\begin{subequations}\label{krt} 
	\begin{align} 
	\Bigl|z-\frac{t_2+t_1}{2}\Bigr|> \frac{t_2-t_1}{2}\,, \tag{\ref{krt}a}\label{zt1t2}\\
\intertext{или в эквивалентной форме} |z|^2-\Re z\, (t_1+t_2)+t_1t_2>0, \tag{\ref{krt}b}\label{kr} 
\end{align} 
\end{subequations}
 то имеют место равенство и оценка сверху 
 \begin{subequations}\label{eo} 
 	\begin{align}
\omega\bigl(z,[t_1,t_2]\bigr)&=\frac{1}{\pi}\arctg\frac{(t_2-t_1)\,\Im z}{|z|^2-\Re z \, (t_1+t_2)+t_1t_2}&
\tag{\ref{eo}a}\label{eo:a} \\ &\leq \frac{1}{\pi}\,\frac{(t_2-t_1)\,\Im z}{|z|^2-\Re z \, (t_1+t_2)+t_1t_2}\, .
\tag{\ref{eo}b}\label{eo:b} 
\end{align} 
\end{subequations}

Если точка $z\in \CC^{\up}$ лежит в открытом полукруге из $\CC^{\up}$ с ди\-а\-м\-е\-т\-р\-ом-от\-р\-е\-з\-к\-ом $[t_1,t_2]$, что в
аналитической форме записывается как \begin{subequations}\label{krtvn} \begin{align} \Bigl|z-\frac{t_1+t_2}{2}\Bigr|<
\frac{t_2-t_1}{2}\,, \tag{\ref{krtvn}a}\label{zt1t2vn}\\ \intertext{или в эквивалентной форме} |z|^2-\Re z\, (t_1+t_2)+t_1t_2<0,
\tag{\ref{krtvn}b}\label{krvn} \end{align} \end{subequations} то имеют место равенство и оценка снизу \begin{subequations}\label{eovn}
\begin{align} \omega\bigl(z,[t_1,t_2]\bigr)&=1+\frac{1}{\pi}\arctg\frac{(t_2-t_1)\,\Im z}{|z|^2-\Re z \, (t_1+t_2)+t_1t_2}&
\tag{\ref{eovn}a}\label{eovn:a} \\ &\geq 1+\frac{1}{\pi}\,\frac{(t_2-t_1)\,\Im z}{|z|^2-\Re z \, (t_1+t_2)+t_1t_2}\, .
\tag{\ref{eovn}b}\label{eovn:b} \end{align} \end{subequations}

Наконец, если точка $z\in \CC^{\up}$ лежит на полуокружности из $\CC^{\up}$ с ди\-а\-м\-е\-т\-р\-ом-от\-р\-е\-з\-к\-ом $[t_1,t_2]$,
что в аналитической форме записывается как 
\begin{subequations}\label{krtvn0} 
\begin{align} \Bigl|z-\frac{t_1+t_2}{2}\Bigr|=
\frac{t_2-t_1}{2}\,, \tag{\ref{krtvn0}a}\label{zt1t2vn0}\\ \intertext{или в эквивалентной форме} |z|^2-\Re z\, (t_1+t_2)+t_1t_2=0,
\tag{\ref{krtvn0}b}\label{krvn0} 
\end{align} 
\end{subequations} то $\omega\bigl(z,[t_1,t_2]\bigr)=1/2$.
\end{propos}
\begin{proof} В случае \eqref{krt} в обозначениях \eqref{d1d2}
имеем \begin{multline*} d_2d_1=\frac{t_2t_1-\Re z\,(t_2+t_1)+(\Re z)^2}{(\Im z)^2} \\= \frac{t_2t_1-\Re z\,(t_2+t_1)+ |z|^2}{(\Im
z)^2}-1\overset{\eqref{kr}}{>}-1, \end{multline*} что позволяет применить формулу разности арктангенсов в \eqref{mult} в форме
\begin{equation*} \arctg d_2 -\arctg d_1=\arctg\frac{d_2-d_1}{1+d_1d_2}\overset{\eqref{d1d2}}{=} \arctg\frac{(t_2-t_1)\,\Im
z}{|z|^2-\Re z \, (t_1+t_2)+t_1t_2}\,. \end{equation*} Правая часть здесь дает \eqref{eo:a}, а затем и \eqref{eo:b} ввиду возрастания
$\arctg$.

В случае \eqref{krtvn} в обозначениях \eqref{d1d2} имеем $d_2d_1<-1$ и $d_2\overset{\eqref{d1d2}}{>}0$, что позволяет применить
формулу разности арктангенсов в \eqref{mult} в форме \begin{multline*} \arctg d_2 -\arctg
d_1=\pi+\arctg\frac{d_2-d_1}{1+d_1d_2}\\
\overset{\eqref{d1d2}}{=} \pi+\arctg\frac{(t_2-t_1)\,\Im z}{|z|^2-\Re z \, (t_1+t_2)+t_1t_2}\,.
\end{multline*} Правая часть здесь дает \eqref{eovn:a}, а затем и \eqref{eovn:b} ввиду возрастания функции $\arctg$ и отрицательности
ее значений в правой части \eqref{eovn:a}. Случай \eqref{krtvn0} по геометрическому смыслу гармонической меры очевиден. \end{proof}
\begin{propos}\label{es:qmlw} При $b>1$ для точки $z\in \CC^{\up}$, лежащей вне открытого полукруга из верхней полуплоскости \begin{equation}\label{zt1t2b} 
\Bigl|z-\frac{t_1+t_2}{2}\Bigr|\geq b\,\frac{t_2-t_1}{2}\,, \quad z\in \CC^{\up},\quad \text{где \;
$b>1$ фиксировано}, \end{equation} справедлива нижняя оценка \begin{subequations}\label{z2lw} \begin{align}
\omega\bigl(z,[t_1,t_2]\bigr)&\geq \frac{b-1}{2\pi b}\, \frac{(t_2-t_1)\,\Im z}{|z|^2-\Re z \, (t_1+t_2)+t_1t_2}
\tag{\ref{z2lw}a}\label{z2lwb}\\ &\geq \frac{b-1}{2\pi b}\, \frac{(t_2-t_1)\,\Im z}{\bigl(|z|+t_1\bigr)\bigl(|z|+t_2\bigr)}\,.
\tag{\ref{z2lw}b}\label{z2lwbb} \end{align} \end{subequations} \end{propos} \begin{proof} Версия равенства \eqref{eo:a} предложения
\ref{pr:gm} --- это \begin{lemma}\label{l:oup} Пусть $r>0$ и $z=te^{i\theta}$ с $t>r$, $0\leq \theta \leq \pi$. Тогда
\begin{equation}\label{eo_rz} \omega\bigl(z,[-r,r]\bigr)=\frac{1}{\pi}\arctg\frac{2rt\sin \theta}{t^2-r^2}, \end{equation} где
макcимум при фиксированных $t>r>0$ достигается при $\sin \theta =1$, а при $t=br$, где $b>1$, равен \begin{equation}\label{m:od}
\omega\bigl(ibr,[-r,r]\bigr)=\frac{1}{\pi}\arctg\frac{2b}{b^2-1}\, . \end{equation} \end{lemma} Заключение леммы \ref{l:oup} о максимуме  \eqref{m:od} для отрезка $[t_1,t_2]$ означает, что при всех $z\in \CC^{\up}$, удовлетворяющих условию \eqref{zt1t2b}, в обозначении 
\begin{equation}\label{xfr} 
x:=\frac{(t_2-t_1)\,\Im z}{|z|^2-\Re z \, (t_1+t_2)+t_1t_2} 
\end{equation}
 имеем $0\leq x\leq \dfrac{2b}{b^2-1}$ согласно  геометрическому смыслу гармонической меры и возрастания функции $\arctg$. Тогда в силу вогнутости  $\arctg$ на $\RR^+$ 
\begin{equation}\label{arcx} 
\arctg x\geq
\frac{\arctg\frac{2b}{b^2-1}}{\frac{2b}{b^2-1}} \,x 
\quad \text{при} \quad 0\leq x\leq
\dfrac{2b}{b^2-1}\,, 
\end{equation} 
откуда согласно неравенствам $ \frac{1}{b-1}\leq\frac{2b}{b^2-1}\leq \frac{2}{b-1}$ при
$b>1$ следует 
\begin{equation*} \arctg x\overset{\eqref{arcx}}{\geq}   \frac{1}{2}\Bigl((b-1)\arctg\frac{1}{b-1}\Bigr)\cdot x\geq \frac12\frac{b-1}{b}\, x \quad \text{при} \quad 0\leq x\leq
\dfrac{2b}{b^2-1}\,, \end{equation*} где последнее неравенство основано на неравенстве $\arctg \frac{1}{y}\geq \frac{1}{1+y}$ при  всех $y=b-1> 0$. Отсюда ввиду \eqref{xfr} получаем  \eqref{z2lwb}. Для перехода от \eqref{z2lwb}  к \eqref{z2lwbb} оценим сверху последний знаменатель в \eqref{z2lwb}: \begin{equation*}
(t_1+t_2)+t_1t_2\leq |z|^2+|z|\bigl(|t_1|+|t_2|\bigr)+|t_1||t_2| =\bigl(|z|+t_1\bigr)\bigl(|z|+t_2\bigr),
\end{equation*} что и завершает доказательство. 
\end{proof}

Оценки сверху из  \cite[лемма 1.1.1]{KhDD92} и \cite[лемма 1.1]{Kha91}  уточняет
\begin{propos}\label{pr:gmA} Пусть $-\infty<t_1<t_2<+\infty$, $z\in \CC^{\up}$, $a\in (0,1)$. Из \begin{equation}\label{minM} a|z|\geq
\max\bigl\{|t_1|,|t_2|\bigr\} \end{equation} следует оценка \begin{equation}\label{afminAM} \omega\bigl(z,[t_1,t_2]\bigr)\leq
\frac{t_2-t_1}{\pi (1-a)^2}\, \Im \frac{1}{\bar z}\,, \end{equation} а при условии \begin{equation}\label{min} a\cdot
\min\limits_{t\in [t_1, t_2]} |t|\geq |z| \end{equation} --- оценка \begin{equation}\label{afmin} \omega\bigl(z,[t_1,t_2]\bigr)\leq
\frac{(t_2-t_1)\,a^2}{\pi (1-a)^2}\, \Im \frac{1}{\bar z}\,, \end{equation} \end{propos} \begin{proof} Левую часть \eqref{kr} можно
оценить снизу: \begin{equation}\label{zsn} |z|^2-\Re z\, (t_1+t_2)+t_1t_2\geq |z|^2-|z|\,
(t_1+t_2)+t_1t_2=\bigl(t_1-|z|\bigr)\,\bigl(t_2-|z|\bigr). \end{equation} При условии \eqref{minM} оба сомножителя в правой части
\eqref{zsn} отрицательны, т.\,е. выполнено условие \eqref{kr}. Следовательно, \begin{equation*}
\omega\bigl(z,[t_1,t_2]\bigr)\overset{\eqref{zsn}, \eqref{eo:b}}{\leq} \frac{1}{\pi}\,\frac{(t_2-t_1)\,\Im
z}{\bigl(|z|-t_1\bigr)\,\bigl(|z|-t_2\bigr)} \overset{\eqref{minM}}{\leq} \frac{1}{\pi}\,\frac{(t_2-t_1)\,\Im
z}{\bigl(|z|-a|z|\bigr)\,\bigl(|z|-a|z|\bigr)}\,, \end{equation*} где правая часть совпадает с правой частью в \eqref{afminAM}.
Условие \eqref{min}, в частности, означает, что точки $t_1,t_2\in \RR$ расположены по одну сторону от нуля и пара сомножителей в
правой части \eqref{zsn} одинакового знака, т.\,е. выполнено условие \eqref{kr}. Следовательно, \begin{equation*}
\omega\bigl(z,[t_1,t_2]\bigr)\overset{\eqref{zsn}}{\leq} \frac{1}{\pi}\,\frac{(t_2-t_1)\,\Im
z}{\bigl(t_1-|z|\bigr)\,\bigl(t_2-|z|\bigr)} \overset{\eqref{min}}{\leq} \frac{1}{\pi}\,\frac{(t_2-t_1)\,\Im
z}{\bigl(|z|/a-|z|\bigr)\,\bigl(|z|/a-|z|\bigr)}\,, \end{equation*} где правая часть совпадает с правой частью в \eqref{afmin}.
\end{proof} 
Оценки снизу из \cite[лемма 1.1.1]{KhDD92} и \cite[лемма 1.1(1), (1.2)]{Kha91}   уточняет 
\begin{propos}\label{pr:gmA+} В обозначениях предложения\/ {\rm \ref{pr:gmA}} при условии\/ \eqref{minM} 
\begin{equation}\label{afminAM-} 
\omega\bigl(z,[t_1,t_2]\bigr)\geq (t_2-t_1)
\frac{1-a}{8\pi}\, \Im \frac{1}{\bar z}\,. 
\end{equation} 
\end{propos}

\begin{proof} При условии \eqref{minM} элементарные геометрические рассуждения показывают, что выполнено \eqref{zt1t2b} с $b:=1/a$.
Кроме того, при условии \eqref{minM}, оставаясь в рамках \eqref{kr}, имеем оценку сверху 
$0 <  \bigl(|z|+t_1\bigr)\bigl(|z|+t_2\bigr) \overset{\eqref{minM}}{\leq} |z|^2(1+a)^2$. 
Отсюда по предложению \ref{es:qmlw} из оценки \eqref{z2lwbb} с $b=1/a$ получаем 
\begin{equation*} \omega\bigl(z,[t_1,t_2]\bigr)\geq
\frac{1/a-1}{2\pi/a}\,\frac{(t_2-t_1)\,\Im z}{|z|^2(1+a)^2} =\frac{1-a}{2\pi(1+a)^2}\,(t_2-t_1)\,\Im \frac{1}{\bar z}\, ,
\end{equation*}
что в силу $0< a<1$ дает нижнюю оценку \eqref{afminAM-}. 
\end{proof}

Для точек $z\in \CC^{\up}$ их аргументы $\arg z$ берутся из интервала $(0,\pi)$. \begin{propos}\label{pr:gmAa} Пусть $0\leq
t_1<t_2<+\infty$, $z\in \CC^{\up}$. Из \begin{equation}\label{mina} \pi <\arg z\leq \frac{\pi}{2}, \quad\text{т.\,е. $\cos \arg z\leq
0$,} \end{equation} следует оценка сверху \begin{equation}\label{afmina} \omega\bigl(z,[t_1,t_2]\bigr)\leq (t_2-t_1)\,
\frac{1}{\pi}\Im \frac{1}{\bar z}\,. \end{equation} Из условия \begin{equation}\label{minMa} -1< \cos \arg
z<\frac{2\sqrt{t_1t_2}}{t_1+t_2} \end{equation} следует оценка сверху \begin{equation}\label{afminAMa}
\omega\bigl(z,[t_1,t_2]\bigr)\leq\frac{( t_2-t_1)\, \Im z}{\pi \bigl(|z|-\sqrt{t_1t_2}\,\bigr)^2}\,. \end{equation} Пусть для
некоторого $a\in (0,1)\subset\RR^+$ выполнено условие\begin{equation}\label{minMaa} -1< \cos \arg z\leq
\frac{2a\sqrt{t_1t_2}}{t_1+t_2} \,. \end{equation} Тогда имеет место оценка сверху \begin{equation}\label{afminAMaa}
\omega\bigl(z,[t_1,t_2]\bigr)\leq\frac{ t_2-t_1}{\pi \bigl(1-a^2\bigr)}\, \Im \frac{1}{\bar z}\,. \end{equation} \end{propos}
\begin{proof} При условии \eqref{mina} для $[t_1,t_2]\subset \RR^+$ имеем 
	$$|z|^2-\Re z\, (t_1+t_2)+t_1t_2\geq |z|^2>0,$$  т.\,е. выполнено \eqref{kr} и по предложению \ref{pr:gm} из \eqref{eo:b} сразу следует \eqref{afmina}. При условии \eqref{minMa} выполнены соотношения
	\begin{equation}\label{sqrt} |z|^2-\Re z\, (t_1+t_2)+t_1t_2>
|z|^2-2|z|\sqrt{t_1t_2}+t_1t_2=\bigl(|z|-\sqrt{t_1t_2}\bigr)^2\geq 0, 
\end{equation} 
т.\,е. выполнено \eqref{kr} и по предложению \ref{pr:gm} из \eqref{eo:b} сразу следует \eqref{afminAMa}. Вариация оценки снизу \eqref{sqrt} при чуть более жёстком условии
\eqref{minMaa} дает \begin{multline}\label{zsnizu} |z|^2-\Re z\, (t_1+t_2)+t_1t_2\geq |z|^2-2a|z|\sqrt{t_1t_2}+t_1t_2 \\
=(1-a^2)|z|^2+\bigl(a|z|-\sqrt{t_1t_2}\bigr)^2\geq (1-a^2)|z|^2>0, \end{multline} откуда по предложению \ref{pr:gm} из \eqref{eo:b}
получаем \eqref{afminAMaa}. 
\end{proof}
Оценку снизу \eqref{afminAM-}  предложения \ref{pr:gmA+} дополняет 
\begin{propos}\label{r:+} Пусть $0<t_1<t_2 <+\infty$ и для $z\in \CC^{\up}$ выполнено условие \eqref{minMa}. Тогда справедлива оценка снизу 
\begin{equation}\label{essnizo} 
\omega\bigl(z,[t_1,t_2]\bigr)\geq \frac{1}{8\pi}\frac{t_1}{t_2}\, (t_2-t_1)\,\Im \frac{1}{\bar z} 
\end{equation} 
и, очевидно, если ещё и $t_2\leq At_1$ для некоторого числа $A>1$, то 
\begin{equation*}
\omega\bigl(z,[t_1,t_2]\bigr)\geq \frac{1}{8\pi A}\, (t_2-t_1)\,\Im \frac{1}{\bar z}\, . 
\end{equation*} 
\end{propos} 
\begin{proof}
Предположим сначала, что для точки $z$ выполнено ус\-л\-о\-в\-ие \eqref{minMaa} при некотором выбранном $a\in (0,1)$. Из элементарных геометрических построений при $0<t_1<t_2$ и
условии \eqref{minMaa} нетрудно вывести, что точка $z$ лежит вне полукруга из предложения \ref{es:qmlw}, а точнее выполнено условие
\eqref{zt1t2b} предложения \ref{es:qmlw} с числом 
\begin{equation}\label{dfb} 
b:=\sqrt{1+\frac{4(1-a^2)t_1t_2}{(t_2-t_1)^2}}>1 .
\end{equation} 
Тогда имеют место оценки снизу \eqref{z2lw} с выбранным в \eqref{dfb} значением $b>1$. В частности, из \eqref{z2lwb} в
сочетании с оценкой \eqref{zsnizu}, выполненной при условии \eqref{minMaa}, имеем 
\begin{equation}\label{z2lwb+-} \omega\bigl(z,[t_1,t_2]\bigr)\geq \frac{b-1}{2\pi b}\, \frac{(t_2-t_1)\,\Im z}{(1-a^2)|z|^2}\geq
\Bigl(1-\frac{1}{b}\Bigr)\frac{1}{4\pi (1-a)}\, (t_2-t_1)\,\Im \frac{1}{\bar z} \, . 
\end{equation} 
Оценим снизу первый сомножитель в
правой части \eqref{z2lwb+-}: \begin{equation}\label{1fr1b} 
1-\frac{1}{b}= 1-\frac{1}{\sqrt{1+\frac{4(1-a^2)t_1t_2}{(t_2-t_1)^2}}}\geq
1-\frac{1}{\sqrt{1+4(1-a^2){t_1}/{t_2}}} 
\end{equation} 
При $x\in [0,4]$ имеет место  неравенство  $1-\frac{1}{\sqrt{1+x}}\geq \frac{1}{8}\, x$, 
согласно которому  можем продолжить оценку снизу для правой части \eqref{1fr1b} в виде \begin{equation*}
1-\frac{1}{b}\geq \frac{1}{8} \, 4(1-a^2)\,\frac{t_1}{t_2} \geq \frac{1-a}{2} \frac{t_1}{t_2}\, . 
\end{equation*} 
Используя эту оценку в правой части \eqref{z2lwb+-}, получаем требуемую оценку \eqref{essnizo} для произвольного $a\in (0,1)$. Поскольку правая часть в \eqref{essnizo} не зависит от $a$, можем в условии \eqref{minMaa} по непрерывности устремить $a$ к единице. Таким образом, оценка \eqref{essnizo}
выполнена и при условии \eqref{minMa}. 
\end{proof}

\begin{remark}\label{r:t1t2+-} Условие $0\leq t_1<t_2<+\infty$ предложения \ref{pr:gmAa} легко адаптируется на случай
$-\infty<t_1<t_2\leq 0$. Таким образом, при $a\in (0,1)$ как в случае $0\leq t_1<t_2<+\infty$, так и в случае $-\infty<t_1<t_2\leq 0$
в той или иной мере неохваченным оценками предложений \ref{pr:gmA} и \ref{pr:gmAa} осталось множество \begin{equation}\label{xset}
\left\{z\in {\CC}^{\overline \up}\colon \; at_1 \leq |z|<\frac{1}{a}\,t_2,\; \cos \arg z\geq \frac{2a\sqrt{t_1t_2}}{t_1+t_2}\,
\right\}, \end{equation} а для $-\infty<t_1\leq 0\leq t_2<+\infty$, $t_1\neq t_2$, с тем же $a\in (0,1)$ --- круг
\begin{equation}\label{xsetd} D\bigl(0,\max\{|t_1|, |t_2|\}/a\bigr). \end{equation} \end{remark}

Из элементарной геометрии, к примеру, из теоремы косинусов, следует 
\begin{lemma}\label{prxset} В описанных в замечании\/ {\rm \ref{r:t1t2+-}} случаях множество \eqref{xset} и круг \eqref{xsetd} содержатся в круге 
	\begin{equation}\label{DsD} D\bigl((t_1+t_2)/2,
(a+1/a)\max\{|t_1|,|t_2|\}\bigr)\supset D\bigl(\max\{|t_1|,|t_2|\}\bigr). 
\end{equation} 
\end{lemma}

Следующее утверждение заполняет отмеченный в замечании \ref{r:t1t2+-} пробел. 
\begin{propos}\label{pr:eshmb} Пусть $0\leq t_1<t_2<+\infty$, $r:=(t_2-t_1)/2$ --- радиус отрезка $[t_1,t_2]\subset \RR^+$ с центром $x_0:=(t_1+t_2)/2$. Тогда для всех $z\in {\CC}^{\overline \up}$ из замкнутой верхней
полуплоскости справедливы оценки сверху 
\begin{equation}\label{es:upom} \hspace{-2mm} \omega\bigl(z,[t_1,t_2]\bigr)\leq \begin{cases}
1 \quad\text{при всех $z\in {D}(x_0,r)\cup\{t_1\}\cup \{t_2\} $},\\ 
{1}/{2} \quad\text{при всех $z\in \CC^{\up} \cap \partial {D}(x_0,r)$},\\ 
\dfrac{1}{\pi}\, \arctg \dfrac{2r|z-x_0|}{|z-x_0|^2-r^2} \quad
\text{при всех $z\notin \overline{D}(x_0,r)$,} 
\end{cases} 
\end{equation} 
которые для каждого $t\in \RR^+$ точны на каждой верхней полуокружности 
\begin{equation*}
\partial D^{\up}(x_0,t):=\bigl\{z\in
{\CC}^{\overline \up} \colon |z-x_0|=t\bigr\}, \quad D^{\up}(x_0,t):={\CC}^{\up}\cap D(x_0,t). 
\end{equation*} 
Кроме того, при всех $z \in  D^{\up}(x_0,r) $ имеют место оценки снизу \begin{equation}\label{essn_o}
\omega\bigl(z,[t_1,t_2]\bigr)\geq 1-\dfrac{1}{\pi}\, \arctg \dfrac{2r|z-x_0|}{r^2-|z-x_0|^2} \geq 1-\dfrac{1}{\pi}\,
\dfrac{2r|z-x_0|}{r^2-|z-x_0|^2} \,. \end{equation} \end{propos} \begin{proof} Верхняя оценка через $1$ в \eqref{es:upom} очевидна. Это значение при $z\in \partial D(x_0,t)^{\up}$ с $t<r$ достигается на верхней
полуокружности $\partial D^{\up}(x_0,t)$ при $z=x_0\pm t=x_0\pm|z-x_0|\in [t_1,t_2]$. При этом на 
$\CC^{\up} \cap \partial {D}(x_0,r)$ из геометрического смысла гармонической меры левая часть \eqref{es:upom} тождественно равна $1/2$. При
$z\notin \clos {D}^{\up}(x_0,t) $ с помощью сдвига $z\mapsto z-x_0$ удобно перейти к отрезку 
$ [-r,r]$, $x_0=0$, $|z-x_0|=t$  и воспользоваться леммой \ref{l:oup} с $\sin \theta =1$, когда максимум достигается. Это доказывает как последние оценки из \eqref{es:upom}, так и ее точность. Для доказательства первой оценки  снизу  \eqref{essn_o} после сдвига $z\mapsto z-x_0$, $z\in \CC^{\overline \up}$ достаточна версия равенства \eqref{eovn:a} предложения \ref{pr:gm} (ср. с леммой \ref{l:oup}) ---
\begin{lemma}\label{l:oups} Пусть\/ $r>0$ и\/ $z=te^{i\theta}$ с\/ $0<t<r$, $0\leq \theta \leq \pi$. Тогда
\begin{equation*}
\omega\bigl(z,[-r,r]\bigr)=1+\frac{1}{\pi}\arctg\frac{2rt\sin \theta}{t^2-r^2}
=1-\frac{1}{\pi}\arctg\frac{2rt\sin \theta}{r^2-t^2}, 
\end{equation*} 
где минимум при фиксированных $r>t>0$ достигается при $\sin
\theta =1$. 
\end{lemma} 
Второе неравенство в \eqref{essn_o} очевидно. 
\end{proof}

\setcounter{equation}{0} \section{Классическое выметание заряда}\label{kvm}  
Все характеристики роста (соответственно исчезания) и понятия принадлежности классу (ра-)\-с\-х\-о\-д\-и\-м\-о\-с\-ти около $\infty$
(соответственно нуля) из \S~\ref{sss:sh} для заряда $\nu\in \mathcal M(\CC)$ полностью определяются одноимёнными характеристиками и понятиями для возрастающей радиальной считающей функции $|\nu|^{\rad}$ из \eqref{df:nup}, суженной  на какой-либо луч $[r_0,+\infty)$ (соответственно интервал $(0,r_0]$), где $|\nu|$ --- полная  вариация $\nu$,  $0<r_0\in \RR^+$. 
Если $0<\type_p^{\infty}[\nu]<+\infty$ (соответственно $<+\infty$,
$=+\infty$ или $=0$), то заряд $\nu$ часто называют зарядом {\it нормальной/средней} (соответственно {\it конечной,\/} {\it
бесконечной/максимальной} или {\it нулевой/минимальной}) {\it верхней плотности при порядке\/} $p$ \cite{Levin56}, \cite{Levin96}. 

\subsection{Условие Бляшке и выметание из  полуплоскости}\label{sss:bal} 
 Характеристики роста  заряда $\nu\in \mathcal M(\CC)$ дополняет

\begin{definition}\label{c:akh} Заряд $\nu\in \mathcal M(\CC)$ удовлетворяет {\it условию Бляшке около бесконечности в верхней
полуплоскости\/} $\CC^{\up}$ (см. и ср. с \cite[гл.~II, 2]{Garnett} и с определением функций класса A из \cite[гл.~V]{Levin56}), если
\begin{equation}\label{ineq:akh} \int\limits_{\CC^{\up}\setminus D(r_0)} \Im \frac{1}{\bar z} \dd |\nu|(z)<+\infty \quad\text{для
некоторого $r_0\in \RR^+_*$.} 
\end{equation} 
\end{definition}
Из предложений \ref{pr:fcc}\eqref{i:i0}-\eqref{i:i1}
и \ref{pr:fcc0}\eqref{i:i10} сразу следует  очевидное
\begin{propos}\label{pr:2st} Если сужение $\nu\bigm|_{\CC^{\up}}$ заряда $\nu\in
\mathcal M(\CC)$ принадлежит классу сходимости при порядке $1$ около $\infty$ 
или $\ord_{\infty}[\nu]<1$, то заряд $\nu$
удовлетворяет условию Бляшке в верхней полуплоскости около $\infty$.
\end{propos}

\begin{definition}\label{df:clbal} Пусть $\nu\in \mathcal M(\CC)$. Заряд $\nu^{\bal}\in \mathcal M(\CC)$, определяемый на каждом
борелевском подмножестве $B\in \mathcal B(\CC)$ равенством \begin{equation}\label{df:bal0up} \nu^{\bal}(B):=\int_{\CC^{\up}} \omega(z,
B\cap \CC^{\up})\dd \nu(z)+ \nu\bigl(B\cap {{\CC}_{\overline \lw}}\bigr), \end{equation} называем {\it выметанием заряда\/ $\nu$ из
верхней полуплоскости\/} $\CC^{\up}$ на замкнутую нижнюю полуплоскость ${\CC}_{\overline \lw}=\CC\setminus \CC^{\up}$.
\end{definition} 
По  определению \ref{df:clbal} выметание $\nu^{\bal}$ {\it положительной меры\/} $\nu\in \mathcal M^+(\CC)$ при его существовании, т.\,е.  при условии конечности на  $B\in \mathcal B_{\rm b} (S)$, также определяет  {\it положительную меру\/} из $\mc M^+(\CC)$ с носителем $\supp \nu^{\bal}\subset  {\CC}_{\overline \lw}$.
\begin{theorem}\label{th:cup} Пусть заряд $\nu\in \mathcal
M(\CC)$ удовлетворяет условию Бляшке около бесконечности в $\CC^{\up}$, $\nu^{\overline{\up}}:=\nu\bigm|_{{\CC}^{\overline \up}}$.
Тогда определено выметание $\nu^{\bal}$ из $\CC^{\up}$ c функцией распределения $\bigl|\nu^{\bal}\bigr|^{\RR}\colon \RR\to \RR$ из
\eqref{nuR} и для любых 
\begin{equation}\label{0rt0} -\infty<t_1<t_2<+\infty,\quad x_0:=\frac{t_1+t_2}{2}\, , \;
r:=\frac{t_2-t_1}{2}\,, \quad a\in (0,1)\subset \RR^{+} 
\end{equation} при $t_1t_2\geq 0$ имеет место оценка
\begin{multline}\label{nu:estTos} 0\leq \bigl|\nu^{\bal}\bigr|^{\RR}(t_2)-\bigl|\nu^{\bal}\bigr|^{\RR}(t_1)\leq
|\nu^{\overline{\up}}|(x_0,r) +\frac{2r}{a|x_0|} \bigl|\nu^{\overline{\up}}\bigr|^{\rad}\Bigl(\,\frac{3}{a}|x_0|\Bigr) \\
+\frac{r}{(1-a)^2}\int\limits_{|z|\geq |x_0|}\Bigl|\Im \frac{1}{z}\Bigr|\dd |\nu^{\overline{\up}}|(z)
+r\biggl(\int_r^{a|x_0|}\frac{|\nu^{\overline{\up}}|(x_0,t)}{t^2}\dd t\biggr)^+. \end{multline} 
\end{theorem} 
\begin{proof} Используем обозначение и соотношения 
\begin{equation}\label{{0rt0}c} 0<|x_0|\leq T:=\max\bigl\{|t_1|, |t_2|\bigr\} \leq 2|x_0|, \quad t_1t_2\geq 0. 
\end{equation} 
Не умаляя общности, можно считать, что $\supp \nu \subset {\CC}^{\overline
\up}$, т.\,е. $\nu^{\overline{\up}}{=}\nu$. Докажем сначала в обозначениях \eqref{0rt0} и \eqref{{0rt0}c} оценку
\begin{multline}\label{nu:estT} \bigl|\nu^{\bal}\bigr|^{\RR}(t_2)-\bigl|\nu^{\bal}\bigr|^{\RR}(t_1)\leq |\nu|(x_0,r) +\frac{2r}{\pi
T}|\nu|\bigl(x_0,(a+1/a)T\bigr) \\ +\frac{2r}{\pi (1-a)^2}\int\limits_{\CC\setminus D(T) }\Bigl|\Im \frac{1}{z}\Bigr|\dd |\nu|(z)
+\frac{2r}{\pi}\int\limits_r^{(a+1/a)T}\frac{|\nu|(x_0,t)}{t^2+r^2}\dd t. \end{multline} Доказательство оценки \eqref{nu:estT} и будет
означать, что определено выметание $\nu^{\bal}$ из $\CC^{\up}$. Объединения множества \eqref{xset} с кругом \eqref{xsetd} из замечания \ref{r:t1t2+-} по лемме  \ref{prxset} содержится в б\'ольшем круге из \eqref{DsD}, совпадающем в обозначениях \eqref{{0rt0}c} с кругом $D\bigl(x_0, (a+1/a)T\bigr)$. Следовательно, по предложениям \ref{pr:gmA} с оценкой \eqref{afminAM} и \ref{pr:gmAa} с оценкой
\eqref{afminAMaa} получаем 
\begin{equation}\label{afminAMg} \omega\bigl(z,[t_1,t_2]\bigr)\overset{\eqref{afminAM},
\eqref{afminAMaa}}{\leq} \frac{t_2-t_1}{\pi (1-a)^2}\, \Im \frac{1}{\bar z}\overset{\eqref{0rt0}}{=} \frac{2r}{\pi (1-a)^2}\, \Im
\frac{1}{\bar z} \end{equation} для всех точек $z\notin D\bigl(x_0, (a+1/a)T\bigr)$. Пусть \begin{equation}\label{nuiCD}
\nu_{\infty}:=\nu\bigm|_{\CC\setminus D\bigl(x_0, (a+1/a)T\bigr)} \end{equation} --- сужение заряда $\nu$ на внешность круга
$D\bigl(x_0, (a+1/a)T\bigr)$, включающего в себя по лемме \ref{prxset} меньший круг $D(T)$ из \eqref{DsD}. Интегрируя обе части
неравенства \eqref{afminAMg} по полной вариации $|\nu_{\infty}|\leq |\nu|$, сосредоточенной по определению \eqref{nuiCD} вне круга $
D\bigl(x_0, (a+1/a)T\bigr)\supset D(T)$, имеем \begin{equation}\label{nuinfty}
\bigl|(\nu_{\infty})^{\bal}\bigr|\bigl([t_1,t_2]\bigr)\overset{\eqref{afminAMg}}{\leq} \frac{2r}{\pi (1-a)^2}\int\limits_{\CC\setminus
D(T)}\Bigl|\Im \frac{1}{z}\Bigr|\dd |\nu|(z). \end{equation} Для остальной части $\nu_0\overset{\eqref{nuiCD}}{:=}\nu-\nu_{\infty}$
заряда $\nu$ с полной вариацией $|\nu_0|\leq |\nu|$ из оценки \eqref{es:upom} предложения \ref{pr:eshmb} и определения \ref{df:clbal}
следует \begin{multline}\label{nu0btt} \bigl|(\nu_{0})^{\bal}\bigr|\bigl([t_1,t_2]\bigr) \overset{ \eqref{df:bal0up}}{\leq}
\int\limits_{D\bigl(x_0,(a+1/a)T\bigr)}\omega \bigl(z,[t_1,t_2]\bigr) \dd|\nu|(z) \overset{\eqref{es:upom}}{\leq} |\nu|\bigl(\overline
D(x_0,r)\bigr) 
\\+\int\limits_{D(x_0, (a+1/a)T)\setminus D(x_0,r)} \dfrac{1}{\pi} \arctg \dfrac{2r|z-x_0|}{|z-x_0|^2-r^2}\dd|\nu|(z). \end{multline}
Последний интеграл здесь равен 
\begin{multline*} 
\int\limits_{r}^{(a+1/a)T} \dfrac{1}{\pi} \arctg \dfrac{2rt}{t^2-r^2}\dd
|\nu|(x_0,t)=\Bigl|\text{интегрирование по частям}\Bigr| \\=\frac{1}{\pi}\Bigl(\arctg
\frac{2r(a+1/a)T}{(a+1/a)^2T^2-r^2}\Bigr)|\nu|\bigl(x_0,(a+1/a)T\bigr) -\frac{1}{2}|\nu|(x_0,r)
\\+\frac{2r}{\pi}\int\limits_{r}^{(a+1/a)T} \frac{|\nu|(x_0,t)}{t^2+r^2}\dd t \leq \frac{2r}{\pi T}|\nu|\bigl(x_0,(a+1/a)T\bigr)
 + \frac{2r}{\pi}\int\limits_{r}^{(a+1/a)T} \frac{|\nu|(x_0,t)}{t^2+r^2}\dd t. 
\end{multline*} 
Отсюда, продолжая оценку \eqref{nu0btt} и учитывая \eqref{nuinfty}, ввиду равенства
$\nu=\nu_0+\nu_{\infty}$ и неравенства $|\nu|\leq |\nu_0|+|\nu_{\infty}|$ получаем требуемую в \eqref{nu:estT} оценку.

При $t_1t_2\geq 0$, не умаляя общности, можем считать, что $ t_1\geq 0$. 
Рассмотрим сначала более деликатный случай  $r\leq a|x_0|$. Для этого разобьём последний интеграл в правой части
\eqref{nu:estT} на сумму интегралов \begin{multline}\label{dl_tz}
\biggl(\int_r^{a|x_0|}+\int_{a|x_0|}^{(a+1/a)T}\biggr)\frac{|\nu|(x_0,t)}{t^2+r^2}\dd t \\ \leq
\int_r^{a|x_0|}\frac{|\nu|(x_0,t)}{t^2+r^2}\dd t+ \int_{a|x_0|}^{(a+1/a)T}\frac{|\nu^{\overline{\up}}|(x_0,t)}{t^2}\dd t \\\leq
\int_r^{a|x_0|}\frac{|\nu|(x_0,t)}{t^2+r^2}\dd t+ |\nu|\bigl(x_0,(a+1/a)T\bigr)\,\Bigl(\frac{1}{a|x_0|}-\frac{1}{(a+1/a)T}\Bigr) \\
\leq \int_r^{a|x_0|}\frac{|\nu|(x_0,t)}{t^2}\dd t + |\nu|^{\rad}\bigl((3/a)|x_0|\bigr)\frac{1}{a|x_0|}\, , \end{multline} 
поскольку
$$|\nu|\bigl(x_0,(a+1/a)T\bigr)\leq |\nu|^{\rad}\bigl((1+a+1/a)T\bigr)\overset{\eqref{0rt0}}{\leq} |\nu|^{\rad}\bigl((3/a)T\bigr).$$
 Отсюда согласно \eqref{dl_tz} из \eqref{nu:estT} получаем \eqref{nu:estTos}. 

Пусть теперь по-прежнему $t_1\geq 0$ и уже  $r>a|x_0|$. Тогда последний интеграл в правой части \eqref{nu:estT} оценивается сверху лишь вторым интегралом из левой части 
\eqref{dl_tz}. Продолжение этой оценки  так же, как и в \eqref{dl_tz},  дает оценку сверху через второе слагаемое в правой части \eqref{dl_tz}. 
Таким образом,  из \eqref{nu:estT} снова получаем \eqref{nu:estTos}. 
 \end{proof}
\begin{corollary}\label{th:b0+} Пусть $\nu \in \mathcal{M}^+(\CC)$. Следующие три утверждения попарно эквивалентны:
\begin{enumerate}[{\rm (b1)}] \item\label{b1} существует выметание $\nu^{\bal}\in \mathcal{M}^+(\CC)$ из $\CC^{\up}$, \item\label{b2}
для некоторого подмножества $B\in \mathcal B_{\rm b} (\RR)$ с $\lambda_{\RR}(B)>0$ \begin{equation}\label{e:lot} \int_{\CC^{\up}}
\omega (z, B )\dd \nu(z)<+\infty , \end{equation} \item\label{b3} мера $\nu$ удовлетворяет условию Бляшке \eqref{ineq:akh} около
$\infty$ в $\CC^{\up}$. \end{enumerate} \end{corollary} \begin{proof} Импликация (b\ref{b3})$\Rightarrow$(b\ref{b1}) --- частный
случай теоремы \ref{th:cup}. Импликация (b\ref{b1})$\Rightarrow$(b\ref{b2}) --- очевидное следствие равенства \eqref{df:bal0up}
определения \ref{df:clbal}. Из (b\ref{b2}) следует существование чисел $-\infty<t_1<t_2<+\infty$, для которых
\begin{equation}\label{e:loB} \int_{\CC^{\up}} \omega\bigl(z, (t_1,t_2)\bigr)\dd \nu(z)\overset{\eqref{e:lot}}{<}+\infty
\end{equation} Положим $a=1/2$ и $T:=\max\bigl\{|t_1|,|t_2|\bigr\}$. Тогда из оценки снизу \eqref{afminAM-} предложения \ref{pr:gmA+}
следует \begin{equation*}
\nu(z)\leq \int_{\CC^{\up}} \omega\bigl(z, (t_1,t_2)\bigr)\dd \nu(z)\overset{\eqref{e:loB}}{<}+\infty, \end{equation*} т.\,е.
выполнено \eqref{ineq:akh} при $r_0=2T$, и из (b\ref{b2}) выведено (b\ref{b3}). \end{proof} \begin{remark} Условие
$\lambda_{\RR}(B)>0$ на $B\in \mathcal B_{\rm b} (\RR)$ можно ослабить. Достаточно предполагать, что (b\ref{b2}) с \eqref{e:lot}
выполнено для некоторого множества $B$ ненулевой логарифмической емкости, т.\,е. неполярного $B$ \cite{Rans}. \end{remark}

\begin{corollary}\label{cor_L} Пусть $\nu\in \mc M (\CC)$ удовлетворяет условию Бляшке около $\infty$ в $\CC^{\up}$ и
$({\nu}^{\bal})^{\RR}$ --- функция распределения на $\RR$ выметания $\nu^{\bal}$. Тогда \begin{enumerate}[{\rm (i)}] \item\label{Li}
если $[x_1,x_2]\cap \supp \nu=\varnothing$ для отрезка $[x_1,x_2] \subset \RR_*$, то функция распределения
$({\nu}^{\bal})^{\RR}$ удовлетворяет условию Липшица на $[x_1,x_2]$; \item\label{Lii} если заряд $\nu$ конечного типа при порядке
$p\in\RR^+$ асимптотически отделён от $\RR$ в $\CC^{\overline{\up}}$ в том смысле, что \begin{equation}\label{der:R}
\liminf_{\substack{z\to \infty\\z\in \CC^{\overline{\up}}\cap \supp \nu}} \frac{\Im z}{|z|}>0, \end{equation} то найдутся числа $b>0$
и $r_0>0$, для которых в обозначениях \eqref{0rt0} при любых $|x_0|\geq r_0$ из $|t_2-t_1|\leq {|x_0|}/{2}$ следует неравенство
\begin{equation}\label{der:Res} 
\begin{split}
\bigl|(\nu^{\bal})^{\RR}(t_2)-(\nu^{\bal})^{\RR}(t_1)\bigr| \leq b|t_2-t_1||x_0|^{p-1} \\
\text{при всех $|t_2-t_1|\leq \frac{|x_0|}{2}$ и  $|x_0|\geq r_0$.} 
\end{split}
\end{equation} 
\end{enumerate} 
\end{corollary} 
\begin{proof} По определению \ref{df:clbal} можно считать, что заряд $\nu=\nu^{\overline{\up}}$, т.\,е. сосредоточен в $\CC^{\overline{\up}}$.
Используем обозначения \eqref{0rt0}.
	
\eqref{Li} Достаточно рассмотреть случай $0<r_0\leq x_1<x_2<+\infty$. По условию можно подобрать $a\in (0,1)$ так, что $\overline
D\bigl(x, a|x|\bigr)\cap \supp \nu =\varnothing$ для всех точек $x\in [x_1, x_2]$. Тогда для любого отрезка $[t_1,t_2]\subset [x_1,
x_2]$ с $x_0=(t_1+t_2)/2$ оценка \eqref{nu:estTos} теоремы~\ref{th:cup} упрощается до вида 
\begin{multline}\label{nu:estTo+}
\bigl|\nu^{\bal}\bigr|^{\RR}(t_2)-\bigl|\nu^{\bal}\bigr|^{\RR}(t_1) \\ \leq \frac{t_2-t_1}{2}\biggl( \sup_{x\in [x_1,x_2]} \frac{2|\nu
|^{\rad}\bigl(\,\frac{6}{a}|x|\bigr)}{|x|} +\frac{1}{(1-a)^2}\int\limits_{|z|\geq |x_0|}\Bigl|\Im \frac{1}{z}\Bigr|\dd |\nu|(z)\biggr).
\end{multline}

\eqref{Lii}  Условие \eqref{der:R} позволяет выбрать $a>0$ в неравенстве \eqref{nu:estTo+}  не зависящим от  $x_1:=t_1$, $ x_2:=t_2$ при достаточно большом $r_0\leq |x_0|$. 
 Для получения \eqref{der:Res} из \eqref{nu:estTo+}
  при $p<1$ в \eqref{nu:estTo+} 
используем неравенства
 \begin{multline}\label{Impm1}
\int\limits_{|z|\geq |x_0|}\Bigl|\Im \frac{1}{z}\Bigr|\dd |\nu|(z)\leq \int_{|x_0|}^{+\infty} \frac{\dd \nu^{\rad}(t)}{t} 
\\
\leq \int_{|x_0|}^{+\infty} \frac{\nu^{\rad}(t)}{t^2} \dd t\leq \const_{\nu} |x_0|^{p-1} .
\end{multline}
При  $p\geq 1$   по  условию Бляшке из оценки \eqref{nu:estTo+} следует \eqref{der:Res}. 
\end{proof}

\begin{propos}\label{cor:ges} Пусть заряд $\nu\in \mc M(\CC)$ удовлетворяет условию Бля\-шке \eqref{ineq:akh}. Если функция $g\colon
\RR^+\to \RR^+$ при всех $r>0$ удовлетворяет неравенству $g(r)>r$, то при любых $r>0$ \begin{equation}\label{es:grbR}
|\nu^{\bal}|^{\rad}(r) \leq |\nu|^{\rad}\bigl(g(r)\bigr)+\frac{2rg^2(r)}{\pi (g(r)-r)^2} \int\limits_{|z|\geq g(r)} \Bigl|\Im \frac{1}{z}\Bigr| \dd |\nu| (z).
\end{equation} 
Если $\nu$ конечного типа при порядке\/ $p\in \RR^+$ около $\infty$, то и $\nu^{\bal}$  конечного типа при порядке\/ $p$ около $\infty$. При $p\geq 1$ имеет место и точная оценка \begin{equation}\label{e:dnuR}
\limsup_{0\leq r\to+\infty}\frac{|\nu^{\bal}|^{\rad}(r)}{r^p} \leq \limsup_{0\leq r\to+\infty}\frac{|\nu|^{\rad}(r)}{r^p}\,.
\end{equation} \end{propos} \begin{proof} При фиксированном $r>0$ положим $a:=r/g(r)<1$. По определению \ref{df:clbal} из равенства
\eqref{df:bal0up} получаем \begin{multline}\label{df:bal0up+} |\nu^{\bal}|^{\rad}(r)\leq \int\limits_{\overline D(g(r))}
\omega\bigl(z, \overline D(r)\bigr)\dd |\nu| (z) +|\nu|\bigl(\overline D(r)\cap {{\CC}_{\overline \lw}}\bigr) \\+
\int\limits_{\CC^{\up}\setminus \overline D(g(r))} \omega\bigl(z, \overline D(r)\bigr)\dd |\nu|(z) \\ \overset{\eqref{afminAM}}{\leq}
|\nu| \bigl(\overline D(g(r))\bigr) +\frac{2r}{\pi (1-a)^2}\,\int\limits_{\CC^{\up}\setminus \overline D(g(r))} \Im \frac{1}{\bar z}
\dd |\nu|(z), \end{multline} где использованы очевидное неравенство $\omega (z,\overline D(r))\leq 1$ и предложение \ref{pr:gmA} при $t_1:=-r$ и $t_2:=r$   при условии \eqref{minM}. Подстановка $a=r/g(r)$ даёт неравенство \eqref{es:grbR}. Если $\type_p^{\infty}[\nu]<+\infty$,
то при $g(r)\equiv (1+\e)r$ с числом $\e>0$ из \eqref{es:grbR} получаем \begin{equation}\label{imnup}
\frac{|\nu^{\bal}|^{\rad}(r)}{r^p} \leq \frac{\nu|^{\rad}\bigl((1+\e)r\bigr)}{r^p}+\frac{2r(1+\e)^2}{\pi r^p{\e}^2}
\int\limits_{|z|\geq (1+\e)r} \Bigl|\Im \frac{1}{z}\Bigr| \dd |\nu| (z). \end{equation} В случае $p\geq 1$ последнее слагаемое здесь
--- величина порядка $o(1)$ при $r\to +\infty$, что даёт при $0<\e\to 0$ оценку \eqref{e:dnuR}, которая переходит в равенство при
$\supp \nu\subset \RR$. При $\type_p^{\infty}[\nu]<+\infty$ с $p<1$ из \eqref{imnup} с $\e=1/2$, применяя \eqref{Impm1}, получаем
$\type_p^{\infty}[\nu^{\bal}]<+\infty$. \end{proof}

Следующий результат был ранее доказан только  для положительных мер $\nu \in \mathcal M(\CC)$, но в наших определениях он легко переносится и на заряды. 
\begin{propos}[{\rm \cite[лемма 3.1]{Kha91}, детально в \cite[лемма
1.4.1]{KhDD92}}]\label{Lv:vpp} Пусть $\nu\in \mc M(\CC)$ --- заряд конечного типа при порядке $p\geq 1$, удовлетворяющий условию
Бляшке \eqref{ineq:akh}, $\supp \nu\subset \CC^{\overline{\up}}$. Тогда для любого $r_0\in \RR_*^+$ \begin{equation}\label{Lcob}
\biggl|\, \int\limits_{r_0<|z|\leq r} \frac{1}{z^p} \dd \, \bigl( \nu^{\bal}(z)-\nu (z)\bigr)\biggl|=O(1) \quad\text{при $r\to
+\infty$}, \end{equation} где для многозначной аналитической функции $z\mapsto z^p$ выбрана какая-либо её аналитическая ветвь в  $\CC^{\overline{\up}}_*$. 
\end{propos}

\subsection{ Выметание заряда на систему лучей}\label{bal_S_0} Значительная часть содержания этого п.~\ref{bal_S_0} в некоторой
степени развивает и обобщает результаты из \cite{Kha91} и \cite[гл.~I]{KhDD92}. В данном п.~\ref{bal_S_0} всюду $S$ --- множество
лучей на $\CC$ с началом в нуле. Далее такое множество лучей $S$ называем {\it системой лучей. \/} 

\subsubsection{Система лучей}\label{Sray} 
Систему лучей $S$ рассматриваем одновременно и как {\it точечное подмножество в\/} $\CC$, т.\,е. как множество всех точек на лучах, образующих систему $S$. В частности, система лучей $S$ {\it замкнутая,\/} если $S$ как точечное подмножество замкнуто в $\CC$. Далее всюду, часто не оговаривая специально, {\it рассматриваем только
непустые замкнутые системы лучей.\/} Связные компоненты дополнения $\CC\setminus S$ замкнутой системы лучей $S$ называем {\it углами,
дополнительными к $S$.\/} И наоборот, точечное множество, образующее замкнутый конус в $\CC$ над $\RR^+$, рассматриваем и как замкнутую систему лучей.

Для  точки $z\in \CC_*$, в отличие от договорённости перед предложением \ref{pr:gmAa}, значения ее аргументов $\Arg z\in\RR$ однозначны лишь с точностью до слагаемого, кратного $2\pi$, но $\Arg 0:=\RR$.  Для пары  $\alpha, \beta\in \RR$ 
полагаем $\angle\,
[\alpha,\beta]:=\bigl\{z\in \CC\colon [\alpha, \beta] \cap \Arg z\neq \varnothing\bigr\}$, 
$\angle (\alpha,\beta){:=}\{0\neq z\in \CC\colon (\alpha, \beta) \cap \Arg z\neq \varnothing\}$
--- соответственно {\it замкнутый угол\/} (с вершиной в нуле при\/
$\alpha \leq \beta$) и {\it открытый угол\/} --- оба {\it раствора\/} $\beta-\alpha\in \RR^+$ при $\alpha\leq \beta$. В  записи углов $\angle  (\alpha,\beta)$, дополнительных к системе лучей $S$, естественно предполагать $\beta\leq \alpha +2\pi$.

\subsubsection{Гармоническая мера для дополнения системы лучей}\label{df:harmm}
{\it Редукцией угла\/} $\angle (\alpha, \beta)$ ненулевого раствора {\it к верхней полуплоскости\/} называем конформную замену
переменной $z':=(ze^{-i\alpha})^{\frac{\pi}{\beta-\alpha}}$, где $z\in \angle\,[\alpha, \beta]$ 
рассматривается ветвь функции $z\mapsto z^{\frac{\pi}{\beta-\alpha}}$, положительная на положительной полуоси $\RR^+$. Пусть $z\in
\angle (\alpha, \beta)$ и $B'$ --- образ пересечения $B\cap \partial \angle\,[\alpha, \beta]\in \mathcal B(\CC)$ при редукции угла
$\angle (\alpha, \beta)$ к верхней полуплоскости. Тогда, в силу конформной инвариантности гармонической меры, для угла $\angle (\alpha, \beta)$ 
\begin{equation}\label{ioang} \omega\bigl(z,B;\angle (\alpha,
\beta)\bigr)\overset{\eqref{df:o}}{=}\omega(z',B')\overset{\eqref{seom:a}}{=} \int_{B'} \frac{1}{\pi}\Im \frac{1}{t-z} \dd
\lambda_{\RR}(t). \end{equation}
 
\begin{propos}\label{aalb} Пусть $a\in (0,1)$. При условии $a|z|\geq r$ имеем  \begin{equation}\label{aalbaa}
\omega\bigl(z,\overline{D}(r);\angle (\alpha, \beta)\bigr)\leq\frac{2r^{\frac{\pi}{\beta-\alpha}}}{\pi
(1-a^{\frac{\pi}{\beta-\alpha}})^2} \left(-\Im \frac{1}{(ze^{-i\alpha})^{\frac{\pi}{\beta-\alpha}}}\right), 
\end{equation} 
а при условии $ar\geq |z|$ справедлива  оценка 
\begin{equation}\label{aalbaad} 
\omega\bigl(z,\CC\setminus D(r);\angle (\alpha,
\beta)\bigr)\leq \frac{2r^{-\frac{\pi}{\beta-\alpha}}}{\pi \bigl(1-a^{\frac{\pi}{\beta-\alpha}}\bigr)^2} \Im\, (ze^{-i\alpha})^{\frac{\pi}{\beta-\alpha}}. 
\end{equation} 
\end{propos}
 \begin{proof} Из \eqref{ioang} для круга $B=\overline{D}(r)$
имеем \begin{equation*} \omega\bigl(z,\overline{D}(r);\angle (\alpha, \beta)\bigr)=
\omega\bigl((ze^{-i\alpha})^{\frac{\pi}{\beta-\alpha}}, [-r^{\frac{\pi}{\beta-\alpha}},r^{\frac{\pi}{\beta-\alpha}}] \bigr),
\end{equation*} Отсюда по предложению \ref{pr:gmA} при $r:=t_2=-t_1>0$ и условии $a|z|\geq r$, соответствующем \eqref{minM}, из  \eqref{afminAM} после редукции угла $\angle (\alpha, \beta)$ к $\CC^{\up}$ сразу следует \eqref{aalbaa}. Инверсия
$\star \colon z\mapsto z^{\star}:= 1/\bar{z}$ по всем $z\in \angle\,[\alpha, \beta]$ относительно единичной окружности $\partial \DD$ сохраняет, как
конформное отображение, рассматриваемую гармоническую меру и дает оценку \eqref{aalbaad}. \end{proof}

\begin{definition}[{\cite[определение 1.2]{Kha91}, \cite[определение 1.1.1]{KhDD92}}]\label{seom:aS} Пусть $S$ --- замкнутая система
лучей. {\it Гармонической мерой множества\/ $B\in \mathcal B(\CC)$ в точке\/ $z\in \CC$ для множества\/}-дополнения $\CC\setminus S$
называем функцию \begin{equation}\label{seom:aSf} (z,B)\mapsto \omega (z,B; \CC\setminus S)\in [0,1], \quad z\in \CC, \quad B\in
\mathcal B(\CC), \end{equation} равную гармонической мере $\omega\bigl(z,B;\angle (\alpha, \beta)\bigr)$ в точках $z$ из  дополнительных к $S$ углов $\angle (\alpha, \beta)$, и равную мере Дирака $\delta_z(B)$ от $B$ при $z\in S$. \end{definition}
Определение \ref{seom:aS} с \eqref{seom:aSf} согласованы с определением гармонической меры для $\CC^{\up}$ из п.~\ref{p_3_1} и
обозначением \eqref{seom:a}, где замкнутая система лучей $S$ образована всеми лучами с началом в нуле из ${\CC}_{\overline \lw}$ и
$\CC\setminus S=\CC\setminus {\CC}_{\overline \lw}=\CC^{\up}$.

\subsubsection{Выметание заряда на систему лучей и условие Бляшке}\label{bsrcB} Редукция угла $\angle (\alpha, \beta)$ к $\CC^{\up}$
из п. \ref{df:harmm} позволяет обобщить определение \ref{df:clbal} как \begin{definition}\label{df:clbala} Пусть $<-\infty<\alpha
<\beta \leq\alpha+ 2\pi<+\infty$ и $\nu\in \mathcal M(\CC)$. Заряд $\nu^{\bal}_{\CC\setminus \angle (\alpha, \beta)}\in \mathcal M(\CC)$, определяемый на каждом  $B\in \mathcal B(\CC)$ равенством \begin{multline}\label{df:bal0upa}
\nu^{\bal}_{\CC\setminus \angle (\alpha, \beta)}(B)\overset{\eqref{ioang}}{:=}\int_{\angle (\alpha, \beta)}
\omega\bigl(z,B;\angle (\alpha, \beta)\bigr)\dd \nu(z) \\ +\nu\bigl(B\cap {(\CC\setminus \angle (\alpha, \beta))}\bigr),
\end{multline} называем {\it выметанием заряда\/ $\nu$ из угла\/} $\angle (\alpha, \beta)$ на дополнение этого угла $\CC\setminus
\angle (\alpha, \beta)$, или систему лучей $\CC\setminus \angle (\alpha, \beta)$. \end{definition}

\begin{definition}[{\cite[определение 1.3]{Kha91}, \cite[определение 1.1.2]{KhDD92}}]\label{df:clbalS} Пусть $S$ --- замкнутая система
лучей на $\CC$, $\nu\in \mathcal M(\CC)$. Заряд $\nu_S^{\bal}\in \mathcal M(\CC)$, определяемый на каждом борелевском множестве $B\in
\mathcal B(\CC)$ равенством \begin{equation}\label{df:bal0upS} \nu^{\bal}_S(B):=\int_{\CC} \omega(z, B; \CC\setminus S)\dd \nu(z),
\end{equation} называем {\it выметанием заряда\/ $\nu$ на\/} $S$, или {\it из\/} $\CC\setminus S$, где в  $\nu_S^{\bal}$ нижний
индекс $S$ будем опускать, если  ясно, какая система лучей $S$ рассматривается. \end{definition} Определения
\ref{df:clbal}, \ref{df:clbala} и \ref{df:clbalS} согласованы. Так, выметание из верхней полуплоскости $\CC^{\up}$ на замкнутую нижнюю
полуплоскость ${\CC}_{\overline \lw}$ --- это выметание на замкнутую систему лучей $S$, образованную всеми лучами с началом в нуле из
${\CC}_{\overline \lw}$. Поскольку множество дополнительных к $S$ углов не более чем счётно, из определений
\ref{df:clbal}--\ref{df:clbalS} сразу следуют свойства \begin{enumerate}[{\rm (s1)}] \item\label{s1} {\it если для каждого из зарядов
$\nu_1,\nu_2\in \mathcal M(\CC)$ существуют выметания $\nu_1^{\bal}, \nu_2^{\bal}\in \mathcal M(\CC)$ на одну и ту же замкнутую
систему лучей, то для любых $c_1,c_2\in \RR$ существует выметание $(c_1\nu_1+c_2\nu_2)^{\bal}=c_1\nu_1^{\bal}+c_2\nu_2^{\bal}$ на ту
же систему лучей} (линейность); \item\label{s2} {\it для заряда $\nu\in \mathcal M(\CC)$ существание выметания
$\nu^{\bal}_{\CC\setminus \angle (\alpha,\beta)}\overset{\eqref{df:bal0upa}}{\in} \mathcal M(\CC)$ из каждого угла
$\angle (\alpha,\beta)$, дополнительного к $S$, эквивалентно существованию выметания $\nu^{\bal}_S\in \mathcal M (\CC)$, получаемого
применением не более чем счетного числа выметаний из всех дополнительных к $S$ углов в смысле определения\/} \ref{df:clbala};
\item\label{s3} {\it выметание $\nu^{\bal}_S$ заряда $\nu \in \mathcal M(\CC)$ на систему лучей $S$ при его существовании --- это
заряд из $\mathcal M (\CC )$ с носителем $\supp \nu_S^{\bal}\subset S$;} \item\label{s4} {\it выметание $\nu^{\bal}_S$ положительной
меры $\nu \in \mathcal M^+(\CC)$ на систему лучей $S$ при его существовании --- положительная мера из $\mathcal M^+(\CC )$.}
\end{enumerate}

Существование выметания $\nu^{\bal}\in \mathcal M(\CC)$ эквивалентна конечности \eqref{df:bal0upS} для любого $B\in \mathcal B_{\rm
b}(\CC)$. Для существования выметания $\nu^{\bal}\in \mathcal M(\CC)$ достаточно существования выметания $|\nu|^{\bal}\in \mathcal
M^+(\CC)$ полной вариации $|\nu|\in \mathcal M^+(\CC)$. По построению для вариаций заряда и его выметания 
\begin{equation}\label{pv+-} |\nu^{\bal}|\leq |\nu|^{\bal}, \quad (\nu^{\bal})^{\pm}\leq (\nu^{\pm})^{\bal}\,. \end{equation}
\begin{example} Неравенства в \eqref{pv+-} могут быть строгими. Пусть $S=\RR$. Рассмотрим заряд $\nu:=\delta_i-\delta_{-i}$ ---
разность мер Дирака в точках $i\in \CC^{\up}$ и $-i\in \CC_{\lw}$. Тогда по определению \ref{df:clbalS}  $\nu_{\RR}^{\bal}=0=|\nu_{\RR}^{\bal}|=(\nu_{\RR}^{\bal})^{\pm}$, 
а в обозначении \eqref{nuR} для плотностей функций распределения на $\RR$ имеем 
\begin{equation*}
\dd\,\bigl(|\nu|^{\bal}_{\RR}\bigr)^{\RR}(t)=\frac{2}{\pi
(1+t^2)}\dd t \, , \quad \dd\,\bigl((\nu^{\pm}_{\RR})^{\bal}\bigr)^{\RR}(t)=\frac{1}{\pi (1+t^2)}\dd t , \end{equation*} 
откуда
$|\nu|^{\bal}_{\RR}(\RR)=2>0=|\nu_{\RR}^{\bal}|(\RR)$ и $(\nu^{\pm})^{\bal}_{\RR}(\RR)=1>0=(\nu_{\RR}^{\bal})^{\pm}(\RR)$.
\end{example}

Редукция к $\CC^{\up}$ позволяет переформулировать определение \ref{c:akh} для угла.
\begin{definition}\label{c:akhan} Пусть $<-\infty<\alpha <\beta \leq\alpha+ 2\pi<+\infty$. Заряд $\nu\in \mathcal M(\CC)$ удовлетворяет {\it условию Бляшке около бесконечности в угле $\angle (\alpha,\beta)$,\/} если 
\begin{equation}\label{ineq:akhan} 
\int\limits_{\angle (\alpha,\beta)\setminus D(r_0)}-\Im
\frac{1}{(ze^{-i\alpha})^{\frac{\pi}{\beta-\alpha}}}\dd |\nu|(z)<+\infty\quad\text{для  $r_0\in \RR^+_*$}. 
\end{equation} 
Заряд $\nu\in \mathcal M(\CC)$ удовлетворяет {\it условию Бляшке около бесконечности вне замкнутой системы лучей $S$,\/} если заряд $\nu$ удовлетворяет условию Бляшке около бесконечности в любом угле, дополнительном к $S$. 
\end{definition}

\begin{theorem}\label{th:Sb} Если для заряда $\nu\in \mathcal M (\CC)$ выполнено условие Бляшке вне системы лучей $S$, то существует
выметание $\nu_S^{\bal}\in \mathcal M (\CC)$.

Для меры $\nu\in \mathcal M^+(\CC)$ выполнение условия Бляшке вне системы лучей $S$ эквивалентно существованию выметания
$\nu_S^{\bal}\in \mathcal M^+(\CC)$. \end{theorem} \begin{proof} По свойству (s\ref{s2}) импликация для заряда $\nu\in \mathcal M
(\CC)$ --- следствие теоремы \ref{th:cup}, а эквивалентность для меры $\nu\in \mathcal M(\CC)$ вытекает из эквивалентности
(b\ref{b3})$\Leftrightarrow$(b\ref{b1}) следствия \ref{th:b0+}.\end{proof}

Оценку роста выметаемого заряда или меры дает следующее обобщение предложения \ref{cor:ges} c незначительным ослаблением оценки
\eqref{es:grbR}. \begin{theorem}\label{th:0nub} Пусть $\nu\in \mathcal{M}(\CC)$, $S$ --- замкнутая система лучей в $\CC$, подмножество $R\subset \RR_*^+$ неограниченное в $\RR^+$ и существует функция 
\begin{subequations}\label{rgrB} 
\begin{align} 
g\colon R\to \RR_*^+,&\quad\text{для которой}\quad a(r):=\frac{r}{g(r)}<1 \quad\text{при всех $r\in R$}; \tag{\ref{rgrB}a}\label{{rgrB}a}\\
C^+_S(r, g;\nu)&:= \sum_{\angle (\alpha,\beta)} r^{\frac{\pi}{\beta-\alpha}}\int\limits_{\angle (\alpha,\beta)\setminus D(g(r))}-\Im
\frac{1}{(ze^{-i\alpha})^{\frac{\pi}{\beta-\alpha}}}\dd |\nu|(z), 
\tag{\ref{rgrB}b}\label{{rgrB}b} 
\end{align}
\end{subequations} 
где сумма в правой части берется по всем дополнительным к $S$ углам $\angle (\alpha,\beta)$. Тогда при конечности
\eqref{{rgrB}b} для всех $r\in R$ существует выметание \eqref{df:bal0upS} с оценками \begin{equation}\label{nubrB} |\nu^{\bal}|^{\rad}
(r)\leq |\nu|^{\rad}\bigl(g(r)\bigr) +\frac{8g^2(r)}{\pi \bigl(g(r)-r\bigr)^2}C^+_S(r,g;\nu) \quad\text{для всех $r\in R$}. 
\end{equation} 
\end{theorem} 
\begin{proof} Достаточно при произвольном $r\in R$
доказать \eqref{nubrB}. Пусть $\nu_{\infty}:=|\nu|\bigm|_{\CC\setminus D(g(r))}\in \mathcal M^+(\CC)$ --- сужение полной вариации
$|\nu|$ на внешность круга $D(g(r))$, $\nu_r:=|\nu|-\nu_{\infty}\in \mathcal M^+(\CC)$ --- сужение полной вариации $|\nu|$ на круг
$D(g(r))$. Из определяющих соотношений \eqref{df:bal0upS} предложения
\ref{df:clbalS} имеем 
\begin{equation}\label{df:bal0upSg} (\nu_r)^{\bal}_S\bigl(D(r)\bigr)=\int\limits_{D(g(r))} \omega\bigl(z, D(r);
\CC\setminus S\bigr)\dd \nu_r(z)\leq 1\cdot \nu_r^{\rad}(g(r)).
\end{equation} В силу \eqref{{rgrB}a}
при $z\in \supp \nu_{\infty}$ выполнено условие $a(r)|z|\geq r$ 
предложения \ref{aalb} и из оценки \eqref{aalbaa} в
обозначениях из \eqref{rgrB} имеем оценку 
\begin{multline*} (\nu_{\infty})^{\bal}_S\bigl(D(r)\bigr)\leq \left(
\sup_{0<\beta-\alpha\leq 2\pi} \frac{2}{\pi \bigl(1-a^{\frac{\pi}{\beta-\alpha}}(r)\bigr)^2}\right)\, C^+_S(r,g;\nu)\\
=\frac{2g(r)}{\pi \bigl(1-\sqrt{a(r)}\,\bigr)^2} C^+_S(r,g;\nu) =\frac{2g(r)}{\pi \bigl(\sqrt{g(r)}-\sqrt{r}\,\bigr)^2} C^+_S(r,g;\nu)
\end{multline*} что при $C^+_S(r,g;\nu)<\infty$ вместе с
\eqref{df:bal0upSg} даёт \eqref{nubrB}. 
\end{proof} Для положительных мер конечного порядка имеет место завершенная
\begin{theorem}\label{th:3p} Пусть $\nu\geq 0$ --- мера конечного типа при порядке $p\in \RR^+$. 
Тогда попарно эквивалентны следующие четыре утверждения:
\begin{enumerate}[{\rm (i)}] 
	\item\label{bi} $\int\limits_{\CC} \omega\bigl(z, D(r_0); \CC\setminus S\bigr)\dd
\nu(z)\overset{\eqref{df:bal0upS}}{<}+\infty$ для некоторого числа $r_0>0$; 
\item\label{bii} для меры $\nu$ в каждом из дополнительных
к $S$ углов  $\angle (\alpha,\beta)$ раствора $\beta-\alpha\geq \pi/p$, число которых не больше $2p$, выполнено условие Бляшке
\eqref{ineq:akhan} определения\/ {\rm \ref{c:akhan}} около бесконечности; 
\item\label{biii} заряд $\nu\in \mathcal M(\CC)$ удовлетворяет условию Бляшке
около бесконечности вне системы лучей $S$; 
\item\label{biv} определено выметание
$\nu_S^{\bal}\overset{\eqref{df:bal0upS}}{\in}\mathcal{M}(\CC)$ и $\type_p^{\infty}[\nu_S^{\bal}]<+\infty$. \end{enumerate}
\end{theorem} \begin{proof} Эквивалентности \eqref{bi}\,$\Leftrightarrow$\,\eqref{bii}\,$\Leftrightarrow$\,\eqref{biv}\, доказаны в
\cite[теорема 1.1]{Kha91}, \cite[теорема 1.1.1]{KhDD92} на основе ослабленных версий \cite[лемма 1.1]{Kha91}, \cite[лемма
1.1.1]{KhDD92} оценок сверху и снизу гармонической меры из предложений \ref{pr:gmA} и \ref{pr:gmA+}. Импликация
\eqref{bii}\,$\Rightarrow$\,\eqref{biv} может быть получена и из теоремы \ref{th:0nub}. Импликация
\eqref{biii}\,$\Rightarrow$\,\eqref{bii} очевидна. Рассуждения, доказывающие импликацию \eqref{bii}\,$\Rightarrow$\,\eqref{biii},
также содержаться в \cite[доказательство теоремы 1.1]{Kha91}, \cite[доказательство теоремы 1.1.1]{KhDD92} в части, касающейся
выметания меры $\mu_2$ из всех дополнительных к $S$ углов раствора $<\pi/p$. \end{proof}
	
\begin{remark}\label{df:baldop} Если $p\in \RR^+$,   $\nu\in \mathcal{M}^+(\CC)$  с $\type_p^{\infty}[\nu]<+\infty$ и выполнено одно (любое) из четырех попарно   	эквивалентных утверждений \eqref{bi}--\eqref{biv} теоремы \ref{th:3p}, то  	в \cite[определение 1.4]{Kha91}, \cite[определение 1.1.3]{KhDD92} 	 система лучей $S$ называется  {\it допустимой для меры\/} $\nu$.
\end{remark}

\subsubsection{Условие Линделёфа}\label{Lc} 
\begin{definition}[{\cite[следствие 2.1]{Kha91}, \cite[следствие
1.3.1]{KhDD92}}]\label{df:Lc} Для заряда $\nu\in \mathcal M(\CC)$ выполнено {\it условие Линделёфа рода $q\in \NN$,} если для некоторого числа $r_0>0$  имеет место асимптотическое соотношение 
\begin{equation}\label{con:Lp} \Biggl|\;\int\limits_{D(r)\setminus D(r_0)} \frac{1}{z^q}
\dd \nu(z) \Biggr|=O(1) \quad\text{при $r\to +\infty$.} 
\end{equation} 
\end{definition} Следующий результат основан на предложении
\ref{Lv:vpp}. \begin{theorem}[{\cite[следствие 2.1, теорема 3.1]{Kha91}, \cite[следствие 1.3.1, теорема 1.4.1]{KhDD92}}]\label{th:Lc}
Пусть $p\in \NN$ и   мера $\nu\in \mathcal{M}^+(\CC)$ с $\type_p^{\infty}[\nu]<+\infty$ удовлетворяет условию Бляшке вне  системы лучей $S$. Если мера $\nu$ удовлетворяет условию Линделёфа \eqref{con:Lp} рода $p$, то выметание $\nu_S^{\bal}$ меры $\nu$ на систему
лучей $S$ также удовлетворяет условию Линделёфа рода $p$. \end{theorem}

\begin{remark}\label{b:am} Пусть $(r_n)_{n\in \NN}$ --- возрастающая неограниченная последовательность чисел из $\RR^+_*$. Для меры
$\nu\in \mathcal M^+(\CC)$, удовлетворяющей условию Бляшке около бесконечности вне замкнутой системы лучей $S$ в смысле определения
\ref{c:akhan}, выметание $\nu_S^{\bal}$ на $S$ можно построить и как $*$-слабый предел \cite[{\bf A.4}]{Rans} последовательности
выметаний меры $\nu$ из не более чем счётного числа ограниченных в $\CC$ объединений секторов $D(r_n)\cap (\CC\setminus S)$
\cite[гл.~IV, \S~1]{L} с теми же итоговыми результатами, что и выше в \S~\ref{kvm}. Но такой подход технически гораздо более
трудоёмкий, во всяком случае в нашей реализации. В настоящей работе он и не потребуется, если использовать сведения, технику и
результаты из следующего \S~\ref{imu}. \end{remark}

\section{ Интеграл семейства мер по мере}\label{imu} \setcounter{equation}{0} Два различных взгляда на теорию меры и соответственно
интегрирования (с одной стороны мера как функция множеств и интеграл Лебега, с другой --- мера как линейный функционал и интеграл
Радона) излагаются почти всегда раздельно (ср. \cite{DSh} и \cite{B}; исключение в \cite{M}). В вопросах, рассматриваемых в настоящей
работе, эти подходы дают один и тот же результат. Однако ряд моментов теории интегрирование достаточно подробно рассмотрен в доступной
литературе лишь для интегралов Радона. В целях автономности изложения, удобства ссылок и во избежание разночтений и несогласованности
по толкованию терминов приведем некоторые понятия и факты теории интегрирования в духе Н.~Бурбаки \cite{B}, адаптированной на частный
случай подмножеств в конечномерном пространстве.

\subsection{Повторные интегралы}\label{pI} Пусть $k\in \NN$, $Z$ --- локально компактное подпространство в $\RR^k$ с естественной
евклидовой топологией; $C_0 (Z)$ --- пространство непрерывных действительных функций на $Z$ с компактным носителем, или $\mathcal K
(Z)$ в обозначениях из \cite{B}, а $\mu$ --- положительная мера Радона на $Z$, т.\,е. положительный линейный функционал на $C_0(Z)$.
Функция $f\colon Z\to \RR_{\pm\infty}$ является $\mu$-интегрируемой на $Z$, если определён  \begin{equation}
\text{\it интеграл}\quad \mu\bigl(|f|\bigr)=\int |f| \dd \mu\in \RR \quad\text{\cite[гл.~IV, \S~4]{B}}. 
\end{equation}

Пусть $p\in \NN$, $X\subset \RR^p$ --- локально компактное подпространство в $\RR^p$, \begin{equation}\label{s:k} T=\{\tau_z\colon
z\in Z\}\quad \text{\it--- семейство положительных мер Радона на\/ $X$.} \end{equation} Предположим, что для любой функции $f\in
C_0(X)$ функция $z\mapsto \tau_{z}(f)$ является $\mu$-интегрируемой на $Z$. Положительный линейный функционал $\nu$ на $C_0(X)$,
определенный по правилу \cite[гл.~V, \S~3]{B}, \cite[Введение, \S~1]{L} \begin{equation*}\label{semtmu} 
\nu(f):=\int \tau_z(f) \dd \mu(z)\, , 
\end{equation*} --- {\it интеграл семейства мер\/} $T$ из \eqref{s:k} {\it по мере\/ $\mu$}. Обозначается это как
\begin{equation}\label{intmm} 
\nu=\int\tau_{z}\dd \mu(z). 
\end{equation} В принятых в этом п.~\ref{pI} обозначениях в несколько
ослабленной, но достаточной для наших применений форме справедливо  развитие и обобщение классической теоремы Фубини  о повторных интегралах. 
\begin{Thopi}[{\cite[гл.~V, \S~3, п.~4, теорема 1]{B}}] Пусть семейство мер $T$ из \eqref{s:k} удовлетворяет
следующим условиям: 
\begin{enumerate}[{\rm (1)}] 
	\item\label{1mu} Для любой функции $f\in C_0(X)$ функция $z\mapsto \tau_z(f)$
$\mu$-инт\-е\-г\-р\-и\-р\-уема. 
\item\label{2mu} Отображение $z\mapsto \bigl\{\tau_z(f)\bigr\}_{f\in C_0(X)}$ пространства $Z$ в
пространство $\RR^{C_0(X)}$ функций на $C_0(X)$ со значениями в $\RR$ $\mu$-измеримо. 
\end{enumerate} Пусть $F\colon X\to \RR_{\pm\infty}$ --- $\nu$-интегрируемая функция и для
$\mu$-почти всех $z\in Z$ функция $F$ является $\tau_z$-интегрируемой, а для каждой функции $f\in C_0(X)$ отображение $z\mapsto
\tau_z(f)$ непрерывно. Тогда функция $z\mapsto \tau_z(F)$ является $\mu$-интегр\-и\-р\-у\-е\-м\-ой и 
\begin{equation} \int_X F(x)\dd
\nu (x)=\int_Z \left(\int_X F(x)\dd \tau_z(x)\right) \dd \mu (z). 
\end{equation} 
\end{Thopi} 
\subsection{Выметание и повторные интегралы} Здесь мы используем теорему о повторных интегралах в частном случае, когда $k=p=2$, $\RR^2$ отождествлено с $\CC$ и
$Z=X=\CC$, а в роли семейства $T$ из \eqref{s:k} рассматривается семейство выметаний мер Дирака на систему лучей $S$.
\begin{theorem}\label{th:IB} Пусть $S$ --- замкнутая система лучей
лучей в $\CC$, $z\in \CC$, $\delta_z^{\bal}$ --- выметание на $S$ меры Дирака\/ $\delta_z$ и существует\/ {\rm (см. теоремы \ref{th:Sb}--\ref{th:3p})} выметание $\nu^{\bal}$  на $S$  заряда\/ $\nu\in \mathcal M (\CC)$. Тогда
\begin{enumerate}[{\rm (i)}] \item\label{dei} для любого  $B\in \mathcal{B}(\CC)$ с 
$ \boldsymbol{1}_B(z):=1$ при $z\in B$ и  $\boldsymbol{1}_B(z):=0$  при $z\notin B$
\begin{equation}\label{1balo} \int
\mathbf{1}_B (z')\dd \delta_z^{\bal}(z')=\omega(z,B;\CC\setminus S) \quad\text{при всех $z\in \CC$}, \end{equation} т.\,е.
$\delta_z^{\bal}=\omega(z,\cdot \,;\CC\setminus S)$ для всех $z\in \CC$; \item\label{deii} для любой функции $f\in C_0(\CC)$ функция
$z\mapsto \delta_z^{\bal}(f)$ непрерывна на $\CC$; \item\label{deiv} для любой $\nu^{\bal}$-интегрируемой функции $F\colon \CC\to
\RR_{\pm\infty}$ 
\begin{subequations}\label{c:del} \begin{align} \int F(z)\dd \nu^{\bal}(z)&=\int \mathcal P_{\CC\setminus S}F(z)\dd
\nu (z), \tag{\ref{c:del}a}\label{{c:del}a}\\ \intertext{где подынтегральная функция в правой части \eqref{{c:del}a}} \mathcal
P_{\CC\setminus S}F(z)&\overset{\eqref{df:PDf}}{:=}\int F(z')\dd \delta_z^{\bal}(z') \tag{\ref{c:del}b}\label{{c:del}b} \end{align}
\end{subequations} $\nu$-интегрируема и называется далее интегралом Пуассона функции $F$ на дополнении $\CC\setminus S$ системы
лучей\/ $S$\/ {\rm (ср. с \eqref{df:PDf})}; \item\label{deiii} выметание $\nu^{\bal}=(\nu^+)^{\bal}-(\nu^-)^{\bal}$ --- разность
выметаний на $S$ верхней и нижней вариаций $\nu$, равных интегралу семейства мер $\delta_z^{\bal}$ по $\nu^{\pm}$, т.\,е.
 имеют место равенства
 \begin{equation}\label{eqomd} (\nu^{\pm})^{\bal}\overset{\eqref{intmm}}{=}\int \delta_z^{\bal} \dd \nu^{\pm}(z)= \int \omega(z,\cdot
\,;\CC\setminus S) \dd \nu^{\pm}(z). 
\end{equation} 
\end{enumerate} 
\end{theorem}
 \begin{proof} \eqref{dei}. Следует из очевидных
равенств \begin{equation*} \omega(z,B;\CC\setminus S)=\int \omega(z',B;\CC\setminus S) \dd \delta_z(z')
\\\overset{\eqref{df:bal0upS}}{=}\delta_z^{\bal}(B)=\int \mathbf{1}_B (z')\dd \delta_z^{\bal}(z'). \end{equation*} \eqref{deii}. Пусть
$f\in C_0(\CC)$. Полагаем 
\begin{equation}\label{dfFF} F(z):=\int f(z')\dd \delta_z^{\bal}(z') 
\end{equation} 
Необходимо доказать непрерывность функции $F$ на $\CC$. По определению гармонической меры и выметания функция $F$ из \eqref{dfFF} является гармоническим
продолжением функции $f$ с границы каждого дополнительного к $S$ угла внутрь этого угла. Из известных классических фактов функция $F$
непрерывна в замыкании каждого дополнительного к $S$ угла и совпадает с $f$ на $S$. 
Это еще не влечет за собой непрерывности функции $F$ на всей плоскости $\CC$. Для доказательства последнего необходимо рассмотреть и возможный случай последовательности  точек $(z_n)_{n\in \NN}$, стремящейся к точке $z_0\in S$ в предположении, что каждая из этих точек $z_n$ попадает в, вообще говоря, различные замыкания $\angle\,[\alpha_n,\beta_n]$, $n\in \NN$, некоторых дополнительных к $S$ углов $\angle (\alpha_n,\beta_n)$ раствора $\beta_n-\alpha_n \underset{n\to \infty}{\longrightarrow} 0$. 
При этом необходимо убедиться, что $\lim\limits_{n\to\infty}F(z_n)=f(z_0)$. Поскольку функция $F$ из \eqref{dfFF}, совпадающая с $f$ на  $S$, непрерывна в замыкании $\angle\,[\alpha_n,\beta_n]$ каждого дополнительного к $S$ угла $\angle (\alpha_n,\beta_n)$, можно   считать, что $z_n\in \angle (\alpha_n,\beta_n)$ при всех $n\in \NN$. Перейдем к рассмотрению такой ситуации. 

 Для произвольного числа $d>0$ выберем достаточно
малое число $r>0$ так, что \begin{equation}\label{drM} \sup_{z\in D(z_0,2r)}\bigl|f(z)-f(z_0)\bigr|\leq d; \quad M:=\sup_{z\in
\CC}\,\bigl|f(z)\bigr|. \end{equation} Для некоторого номера $n_d\in \NN$ при $n\geq n_d$ все точки $z_n$ лежат в $D(z_0,r)$, а
стороны-границы каждого угла $\angle (\alpha_n,\beta_n)$ пересекают границу $\partial D(z_0,r)$ этого круга $D(z_0,r)$. Увеличивая, если нужно, номер $n_d$,
можно добиться того, что для некоторого числа $a\in (0,1)$, не зависящего от $n\geq n_d$, 
имеем
\begin{equation}\label{azzn0} a|z_n|\geq |z| \quad \text{при $z\in {\rm C}_n^0$}; \quad |z_n|\leq a|z| \quad \text{при $z\in {\rm
C}_n^{\infty}$}, 
\end{equation} 
где ${\rm C}_n^0$ и ${\rm C}_n^{\infty}$ --- соответственно ограниченная и неограниченная  в $\CC$ связные компоненты множества
$\angle (\alpha_n,\beta_n)\setminus D(z_0,2r)$. Заметим, что при $z_0=0$ множество ${\rm C}_n^0$ пусто. Оценим разность 
\begin{equation}\label{dfFFf} f(z_0)-F(z_n)\overset{\eqref{dfFF}}{=}\int \bigl(f(z_0)-f(z)\bigr)\dd \delta_{z_n}^{\bal}(z).
\end{equation} Разбивая последний интеграл на сумму интегралов по трем множествам $D(z_0,r)$, ${\rm C}_n^0$ и ${\rm C}_n^{\infty}$,
получаем неравенство типа теоремы о двух константах: \begin{multline}\label{lm_oma}
\bigl|f(z_0)-F(z_n)\bigr|\overset{\eqref{dfFFf}}{\leq} \int_{D(z_0,r)} \bigl|f(z_0)-f(z)\bigr| \dd \delta_{z_n}^{\bal}(z)\\
+\left(\int_{{\rm C}^0_n}+\int_{{\rm C}^{\infty}_n}\right) \bigl(|f(z_0)|+|f(z)|\bigr) \dd
\delta_{z_n}^{\bal}(z)\overset{\eqref{drM}}{\leq} d+\int\limits_{{\rm C}^{0}_n\cup {\rm C}^{\infty}_n} 2M\dd \delta_{z_n}^{\bal}(z)\\
\overset{\eqref{1balo}}{=}d+2M \omega\bigl(z_n, {\rm C}^{0}_n\cup {\rm C}^{\infty}_n; \CC \setminus S\bigr) =d+2M \omega\bigl(z_n,
{\rm C}^{0}_n\cup {\rm C}^{\infty}_n; \angle (\alpha_n,\beta_n)\bigr)\\ \leq d+2M \omega\bigl(z_n, {\rm C}^{0}_n;
\angle (\alpha_n,\beta_n)\bigr)+2M \omega\bigl(z_n, {\rm C}^{\infty}_n; \angle (\alpha_n,\beta_n)\bigr), \end{multline} где
последнее равенство следует из определения гармонической меры для дополнения системы лучей из 
пп.~\ref{df:harmm} ввиду $z_n\in
\angle (\alpha_n,\beta_n)$. Покажем, что оба последних слагаемых в правой части \eqref{lm_oma} стремятся к нулю при $z_n\underset{n\to \infty}{\longrightarrow}  z_0$. Для
этого положим \begin{equation}\label{tnozn} t_n^0:=\sup_{z\in {\rm C}^{0}_n}|z|, \quad t_n^{\infty}:=\inf_{z\in {\rm
C}^{\infty}_n}|z|. \end{equation} При этом из соотношений \eqref{azzn0} следует \begin{equation}\label{aznd} a|z_n|\geq t_n^0,\quad
|z_n|\leq at_n^{\infty} \quad\text{при $n\geq n_d$}. \end{equation} В обозначениях \eqref{tnozn} имеем ${\rm C}^{0}_n\subset
\overline{D}( t_n^0)$ и ${\rm C}^{\infty}_n\subset \CC\setminus D(t_n^{\infty})$, откуда\begin{multline}\label{lm_omaD}
\bigl|f(z_0)-F(z_n)\bigr|\overset{\eqref{lm_oma}}{\leq} d+2M \omega\bigl(z_n, {\rm C}^{0}_n; \angle (\alpha_n,\beta_n)\bigr)+2M
\omega\bigl(z_n, {\rm C}^{\infty}_n; \angle (\alpha_n,\beta_n)\bigr)\\ \leq d+2M \omega\bigl(z_n, \overline{D}( t_n^0);
\angle (\alpha_n,\beta_n)\bigr) +2M \omega\bigl(z_n, \CC\setminus D(t_n^{\infty}); \angle (\alpha_n,\beta_n)\bigr) \end{multline} По
предложению \ref{aalb} при условии $a|z|\geq r$, соответствующему первому неравенству в \eqref{aznd} с $z=z_n$ и $r=t_n^0$, из
\eqref{aalbaa} получаем \begin{multline}\label{Dt} \omega\bigl(z_n, \overline{D}( t_n^0); \angle (\alpha_n,\beta_n)\bigr) \leq
\frac{2 (t_n^0)^{\pi/(\beta_n-\alpha_n)}}{\pi\bigl(1-a^{\pi/(\beta_n-\alpha_n)}\bigr)^2}\, \left(-\Im
\frac{1}{(z_n)^{\pi/(\beta_n-\alpha_n)}}\right)\\ \leq \frac{2}{\pi \bigl(1-\sqrt{a}\,\bigr)^2}
\left(\frac{t_n^0}{|z_n|}\right)^{\pi/(\beta_n-\alpha_n)} \overset{\eqref{aznd}}{\leq} \frac{2}{\pi \bigl(1-\sqrt{a}\,\bigr)^2} \,
a^{\pi/(\beta_n-\alpha_n)} \underset{n\to \infty}{\longrightarrow} 0 \,, \end{multline} 
поскольку $(0,2\pi]\ni \beta_n-\alpha_n\to 0$
при $n\to \infty$ и $0<a<1$. Аналогично, по предложению \ref{aalb} при условии $ar\geq |z|$, соответствующему второму неравенству в
\eqref{aznd} с $z=z_n$ и $r=t_n^{\infty}$, из \eqref{aalbaad} получаем 
\begin{multline}\label{Dti} \omega\bigl(z_n, \CC\setminus
D(t_n^{\infty}); \angle (\alpha_n,\beta_n)\bigr) \leq \frac{2
(t_n^{\infty})^{-\pi/(\beta_n-\alpha_n)}}{\pi\bigl(1-a^{\pi/(\beta_n-\alpha_n)}\bigr)^2}\, \Im\, (z_n)^{\pi/(\beta_n-\alpha_n)}\\ \leq
\frac{2}{\pi \bigl(1-\sqrt{a}\,\bigr)^2} \left(\frac{|z_n|}{t_n^{\infty}}\right)^{\pi/(\beta_n-\alpha_n)} \overset{\eqref{aznd}}{\leq}
\frac{2}{\pi \bigl(1-\sqrt{a}\,\bigr)^2} \, a^{\pi/(\beta_n-\alpha_n)} \underset{n\to \infty}{\longrightarrow} 0 \,. \end{multline}
Из соотношений \eqref{Dt} и \eqref{Dti} согласно \eqref{lm_omaD} следует неравенство $$\lim\limits_{n\to
\infty}\bigl|f(z_0)-F(z_n)\bigr|\leq d.$$ Таким образом, в силу произвола в выборе числа $d>0$ и точки $z_0\in \CC$ получаем
$\lim\limits_{n\to \infty} F(z_n)=f(z_0)$ для всех $z_0\in \CC$. Условия \eqref{1mu}--\eqref{2mu} теоремы о повторных интегралах в
случае, когда в роли мер $\tau_z$ выступают выметания $\delta_z^{\bal}$ мер Дирака достаточно очевидны, что вместе с \eqref{deii} по
теореме о повторных интегралах доказывает \eqref{c:del} из \eqref{deiv}.

Наконец, равенства \eqref{eqomd} п.~\eqref{deiii} сразу следуют из равенства \eqref{1balo} п.~\eqref{dei} по определению
\ref{df:clbala} после интегрирования обеих частей равенства \eqref{1balo} по верхней и нижней вариациям $\nu^{\pm}$. \end{proof}

\section{Классическое выметание субгармонической функции\\ на систему лучей}\label{kvsf} \setcounter{equation}{0} 
\subsection{Субгармонические ядра и глобальное представление Рисса}\label{SkRd}
Пусть $B$ --- борелевское подмножество в области $\Omega$ и $\nu \in \mc M^+(\Omega)$. По определению $L^1(B,\dd \nu)$ --- множество
всех функций $g\colon B \to \RR_{\pm \infty}$, интегрируемых по мере $\nu\bigm|_B$, т.\,е. таких, что $\int_B|g| \dd \nu <+\infty$.

\begin{definition}[{\cite[определение 2]{Kh07}}]\label{df:subk}Пусть $B$ --- борелевское подмножество в $\Omega$, а борелевская
функция $h \colon B\times \Omega \to \RR$ локально ограничена и для каждой фиксированной точки $\zeta \in B$ функция $h (\zeta ,\cdot
)\colon \Omega \to \RR$ {\it гармоническая на\/} $\Omega$. Тогда функцию \begin{equation}\label{rep:subk} k \colon (\zeta , z) \,
\longmapsto \, \log |\zeta -z|+h (\zeta , z), \quad (\zeta , z)\in B\times \Omega, \end{equation} будем называть {\it субгармоническим
ядром на\/} $B\times \Omega$ ({\it с гармонической компонентой\/ $h$ и несущим множеством\/ $B$}). \end{definition}

\begin{definition}[{\cite[определение 3]{Kh07}}]\label{df:subks}Пусть мера $\nu \in \mc M^+(\Omega)$ {\it сосредоточена на борелевском
подмножестве\/} $B\subset \Omega$. Субгармоническое ядро $k$ на $B\times \Omega$ называем {\it подходящим для\/} $\nu$, если для
каждой точки $z\in \Omega$ найдутся подобласть $D_z\Subset \Omega$, содержащая точку $z$, и функция $g_z \in L^1\bigl((\Omega
\setminus D_z)\cap B, \dd \nu \bigr)$, для которых
\begin{equation}\label{dg:nuk} \sup_{w\in D_z}\bigl|
k(\zeta , w )\bigr|\leq g_z(\zeta) \quad \text{\it при всех\/ \; $\zeta \in (\Omega \setminus D_z)\cap B$}. \end{equation}
\end{definition}

Следующее предложение представляет собой глобальную версию классической теоремы Рисса о локальном представлении субгармонической
функции в виде суммы логарифмического потенциала ее меры Рисса и некоторой гармонической функции \cite[Теорема~3.9]{HK},
\cite[Теорема~3.7.9]{Rans}. \begin{propos}[{\cite[предложение 3.11]{Kh07}}]\label{pr:Riesz} Пусть мера\/ $\nu \in \mc M^+(\Omega)$
сосредоточена на борелевском подмножестве\/ $B\subset \Omega$ и субгармоническое ядро\/ $k$ на\/ $B\times \Omega$ --- подходящее для
меры\/ $\nu$. Тогда интеграл 
\begin{equation}\label{rep:glpot} 
U^{\nu}_k(z):= \int_B k (\zeta , z) \dd \nu( \zeta ), \quad z\in \Omega, 
\end{equation} 
определяет субгармоническую на\/ $\Omega$ функцию с мерой Рисса\/ $\nu$, а для каждой функции $v\in \sbh
(\Omega)$ с мерой Рисса\/ $\nu_v=\nu$ имеет место представление 
\begin{equation}\label{repr:Riesz}
 v= U_k^{\nu}+H \quad \text{на \;
$\Omega$, \; где $H\in \Har (\Omega)$}. 
\end{equation} 
\end{propos}

В случае произвольной субгармонической в $\Omega:= \CC$ функции $v\neq \boldsymbol {-\infty}$ с мерой Рисса $\nu:=\nu_v$ конструкция
одного из возможных вариантов подходящего для меры $\nu$ субгармонического ядра давно известно как ядро Вейерштрасса \cite[4.1]{HK}.
Для согласованности с обозначениями из \cite{HK} введем его в той же форме, что и в \cite[4.1]{HK} (в наших записях переменная $z$
соответствует переменной $x$, а $\zeta$ --- соответственно $\xi$ или $\zeta$ в \cite[4.1]{HK} и в некоторых  выражениях переменные переставлены местами). Положим 
\begin{equation}\label{Ko} 
K(z):=\log |z|, \quad K(\zeta-z)=\sum_{j=0}^\infty a_j(\zeta, z)= \Re \Bigl\{\log
|\zeta|-\sum_{j=1}^{\infty}\frac1{j}\Bigl(\frac{z}{\zeta}\Bigr)^j\Bigr\} \end{equation} --- разложение гармонической по $z$ в $\CC$,
исключая точку $z=\zeta$, в степенной ряд по переменным $z,\bar z$, где при фиксированном $j$ и $\zeta\neq 0$ через $a_j(\zeta, z)$,
$j\in \NN$, обозначены одночлены от $z,\bar z$ степени $j$. При этом одночлены $a_j(\zeta, z)$, $j\in \NN$, гармоничны по $z$ при
фиксированном $\zeta$ и непрерывны по совокупности переменных $z,\zeta$ при $|\zeta|\neq 0$ \cite[лемма 4.1]{HK}. Наряду с \eqref{Ko}
рассмотрим <<урезанные>> ядра \cite[(4.1.2)]{HK} 
\begin{equation}\label{Kqz} K_q(\zeta,z):= K(\zeta-z)-\sum_{j=0}^q a_j(\zeta, z),
\quad q\in \{-1\}\cup \NN_0, \quad  \sum_{j=0}^{-1}\dots:=0,
\end{equation} 
--- субгармонические ядра, которые в \cite[\S~2, 2.1, примеры, 3(${\mathrm E}_q$)]{Kh07} названы ядрами Вейерштрасса\,--\,Адамара рода $q$.  Для  меры $\nu\in \mc M^+(\CC)$ всегда  можно  подобрать положительную возрастающую функцию $q\colon \RR^+\to \{-1\}\cup \NN_0$, 
непрерывную справа, тождественно равную $-1$ на  $[0,1)$, для которой 
\begin{equation}\label{conver} \int_{0}^{\infty} \Bigl(\frac{x_0}{t}\Bigr)^{q(t)+1} \dd
\nu^{\rad}(t)<+\infty \quad\text{при любом $x_0\in \RR_*^+$}. 
\end{equation} 
Пусть $(r_n)_{n\in \NN_0}\subset \RR_*^+$, $r_0=0$, $r_1\geq 1$, --- возрастающая последовательность всех точек, в которых происходит скачок функции $q$. Положим $q_n:=q(r_n)$, $n\in \NN_0$. Субгармоническое {\it ядро Вейерштрасса\/} на  $\CC \times \CC$  относительно этих двух последовательностей  $(r_n)_{n\in \NN_0}$ и $(q_n)_{n\in \NN_0}$ --- это функция
\cite[(4.1.10)]{HK}, \cite[\S~2, 2.1, примеры, 3(${\mathrm W}_{\infty}$)]{Kh07}
\begin{equation}\label{Kqzz}
K_{q(|\cdot|)}(\zeta,z)\overset{\eqref{Kqz}}{:=}
K_{q_n}(\zeta , z)\quad \text{при  $r_{n}\leq |\zeta|< r_{n+1}$, $n\in \NN_0$}.
\end{equation}
Из известных оценок ядер \eqref{Kqz} \cite[4.1.1]{HK}  с помощью   интегрирования по частям интеграла \eqref{conver} легко установить, что  $K_{q(|\cdot|)}$ из \eqref{Kqzz} --- подходящее для меры $\nu$ субгармоническое ядро.  В этом случае глобальное представление Рисса \eqref{repr:Riesz} из предложения \ref{pr:Riesz} задаётся через интеграл
 \begin{equation}\label{rep:glpotU} 
 U^{\nu}_{K_{q(|\cdot|)}}(z)\overset{\eqref{rep:glpot}}{:=} \int_\CC K_{q(|\cdot|)} (\zeta , z) \dd \nu( \zeta ),
\quad z\in \CC, 
\end{equation} 
представлением Вейерштрасса \cite[теорема 4.1]{HK} с
$H\in \Har (\CC)$ в виде 
\begin{equation}\label{reprW} 
v(z)\overset{\eqref{repr:Riesz}}{=}U^{\nu}_{K_{q(|\cdot|)}}(z)
+H(z)= \int_{\CC} K_{q(|\cdot|)}(\zeta, z) \dd \nu_v(\zeta)+H(z), \quad z\in \CC. 
\end{equation} 

\begin{remark}\label{rempod} Субгармоническое ядро $K_{q(|\cdot|)}$ из \eqref{Kqzz}, подходящее для меры $\nu \in \mathcal{M}^+ (\CC)$, строится на основе единственного условия  \eqref{conver}. При этом для любой  положительной возрастающей непрерывной справа функции $\check q\colon \RR^+\to \{-1\}\cup \NN_0$, тождественно равной $-1$ на  $[0,1)$, из условия мажорирования $\check q\geq q$ на $\RR^+$ сразу следует, что субгармоническое ядро $K_{\check q(|\cdot|)}$ также подходящее для той же меры $\nu \in \mathcal{M}^+ (\CC)$.
\end{remark}

\subsection{Выметание субгармонической функции}\label{vsbf} \begin{definition}[{\cite[определение 2.1]{Kha91}, \cite[определение
0.1]{KhDD92}}] Пусть $v$ --- субгармоническая функция на $\CC$. {\it Выметанием функции $v$ на замкнутую систему лучей\/} $S$ называем функцию $v_S^{\bal}\in \sbh (\CC)$, гармоническую в $\CC \setminus S$ с сужением $v_S^{\bal}\bigm|_S=v\bigm|_S$, т.\,е. $v_S^{\bal}(z)=v(z)$ для всех $z\in S$. Выметание $v_S^{\bal}$ называем также {\it выметанием из дополнительных углов\/} $\CC\setminus S$. 
\end{definition} 
\begin{theorem}\label{th:bs} Пусть для меры $\nu\in \mathcal M^+ (\CC)$ выполнено условие Бляшке вне системы лучей $S$, т.\,е. существует выметание $\nu_S^{\bal}\in \mathcal M^+ (\CC)$ {\rm (теорема \ref{th:Sb})}. Тогда для любой функции $v\in \sbh_*(\CC)$ с мерой Рисса $\nu$ существует выметание $v_S^{\bal}$ на $S$ с мерой Рисса $\nu_S^{\bal}$. 
\end{theorem}
\begin{proof} По замечанию \ref{rempod} всегда можно выбрать функцию $q$ требуемого в 
\eqref{conver} вида так, что субгармоническое ядро $K_{q(|\cdot|)}$ подходящее для обеих мер $\nu$ и  $\nu^{\bal}_S$. С целью применить теорему \ref{th:IB}\eqref{deiv} мы будем изменять ядро  $K_{q(|\cdot|)}$, ликвидируя в нем логарифмические особенности. Сначала  при заданном числе $s\in (0,1)$ преобразуем  ядра $K_q(z,\zeta)$ рода $q$ из \eqref{Kqz}. Для этого  присутствующее в ядре \eqref{Kqz} логарифмическое выражение $K(\zeta-z)=\log |\zeta-z|$ заменим  на непрерывную функцию  $\max\bigl\{\log |\zeta-z|, \log s\bigr\}$. Полученную таким образом из ядра \eqref{Kqz} рода $q$ новую функцию обозначим 
через $K_q^s$. Рассмотрим теперь убывающую последовательность $\sigma=(s_n)_{n\in \NN_0}$, где числа $s_n\in (0,1)$, и положим 
$K_{q(|\cdot|)}^{\sigma} (\zeta , z)\overset{\eqref{Kqzz}}{:=}
K_{q_n}^{s_n} (\zeta , z)$ при $r_n\leq |\zeta|<r_{n+1}$, $n\in \NN_0$, $z\in \CC$. Если числа $s_n\in (0,1)$ выбрать достаточно малыми, а функцию $q\colon \RR^+\to \{-1\}\cup \NN_0$, сохраняя непрерывность справа и тождественное равенство $-1$ на  $[0,1)$, при необходимости, достаточно увеличить, то при каждом фиксированном $z\in S$ 
функция $F(\zeta)=K_{q(|\cdot|)}^{\sigma} (\zeta , z)$ интегрируема по мере  
$\nu_S^{\bal}$ на $\CC$. Применим теорему \ref{th:IB}\eqref{deiv} к такой $\nu_S^{\bal}$-интегрируемой функции $F$, используя в  \eqref{{c:del}a}  переменную $\zeta$ вместо переменной $z$:
\begin{equation*} \int_S K_{q(|\cdot|)}^{\sigma}(\zeta,z) \dd \nu^{\bal}_S(\zeta) \overset{\eqref{c:del}}{=}\int_{\CC}\biggl(\,
\int_S K_{q(|\cdot|)}^{\sigma}(\xi,z) \dd \delta_{\zeta}^{\bal}(\xi)\biggl )\dd \nu(\zeta) \quad\text{при $z\in S$}, 
\end{equation*} 
где $\delta_{\zeta}^{\bal}$ --- выметание меры Дирака $\delta_{\zeta}^{\bal}$ на $S$. Из этого равенства  с учетом замечания \ref{rempod}, устремляя одновременно и монотонно все члены $s_n\in (0,1)$ последовательности $\sigma$ к нулю, по теореме о сходимости интегралов от монотонно убывающей последовательности функций получаем  такое же равенство уже без верхних индексов $\sigma$ в подынтегральных функциях:  
\begin{multline*} \int_S K_{q(|\cdot|)} (\zeta,z) \dd \nu^{\bal}_S(\zeta) \overset{\eqref{c:del}}{=}\int_{\CC}\biggl(\,
\int_S K_{q(|\cdot|)} (\xi,z) \dd \delta_{\zeta}^{\bal}(\xi)\biggl )\dd \nu(\zeta) 
\\
\overset{\eqref{{c:del}b}}{=}\int_{\CC}  K_{q(|\cdot|)} (\zeta,z) \dd \nu(\zeta) \overset{\eqref{rep:glpotU}}{=}   U^{\nu}_{K_{q(|\cdot|)}}(z) \quad\text{при $z\in S$}, 
\end{multline*} 
где предпоследнее равенство с интегралом Пуассона во внутреннем интеграле  следует из гармоничности ядра  $K_{q(|\cdot|)} (\cdot ,z)$ по первой переменной в дополнительных к $S$ углах при $z\in S$. Отсюда в качестве выметания на систему лучей $S$ субгармонической функции $U^{\nu}_{K_{q(|\cdot|)}}(z)$, $z\in \CC$, с мерой Рисса $\nu$ можем выбрать функцию 
\begin{equation*} 
\bigl(U^{\nu}_{K_{q(|\cdot|)}}\bigr)^{\bal}(z):=\int_S K_{q(|\cdot|)}(\zeta,z) \dd \nu^{\bal}(\zeta),
\; z\in \CC,\quad\text{с мерой Рисса $\nu^{\bal}_S$}. 
\end{equation*}
Поскольку произвольная функция $v\in \sbh_*(\CC)$ с мерой Рисса $\nu$
представляется в виде $v\overset{\eqref{reprW}}{=}U^{\nu}_{K_{q(|\cdot|)}} +H$ с $H\in \Har (\CC)$, в качестве ее выметания на $S $ можно выбрать функцию 
$v^{\bal}_S:=\bigl(U^{\nu}_{K_{q(|\cdot|)}}\bigr)^{\bal}+H$
с мерой Рисса $\nu^{\bal}_S$. \end{proof}

Приведенное здесь доказательство теоремы \ref{th:bs} технически упрощает доказательство \cite[основная теорема]{Kha91}, где был рассмотрен лишь случай функций конечного порядка и позволяет не исключать для системы $S$ из одного луча как особый незавершённый случай функций и мер порядка $p\geq 1$. Теперь основной результат \cite[основная теорема]{Kha91} можно сформулировать в законченной
форме: \begin{theorem}\label{mthfo} Пусть $p\in \RR^+_*$, $v\in \sbh_*(\CC)$ с $\type_p^{\infty}[v]<+\infty$ и мерой Рисса $\nu_v$, удовлетворяющей условию Бляшке вне системы лучей $S$, а $\nu_v^{\bal}$ --- выметание меры $\nu$ на $S$. Тогда существует выметание $v^{\bal}\in \sbh_*(\CC)$ на $S$ с
$\type_p^{\infty}[v^{\bal}]<+\infty$ и мерой Рисса $\nu^{\bal}$ с $\type_p^{\infty}[\nu^{\bal}]<+\infty$, удовлетворяющей при целом
$p\in \NN$ условию Линделёфа рода $p$.
\end{theorem}

\begin{remark}\label{b:amv} Если воспользоваться альтернативным подходом к выметанию меры на систему лучей из замечания \ref{b:am}, то
можно получить все результаты п.~\ref{vsbf} с помощью последовательности счётного числа выметаний $v_n^{\bal}$ функции $v\in
\sbh_*(\CC)$ из объединений секторов $D(r_n)\cap (\CC\setminus S)$ с последующим предельным переходом в пространстве распределений
$\mathfrak D'(\CC)$ и в $L_1^{\loc}(\CC)$ для некоторой сходящейся в $\mathfrak D'(\CC)$ и в $L_1^{\loc}(\CC)$ подпоследовательности
из $(v_n^{\bal})$, выбираемой из соображений предкомпактности локально равномерно ограниченной сверху последовательности
$(v_n^{\bal})$ \cite[2.7]{Az}, \cite[теоремы 3.2.12, 3.2.13]{Ho}.
	
\end{remark}	

\subsection{Условие Бляшке в терминах субгармонической функции}  Ус\-л\-овие Бляшке вне $S$, ключевое для теорем \ref{th:bs} и \ref{mthfo}, формулируется в терминах меры Рисса $\nu_v$ функции $v\in \sbh(\CC)$. Естественно рассмотреть возможность его формулировки в терминах функции $v$ без привлечения её меры Рисса $\nu_v$. Важную роль в этом сыграет интеграл 
\begin{equation}\label{fK:abp+}
J_{\alpha,\beta}(r_0,r;v)
:=\int_{r_0}^r\frac{v(te^{i\alpha})+v(te^{i\beta})}{t^{\frac{\pi}{\beta-\alpha}+1}} \dd t, \quad 
 \begin{cases}
0<r_0<r<+\infty, \\
 \alpha <\beta\leq 2\pi+\alpha. 
\end{cases}
\end{equation}

\begin{propos}\label{LKcr} Мера Рисса $\nu_v$ функции $v\in \sbh_*(\CC)$
	удовлетворяет условию Бляшке вне системы лучей $S$, если и только
если для каждого дополнительного к $S$ угла $\angle  (\alpha ,\beta)$ произвольного раствора $\beta-\alpha \leq 2\pi$ при некотором (любом)\/ $r_0\in \RR^+_*$ выполнено соотношение
\begin{subequations}\label{fK:ab} 
	\begin{align} 
&A_{\alpha,\beta}(r_0,r,v)+B_{\alpha,\beta}(r_0,r,v)\leq O(1) \quad\text{при
$r\to+\infty$, где} \tag{\ref{fK:ab}a}\label{fK:aba} \\
&A_{\alpha,\beta}(r_0,r,v):=\frac{1}{2(\beta-\alpha)}\int_{r_0}^r\biggl(\,\frac{1}{t^{\frac{\pi}{\beta-\alpha}}}
-\frac{t^{\frac{\pi}{\beta-\alpha}}}{r^{\frac{2\pi}{\beta-\alpha}}}\biggr)
\bigl(v(te^{i\alpha})+v(te^{i\beta})\bigr)
\frac{\dd t}{t} \tag{\ref{fK:ab}b}\label{fK:abb} \\
&\overset{\eqref{fK:abp+}}{=}\frac{1}{2(\beta-\alpha)}\biggl(J_{\alpha,\beta}(r_0,r;v)
-\frac{1}{r^{\frac{2\pi}{\beta-\alpha}}}
\int_{r_0}^r \frac{v(te^{i\alpha})+v(te^{i\beta})}{t^{1-\frac{\pi}{\beta-\alpha}}}
\dd t\biggr)
\tag{\ref{fK:ab}c}\label{fK:abb+}
 \\
&\overset{\eqref{fK:abp+}}{=}\frac{\pi}{(\beta-\alpha)^2r^{\frac{2\pi}{\beta-\alpha}}}\int_{r_0}^r
J_{\alpha,\beta}(r_0,t;v)\,t^{\frac{2\pi}{\beta-\alpha}-1}
\dd t, \tag{\ref{fK:ab}d}\label{fK:abd}
\\
&B_{\alpha,\beta}(r_0,r,v):=\frac{1}{(\beta-\alpha)r^{\frac{\pi}{\beta-\alpha}}}\int_{\alpha}^{\beta} v(re^{i\theta})\sin\frac{\pi (\theta-\alpha)}{\beta-\alpha} \dd \theta. 
\tag{\ref{fK:ab}e}\label{fK:abc} 
\end{align}
\end{subequations} 

Пусть $p\in \RR^+_*$. Если  в обозначении\/ \eqref{df:MCB} выполнено соотношение
\begin{equation}\label{afr}
\limsup_{r\to +\infty} \frac{1}{r^p} \, C(r;v) <+\infty, 
\text{ в частности, если $\type_p^{\infty}[v]<+\infty$,}
\end{equation} 
то мера Рисса $\nu_v$ удовлетворяет условию Бляшке 
в  любом угле $\angle  (\alpha ,\beta)$  раствора $\beta-\alpha <\pi/p$, где $\beta-\alpha \leq 2\pi$.

Наконец, если $p\geq 1/2$ и $\type_p^{\infty}[v]<+\infty$, 
 то условие  Бляшке в угле $\angle  (\alpha ,\alpha+\pi/p)$ раствора $\pi/p$ для меры Рисса $\nu_v$ эквивалентно  условию 
\begin{equation}\label{fK:abp}
J_{\alpha,\alpha+\pi / p}(r_0,r;v) \overset{\eqref{fK:abp+}}{=}\int_{r_0}^r\frac{v(te^{i\alpha})+v(te^{i(\alpha+\pi/p)})}{t^{p+1}} \dd t
\underset{r\to+\infty}{\leq} O(1).
\end{equation} 

\end{propos}
\begin{proof} 
Представление \eqref{fK:abb+} для $A_{\alpha,\beta}(r_0,r,v)$ сразу следует из определений 
 \eqref{fK:abb} и \eqref{fK:abp+}. 
Перейти от \eqref{fK:abb} к \eqref{fK:abd} можно аналогично \cite[(5.07)]{Levin56}
интегрированием по частям, исходя из  равенств
\begin{multline*}
A_{\alpha,\beta}(r_0,r,v)\overset{\eqref{fK:abb}}{=}
\frac{1}{2(\beta-\alpha)}\int_{r_0}^r \biggl( 1
-\Bigl(\frac{t}{r}\Bigr)^{\frac{2\pi}{\beta-\alpha}}\biggr) \dd
\int_{r_0}^t\frac{v(se^{i\alpha})+v(se^{i\beta})}{t^{\frac{\pi}{\beta-\alpha}+1}} \dd s
\\
=\frac{1}{2(\beta-\alpha)}
\frac{2\pi}{\beta-\alpha}
\int_{r_0}^r 
\biggl(\int_{r_0}^t\frac{v(se^{i\alpha})+v(se^{i\beta})}{s^{\frac{\pi}{\beta-\alpha}+1}} \dd s\biggr)
\frac{t^{\frac{2\pi}{\beta-\alpha}-1}}{r^{\frac{2\pi}{\beta-\alpha}}} \dd t,
\end{multline*}
где внутренний интеграл  равен $J_{\alpha,\beta}(r_0,t; v)$, $t\geq r_0$.
	
После редукции угла $\angle (\alpha,\beta)$ к  верхней полуплоскости достаточно  доказать \eqref{fK:aba} в случае $\alpha=0$ и $\beta=\pi$ для функции, субгармонической в открытой окрестности множества 
$\CC^{\overline \up}_*$ с возможной  особенностью в нуле. Используем  субгармоническую
версию  формулы Карлемана \cite[гл.~V, \S~1, теорема 1]{Levin56}, \cite[теорема 3.1]{GO}, \cite[лемма 2]{MOS}.
\begin{lemma}\label{lemfK} Пусть $0<r_0<r<+\infty$. 
Для полунепрерывной  сверху в $\CC^{\overline \up}$  функции $v\in \sbh_* \bigl(\CC^{\overline{\up}}\cap \overline D_*(r)\bigr)$,  $\lambda_{\RR}$-интегрируемой  в окрестности нуля и с мерой Рисса $\nu_v\in \mathcal M\bigl(\CC^{\overline{\up}}\cap \overline D(r)\bigr)$, справедлива формула
\begin{subequations}\label{fKmy}
\begin{align}
\int_{r_0<|z|\leq r} \Im \Bigl(\frac{1}{\bar z} -\frac{z}{r^2}\Bigr) \dd \nu_v(z)
+&\Bigl(\frac{1}{r_0^2}-\frac{1}{r^2}\Bigr) \int_{0<|z|\leq r_0} \Im z \dd \nu_v(z)
\tag{\ref{fKmy}a}\label{fKmya}
\\
=A_{0,\pi}(r_0,r,v)+&B_{0,\pi}(r_0,r,v)
\tag{\ref{fKmy}b}\label{fKmyb}
\\
+\Bigl(\frac{1}{r_0^2}-\frac{1}{r^2}\Bigr) \int_0^{r_0} \bigl(v(-t)+v(t)\bigr) \dd t
&-\frac{1}{\pi r_0} \int_{0}^{\pi} v(r_0e^{i\theta})\sin \theta \dd \theta .
\tag{\ref{fKmy}c}\label{fKmyc}
\end{align} 
\end{subequations}	
\end{lemma}	
\begin{proof}[К доказательству леммы \ref{lemfK}] В \cite[теорема 1]{Kha88} формула \eqref{fKmy} доказана для функций $v=\log |f|$ с $f\in \Hol \bigl(\CC^{\overline \up}\cap \overline D(r)\bigr)$. Для  дважды непрерывно дифференцируемой субгармонической в $\CC^{\overline{\up}}\cap \overline D_*(r)$ функции $v$ 
она получается так же. Некоторые необходимые технические процедуры, связанные с возможной
особенностью в нуле, опускаем.  Для перехода к функции $v$ из леммы \eqref{lemfK}  можно использовать стандартное приближение  её убывающей последовательностью дважды непрерывно  дифференцируемых субгармонических функций,  меры Рисса которых ввиду убывания последовательности к $v$
$*$-слабо сходятся  \cite[{\bf A.4}]{Rans} к мере Рисса $\nu_v$ \cite{Az}, \cite{Ho}.
\end{proof}
Если мера $\nu_v$ удовлетворяет условию Бляшке в  $\CC^{\up}$, то строка 
\eqref{fKmya}$\,\leq O(1)$ и строка \eqref{fKmyc}$\,=O(1)$ при $r\to +\infty$, откуда   для средней строки
 в \eqref{fKmy} имеем \eqref{fKmyb}$\,\leq O(1)$ при $r\to +\infty$, что даёт  \eqref{fK:aba}. Обратно, 
 из последнего соотношения для средней строки, т.\,е. из  \eqref{fK:aba}, следует, что, ввиду 
 \eqref{fKmyc}$\,=O(1)$ при $r\to +\infty$, выполнено соотношение  \eqref{fKmya}$\,\leq O(1)$ при $r\to +\infty$. Отсюда 
 \begin{equation*}
 \frac{3}{4}\int_{r_0<|z|\leq r/2}\Im \frac{1}{\bar z} \dd \nu_v(z)
\leq \int_{r_0<|z|\leq r} \Im \Bigl(\frac{1}{\bar z} -\frac{z}{r^2}\Bigr) \dd \nu_v(z)
= O(1)
 \end{equation*} 	
 при $r\to +\infty$   и мера $\nu_v$ удовлетворяет условию Бляшке в $\CC^{\up}$.

Если 	выполнено \eqref{afr}, то $\type_p^{\infty}[\nu_v]<+\infty$.  По предложению 
\ref{pr:fcc}\eqref{i:i0} и предложению \ref{pr:2st}, переформулированному для угла $\angle (\alpha,\beta)$, 
 для любого угла $\angle (\alpha,\beta)$  раствора $\beta-\alpha <\pi/p$ условие Бляшке в $\angle (\alpha,\beta)$ выполнено для $\nu_v$. Это завершает рассмотрение случая с условием 
 \eqref{afr}.

Наконец, если $\beta-\alpha=\pi/p$ и $\type_p^{\infty}[v]<+\infty$, воспользуемся доказанным соотношением 
\eqref{fK:aba} для случая $S=\CC\setminus \angle (\alpha, \beta)$, где, после поворота  $\CC$ на угол $-\alpha$, можно полагать $\alpha:=0$. При этом   \eqref{fK:aba}, эквивалентное
условию Бляшке для $\nu_v$ вне  $S$, с представлением \eqref{fK:abb+} запишется в виде
\begin{multline}\label{ppff}
\frac{p}{2\pi}J_{0,\pi/p}(r_0,r;v)
-\frac{p}{2\pi r^{2p}}\int_{r_0}^r\bigl(v(t)+v(te^{i\pi/p})\bigr) t^{p-1} \dd t
\\+B_{0,\pi/p}(r_0,r,v) \leq  O(1) \quad\text{при $r\to +\infty$}.
\end{multline}
Избавиться от последнего слагаемого в левой части \eqref{ppff} позволяет 
\begin{lemma}\label{lemCB} Пусть $0<r_0<r\in \RR_*^+$ и $v\in \sbh_*\bigl(\overline D(r)\bigr)$. Тогда 
	\begin{equation}\label{es:Cv}
	C(r; |v|)\leq 2C(r;v^+)-C(r_0;v)\leq 2C(r;v^+)-v(0).
	\end{equation}
	В частности, для  $v\in \sbh_*(\CC)$ при $\type_p^{\infty}[v]<+\infty$ 
	и $\beta-\alpha=\pi/p\leq 2\pi$
	\begin{equation}\label{forBab}
	B_{\alpha,\beta}(r_0,r,v)\overset{\eqref{fK:abc}}{=}B_{\alpha,\alpha+\pi/p}(r_0,r,v)	= O(1) \quad\text{при $r\to+\infty$} .
	\end{equation}
\end{lemma}   
\begin{proof}[Доказательство леммы \ref{lemCB}] Из очевидного равенства $|v|=2v^+-v$ имеем 
$C(r; |v|)= 2C(r;v^+)-C(r;v)$ и неравенства 	\eqref{es:Cv} следуют из возрастания интегральных средних  $C(\cdot; v)$ функции $v\in \sbh_*\bigl(\overline D(r)\bigr)$ на $[0,r]$ от $v(0)$ до $C(r;v)\neq -\infty$. По определению \eqref{fK:abc}  при $\beta-\alpha=\pi/p\leq 2\pi$ 
\begin{multline*}
\bigl|B_{\alpha,\beta}(r_0,r,v)\bigr|\overset{\eqref{fK:abc}}{\leq}\frac{p}{\pi r^p}\int_{\alpha}^{\beta} \bigl|v(re^{i\theta})\bigr| \dd \theta
\overset{\eqref{df:MCB}}{\leq} \frac{2p}{r^p}C\bigl(r;|v|\bigr)
\\ 	\overset{\eqref{es:Cv}}{\leq} \frac{2p}{r^p} \bigl(2C(r;v^+)-C(r_0;v)\bigr)
= O(1) \quad\text{при $r\to+\infty$}б
\end{multline*}
поскольку $C(r;v^+)=O(r^p)$ при $\type_p^{\infty}[v]<+\infty$ и $C(r_0;v)\in \RR$.
\end{proof}	
По лемме \ref{lemCB} соотношение \eqref{ppff} эквивалентно соотношению 
\begin{equation}\label{ppff+}
J_{0,\pi/p}(r_0,r;v)-\frac{1}{ r^{2p}}\int_{r_0}^r\bigl(v(t)+v(te^{i\pi/p})\bigr) t^{p-1} \dd t
\leq  O(1) 
\end{equation}
при  $r\to +\infty$. Если мера $\nu_v$ удовлетворяет условию Бляшке в $\angle (0,\pi/p)$, то 
выполнено \eqref{ppff+}  и ввиду $\type_p^{\infty}[v]<+\infty$ получаем \eqref{fK:abp}. Обратно, если имеет место \eqref{fK:abp}, то из представления  \eqref{fK:abd} для $A_{\alpha,\beta}(r_0,r,v)$ и 
\eqref{fK:abb+} левая часть  соотношения \eqref{ppff+} равна 
\begin{equation*}
\frac{2p}{r^{2p}}\int_{r_0}^r J_{0,\pi/p}(r_0,t;v) t^{2p-1} \dd t\leq O(1)\quad\text{при $r\to+\infty$}.
\end{equation*}
Это означает, что выполнено  \eqref{ppff+} при  $r\to +\infty$, эквивалентное  \eqref{ppff}. Следовательно,
мера $\nu_v$ удовлетворяет условию Бляшке в $\angle (0,\pi/p)$.
\end{proof} 

По аналогии с целыми функциями класса A из п.~\ref{bfR}   примем
\begin{definition}\label{df:clA} Пусть $S$ --- замкнутая система лучей с началом в нуле.
	 Функция $v\in \sbh_*(\CC)$ принадлежит классу Ахиезера, или классу\/ {\rm A}, относительно системы лучей $S$, если выполнено одно  из двух эквивалентных друг другу в силу предложения 
\ref{LKcr}    условий:
\begin{enumerate}[{\rm (i)}]
\item\label{Aci} мера Рисса функции $v$ удовлетворяет условию Бляшке вне $S$; 
\item\label{Acii} функция $v$ для  каждого дополнительного к $S$ угла $\angle  (\alpha ,\beta)$ 
удовлетворяет условию \eqref{fK:aba} с соответствующими упрощениями \eqref{afr}--\eqref{fK:abp}.
\end{enumerate}
\end{definition}

\subsection{Применения классического выметания на систему лучей}\label{applA}

\subsubsection{Вполне регулярный рост на системе лучей}\label{prr} Напомним некоторые основные определения и понятия теории
субгармонических и/или целых функций вполне регулярного роста, а также правильно распределенной меры и/или последовательности нулей
\cite{Levin56}, \cite{Az}.

Пусть $p\in \RR_*^+$. Пусть $v\in \sbh_*(\CC)$ --- функция c $\type_p^{+\infty}[v]<+\infty $. Индикатор роста $h_v\colon \RR \to \RR$ функции $v$ при порядке $p$ определяется равенством 
$h_v(\theta):=\limsup\limits_{r\to +\infty}\frac{v(re^{i\theta})}{r^p}\, $. 
Функция $u$ {\it вполне регулярного роста,\/} или удовлетворяет условию {\it полной регулярности роста,\/} на луче
$l_{\theta}:=\{re^{i\theta}\colon r\in \RR^+\}$ при порядке $p$, если существует предел 
\begin{equation}\label{df:vpregr}
h_v(\theta)=\lim_{\substack{r\to +\infty\\r\notin E}} \frac{v(re^{i\theta})}{r^p}\, , \quad \text{где}\quad \lim_{t\to
\infty}\frac{\lambda_{\RR}\bigl(E\cap [-t,t]\bigr)}{t}=0, 
\end{equation} 
т.\,е.  исключительное  множество $E\subset \RR^+$ {\it нулевой относительной линейной меры.\/} Целая функция $f\neq 0$ с $\type_p^{\infty}[\log|f|]<+\infty$ вполне регулярного роста на луче $l_{\theta}$ при порядке $p$, если таковой является функция $v=\log|f|$. Функция вполне регулярного роста на
системе лучей $S$ при порядке $p$, если она конечного типа при порядке $p$ около бесконечности и вполне регулярного роста на каждом
луче из $S$.

Мера $\nu\in \mc M^+(\CC)$ с $\type_p^{\infty}[\nu]<+\infty$ {\it правильно распределена на $\CC$} при нецелом порядке $p$, если найдётся  не более чем счётное множество $N\subset \RR$, для которого   существует конечный предел {\it (угловая плотность)}
\begin{equation}\label{ugpl}
 \lim_{r\to +\infty} \frac{\nu\bigl(\,\overline D(r)\cap \angle [\alpha, \beta] \bigr)}{r^p} 
 \quad\text{ при  	всех  $\alpha,\beta\in \RR\setminus N$, $\alpha<\beta$}.
\end{equation} 
 При целом $p$ для правильного распределения  меры $\nu$ дополнительно требуется {\it существование конечного предела\/} (ср. с определением \ref{df:Lc})
\begin{equation}\label{upL} 
\lim_{r\to +\infty}\int\limits_{1\leq |z|\leq r} \frac{1}{z^p} \dd \nu(z),
\end{equation}
что строго жестче условия Линделёфа рода $p$ из \eqref{con:Lp}.

В \cite[теорема 4.1]{Kha91}, \cite[теорема 1.5.1]{KhDD92} дано основное приложение выметания на систему лучей. При этом накладывалось дополнительное ограничение: {\it для системы лучей $S$, состоящей из одного луча, предполагалось, что порядок\/ $p<1$.\/} По той же схеме, что и применявшаяся при доказательствах упомянутых
теорем, теорема \ref{mthfo} позволяет снять это ограничение. 
\begin{theorem}[{\rm \cite[теорема 4.1]{Kha91}, \cite[теорема
1.5.1]{KhDD92}}]\label{crgr} Пусть $p\in \RR^+_*$, $S$ --- замкнутая система лучей с общим началом в нуле  и $\nu_v$ --- мера Рисса функции $v\in \sbh_*(\CC)$ с $\type_p^{\infty}[v]<+\infty$, принадлежащей классу 
{\rm A} относительно $S$ в смысле определения\/ {\rm \ref{df:clA},} что в данном случае означает выполнение одного
 из следующих двух эквивалентных друг другу условий:
 в каждом дополнительном к $S$ угле $\angle (\alpha,\beta)$ раствора $\beta-\alpha \geq \pi/p$
  \begin{enumerate}[{\rm (i)}]
 	\item\label{Bli} мера $\nu_v$ удовлетворяет условию Бляшке \eqref{ineq:akhan};
\item\label{Blii}  выполнено условие \eqref{fK:aba} с упрощением \eqref{fK:abp} при $\beta-\alpha=\pi/p$.
 \end{enumerate}
Тогда следующие три утверждения попарно эквивалентны:
\begin{enumerate}[{\rm (a)}]
\item\label{vBla} функция $v$ вполне регулярного роста на $S$ при порядке $p$ и для каждого дополнительного к $S$ угла $\angle (\alpha,\beta)$ раствора $\beta-\alpha=\pi k/p$ с $k\in \NN$ для некоторого (любого) $r_0>0$ существует конечный предел
\begin{equation}\label{Govlim}
\lim_{r\to +\infty}\int_{r_0}^r \frac{v(te^{i\alpha})+(-1)^{k-1}v(te^{i\beta})}{t^{p+1}} \dd t \in \RR;
\end{equation}
\item\label{vBlb} выметание $\nu_v^{\bal}$ на $S$ --- правильно распределённая мера при порядке $p$;
\item\label{vBlc} выметание $v_S^{\bal}\in \sbh_*(\CC)$ --- функция вполне регулярного роста на всей комплексной плоскости $\CC$ при порядке $p$.
\end{enumerate}
\end{theorem} 

\subsubsection{Условия вполне регулярного роста на конечном числе лучей}\label{prr1vr} 
\begin{propos}\label{pr:nR+} Для $q\in \NN_0$ и  субгармонического ядра Вейерштрасса\,--\,Адамара $K_q$ рода $q$  из  \eqref{Ko}--\eqref{Kqz}  имеем 
\begin{equation}\label{Kqdiff}
\frac{\dd}{\dd t} K_q(z,t)=\Re \frac{z^{q+1}}{t^{q+1}(t-z)}\quad\text{при $t\in \RR^+_*$ и $z\neq t$.}
\end{equation}
Для  функции  $n\colon \RR^+\to \RR$  локально ограниченной вариации
с полной вариацией $\var [n] (t)$ на отрезке $[0,t]$, $t\in \RR^+$,
принадлежащей  классам сходимости при порядке $q+1$ около $\infty$ и при порядке $q$	 около нуля, найдётся исключительное  множество $E[n]\subset \RR^+$ нулевой лебеговой меры\footnote{более того, и нулевой логарифмической ёмкости как подмножество в $\CC$} на\/ $\RR^+$, вне которого сходится интеграл 
\begin{equation}\label{rep:nKq}
\int_0^{+\infty} K_q(z,t) \dd n(t)=\fint_0^{+\infty} 
\Re \frac{z^{q+1}}{t^{q+1}(z-t)} \, n(t) \dd t, \quad z\in \CC \setminus E[n],
\end{equation}
где при $z\notin \RR^+_*$ перечёркнутый интеграл $\fint$ --- обычный интеграл Стильтьеса,
а при $z\in \RR^+_*$, как и в \eqref{Enln}, он равен  главному значению в смысле Коши 
в точке $z$. При этом  можно выбрать исключительное множество
\begin{subequations}\label{c:0l}
\begin{align}
 E[n]:&=\RR^+_*\setminus F[n], \quad\text{где в обозначении}
 \tag{\ref{c:0l}a}\label{c:0la}
 \\
&\var[n](x,t)\overset{\eqref{df:nup}}{:=}\var[n](x+t)-\var[n](x-t)
 \tag{\ref{c:0l}b}\label{c:0lb}
 \\
F[n]:&=\left\{x\in \RR^+_*\colon  \int_0^{r_x}\frac{\var[n](x,t)}{t}\dd t<+\infty, \quad 
0<r_x\leq x\right\}
\tag{\ref{c:0l}c}\label{c:0lc}
\\
&=\left\{x\in \RR^+_*\colon  \int_0^{r_x}\log t\dd\, \var[n](x,t) >-\infty \quad 0<r_x\leq x \right\}.
\tag{\ref{c:0l}d}\label{c:0ld}
\end{align}
\end{subequations} 
 В случае возрастающей функции $n$ равенство из \eqref{rep:nKq} имеет место для всех $z\in \CC$, если для интегралов допускать значение $-\infty$ как в\/ \cite[3.3]{HK}.
\end{propos}
\begin{proof} Равенство \eqref{Kqdiff} --- прямое вычисление. Из явного 
представления функции локально ограниченной вариации \cite[гл.~VIII, \S~3, теорема 6]{Nat} в виде разности двух возрастающих функций $\var[n]$ и $\var[n]-n$	следует, что  обе они принадлежат тем же классам сходимости, что и функция $n$. Таким образом, далее можно считать  функцию $n$  возрастающей на $\RR^+$. Тогда  при выборе $E[n]$ как в \eqref{c:0la} из конечности интегралов в \eqref{c:0lc}--\eqref{c:0ld} следует конечность интеграла в  \eqref{rep:nKq}, который представляет собой субгармоническую функцию \cite[4.2]{HK}, конечную во всех точках из $F[n]\cup (\CC\setminus \RR^+_*)$, и равную $-\infty$   на  $E[n]$ из \eqref{c:0la}. Такое исключительное множество $E[n]$ обладает требуемыми в предложении
\ref{pr:nR+} свойствами \cite[3.5]{Rans}. При этом  равенство \eqref{rep:nKq}	--- результат интегрирования по частям, которое возможно по предложениям \ref{pr:fcc}, \ref{pr:fcc0} 
и известным оценкам \cite[4.2]{HK} субгармонического ядра Вейерштрасса\,--\,Адамара $K_q$ рода $q$.

Равенство \eqref{rep:nKq} продолжается и на  $E[n]$ значениями $-\infty$ в случае возрастающей функции   $n$.
\end{proof}	
Из предложения \ref{pr:nR+} сразу вытекает
\begin{corollary}\label{rep:llj} Пусть $q\in \NN_0$ и   $v\in \sbh_*(\CC)$ с мерой Рисса $\nu_v$, сосредоточенной на конечном числе лучей 
\begin{equation}\label{lk1_k}
\begin{split}
l_j:=l(\theta_j)&:=\{te^{i\theta_j}\in \CC\colon t\in \RR^+\},  
\\ 
j=1, \dots, k\in &\NN, \quad \theta_1<\dots <\theta_k<\theta_1+2\pi.	
\end{split}
\end{equation}
Если каждая возрастающая  функция распределения 
\begin{equation}\label{nklj}
n_j(t):=\nu_v \bigl(l_j\cap \overline  D(t)\bigr), \quad t\in \RR^+,
\end{equation}
принадлежит классам сходимости  при порядке $q+1$ около $\infty$ и при порядке $q$	 около нуля, то для  некоторой функции $h\in \Har(\CC)$
\begin{equation}\label{rrvKq}
v(z)=\sum_{j=1}^{k} \fint_0^{+\infty}\Re \frac{z^{q+1}}{t^{q+1}(z-t)} \, n_j(t) \dd t+h(z)\quad\text{при всех $z\in \CC$.}
\end{equation}
\end{corollary}

\begin{theorem}\label{th10} Пусть $S\overset{\eqref{lk1_k}}{=}\{l_j\}_{j=1,\dots,k}$ и выполнены условия 
теоремы\/ {\rm \ref{crgr}}. Функция $v$ вполне регулярного роста на $S$ при порядке $p$, если  и только если существуют множества $E_j\subset \RR^+$, $j=1,\dots,k$,  нулевой относительной меры, для которых с функциями распределения $n_j$ 
из \eqref{nklj},  где  мера Рисса $\nu_v$  заменена на её выметание $(\nu_v)^{\bal}_S$, 
существуют  пределы
\begin{equation}\label{c:vprreg}
\lim_{\substack{r\to +\infty\\r\notin E_j}} {r^{[p]+1-p}}\sum_{j'=1}^{k} \fint_0^{+\infty}\Re \frac{e^{i([p]+1)\theta_{j}}}{t^{[p]+1}(re^{i\theta_j}-t)} \, n_{j'}(t) \dd t, \quad j=1,\dots,k.
\end{equation}
\end{theorem}
\begin{proof}  Для существующего в условиях  теоремы \ref{crgr}  по теореме \ref{mthfo}
	выметания $v_S^{\bal} \in \sbh(\CC)$
с $\type_p^{\infty}[v_S^{\bal}]<+\infty$	и с мерой Рисса $(\nu_v)_S^{\bal}$, для которой  $\type_p^{\infty}[(\nu_v)_S^{\bal}]<+\infty$,   по следствию 
\ref{rep:llj}   имеет место представление 
\begin{multline}\label{vrep}
v(z)\overset{\eqref{rrvKq}}{:=} \int_{\DD} \log|z-z'| \dd \nu_S^{\bal} (z')+
\\
\sum_{j=1}^{k} \fint_1^{+\infty}\Re \frac{z^{[p]+1}}{t^{[p]+1}(z-t)} \, n_j(t) \dd t+h(z), \quad z\in \CC.
\end{multline}	
где $h$ --- гармонический многочлен степени не выше $[p]$, а  для первого слагаемого-интеграла в правой части  \eqref{vrep} имеем соотношение $O\bigl(\log |z|\bigr)$ при $z\to \infty$. Остаётся обратиться к определению \eqref{df:vpregr} функции вполне регулярного роста на луче.  
\end{proof}

\subsubsection{Примеры условий вполне регулярного роста}\label{prr1} 

Два примера применения теорем  \ref{crgr} и \ref{th10} при $S:=\{ \RR^-, \RR^+\}$ дают соответственно утверждения теоремы A2 после \eqref{Enl}    в пп.~(iii)  и (iv)  с \eqref{Enln}, соответствующем 
\eqref{c:vprreg}. 

\begin{example}[{\rm \cite[пример 1.5.1]{KhDD92}}]\label{exgr2} Пусть $f$ --- целая функция экспоненциального типа с нулями на лучах
\begin{equation} l_k:=\Bigl\{ t\exp i\Bigl(\frac{\pi}{4}+\frac{\pi k}{2}\Bigr)\colon t\geq 0\Bigr\}, \quad k=0,1,2,3, \end{equation}
т.\,е. на биссектрисах четвертей координатной плоскости, $n_k(t)\geq 0$ --- функции распределения \eqref{nuR}, или считающие
радиальные функции, нулей, лежащих на лучах $l_k$, $t\geq 0$. Полагаем $n_4:= n_0$. Функция $f$ вполне регулярного роста при
порядке $1$ одновременно на вещественной и мнимой осях тогда и только тогда, когда существуют конечные пределы
\begin{subequations}\label{densn} \begin{align} \lim_{t\to+\infty} 2\int_0^{+\infty}\frac{n_k(s)+n_{k+1}(s)}{s^4+t^2}& \, s \dd s=:b_k,
\quad k=0,1,2,3, \tag{\ref{densn}d}\label{densnd}\\ \lim_{r\to+\infty}\int_1^r\int_0^{+\infty}\Bigl( \sum_{k=0}^3 i^{k+1}\bigl(
n_k(s)&+n_{k+1}(s)\bigr)\Bigr) \frac{s \dd s \dd t}{t(s^4+t^2)}. \tag{\ref{densn}L}\label{densnL} \end{align} \end{subequations} При
этом индикатор роста $h_{\log |f|}$ при порядке $1$ таков, что 
\begin{equation*} 
h_{\log |f|}\Bigl(\frac{\pi}{2}\Bigr)+h_{\log
|f|}\Bigl(-\frac{\pi}{2}\Bigr)=b_0+b_2, \quad h_{\log |f|}(0)+h_{\log |f|}(\pi)=b_1+b_3. 
\end{equation*} 
Существование четырёх пределов \eqref{densnd} соответствует существованию угловой плотности (ср. с \eqref{ugpl}), а существование конечного предела \eqref{densnL} ---
 существованию  конечного предела \eqref{upL} при  $p=1$.
\end{example} 
\subsubsection{Мультипликаторы и неполнота}\label{bAmp} 
Докажем  более общую версию  ут\-в\-е\-р\-ж\-д\-ения пп.~{\rm (v)} теоремы\/ {\rm A2}, которое следует из этой версии при $v=\log |f|\in \sbh_*(\CC)$ с целой функцией $f\neq 0$ класса A при условии \eqref{fsr}. 

\begin{theorem}\label{thmA} Пусть $v\in \sbh_*(\CC)$  с мерой Рисса $\nu:=\nu_v\in \mc M^+(\CC)$, $b\in \RR^+$, а также  выполнено одно из двух условий:
\begin{enumerate}[{\rm (i)}]
\item\label{muli}  интегральные средние $C(\cdot ;v)$ из \eqref{df:MCB} 
по окружностям определяют функцию  конечного типа $\leq b$ при порядке $1$, т.е.\/ {\rm (ср. с \eqref{afr})}
\begin{equation}\label{es:Is}
\limsup_{r\to +\infty} \frac{1}{r}\,C(r;v)=\limsup_{r\to +\infty}\frac{1}{2\pi r}\int_{0}^{2\pi} v(re^{i\theta}) \dd \theta \leq b\in \RR^+,
\end{equation}
и функция $v$ принадлежит классу\/ {\rm A} относительно системы двух лучей  $S:=\RR^-\cup \RR^+$
в смысле определения\/  {\rm \ref{df:clA}\eqref{Acii};} 
\item\label{mulii} при любом $\beta \in (0, \pi/4]$ для пары вертикальных углов 
\begin{equation}\label{ang_e}
U_{\beta}:=\bigl\{z\in \CC\colon |\arg z|\leq \beta, \quad |\arg z-\pi|\leq \beta\bigr\}
\end{equation}
и сужения $\nu\bigm|_{\CC\setminus U_{\beta}}$ меры $\nu$ на  $\CC\setminus U_{\beta}$ имеют место соотношения   
\begin{equation}\label{c:nuvtp}
\type_1^{\infty}\bigl[\nu\bigm|_{\CC\setminus U_{\beta}}\bigl]=0,\quad
\type_1^{\infty}[\nu]\leq eb\in \RR^+ .
\end{equation}
\end{enumerate}
Тогда для  любого числа $\e\in (0,1]$ найдётся  функция $g\in \Hol_*(\CC)$,  для которой 
индикатор роста $h_{v+\log|g|}$  функции $v+\log |g|\in \sbh_*(\CC)$  при порядке $1$  удовлетворяет
ограничениям  
\begin{equation}\label{hvge}
h_{v+\log|g|}(\theta)\leq1000\,b\bigl(\e|\cos \theta| +(1/\e)|\sin \theta| \bigr), \quad \theta \in [0,2\pi].
\end{equation}
\end{theorem}

\begin{proof} Условие \eqref{muli} влечёт за собой \eqref{mulii}. Действительно, из  
условия \eqref{es:Is}	по известному неравенству $\nu_v^{\rad}(r)\leq C(er;v)+\const_v$, $r\in \RR^+_*$, следует неравенство из \eqref{c:nuvtp}. Мера Рисса $\nu$ функции $v$ класса A относительно вещественной оси по предложению \ref{LKcr} и определению \ref{df:clA} удовлетворяет условию Бляшке в верхней и нижней полуплоскостях.  В частности,  сужение $\nu\bigm|_{\CC\setminus U_\beta}$ принадлежит классу сходимости 
при порядке $1$. По предложению \ref{pr:fcc}\eqref{i:i2}  мера $\nu\bigm|_{\CC\setminus U_\beta}$ нулевого типа при порядке $1$, что даёт равенство  из \eqref{c:nuvtp}. Далее доказательство продолжается 
при условии \eqref{mulii}. 

В случае $b=0$ мера Рисса $\nu$ нулевого типа при порядке $1$ и из представления Вейерштрасса\,--\,Адамара рода $1$ для функции
\begin{equation*}
v(z)=\int_{\DD}\log |\zeta-z| \dd \nu(\zeta)+\int_{\CC\setminus \DD}K_1(\zeta,z)\dd \nu(\zeta)
+H(z), \quad z\in \CC,
\end{equation*}
 где $H\in \Har(\CC)$, с субгармонической  функцией 
 \begin{equation*}
 u(z):=\int_{\CC\setminus \DD}K_1(z',-z)\dd \nu(z')-H(z), \quad z\in \CC,
 \end{equation*}  
  имеет место равенство 
 \begin{equation*}
 v(z)+u(z)=\int_{\DD}\log |\zeta-z| \dd \nu(z')+\int_{\CC\setminus \DD}\log\Bigl|1-\frac{z^2}{\zeta^2}\Bigr|
 \dd \nu(\zeta), \quad z\in \CC.
 \end{equation*}
 Ввиду $\type_1^{\infty}[\nu]=0$ в правой части последнего равенства  записана субгармоническая функция  нулевого типа при порядке $1$ и для любого $\e>0$  при всех $z\in \CC$ имеет место 
 оценка $ v(z)+u(z)\leq \e |z|+\const_{\e,\nu}$. Отсюда, усредняя обе части последней оценки по окружностям $\partial D(z,1)$, получаем 
 \begin{multline}\label{estuv}
 v(z)+\frac{1}{2\pi}\int_{0}^{2\pi} u(z+e^{i\theta}) \dd \theta
 \leq \frac{1}{2\pi}\int_{0}^{2\pi} v(z+e^{i\theta}) \dd \theta
\\+\frac{1}{2\pi}\int_{0}^{2\pi} u(z+e^{i\theta}) \dd \theta\leq 
 \e |z|+\const_{\e,\nu_v} \quad\text{при всех $z\in \CC$.}
 \end{multline}
 
 \begin{lemma}[{\rm \cite[следствие 2]{KhaBai16}}]\label{l:ufe}
 Пусть $u\in \sbh_*(\CC)$. Для любого числа $c\in \RR^+$ найдётся ненулевая целая функция $g\in \Hol_*(\CC)$, для которой
 \begin{equation*}
 \log \bigl| g(z)\bigr|\leq \frac{1}{2\pi}\int_{0}^{2\pi} u\bigl(z+(1+|z|)^{-c}e^{i\theta}\bigr) \dd \theta
 \quad\text{для всех $z\in \CC$.}
 \end{equation*}
\end{lemma}
 Применяя  лемму \ref{l:ufe} при $c=0$ к функции $u$ из \eqref{estuv}, получаем функцию 
$ v+\log |g|$ нулевого типа при порядке $1$, что даёт \eqref{hvge} при $b=0$. 

Рассмотрим  $b>0$.   С помощью выметания на $\RR\cup i\RR$ в \cite{Kha01l} была доказана   
	\begin{lemma}[{\rm часть \cite[основная теорема]{Kha01l}}]\label{l:vbpr} 
Пусть число $\beta \in (0, {\pi}/2)$, а $\mu$ --- мера конечной верхней плотности
${\Delta}_{\mu}>0$ (при порядке $1$) и 
\begin{equation}\label{ugoliR}
\supp \mu \subset \Bigl\{ z:\bigl|\arg z-\frac{\pi}2\bigr|\leq \beta \Bigr\}
\cup \Bigl\{ z:\bigl|\arg z+\frac{\pi}2\bigr|\leq \beta \Bigr\}\, .
\end{equation}
Тогда для любой субгармонической функции $u_{\mu}$ с мерой  Рис\-са $\mu$ найдётся такая целая функция $g\neq 0$, что  субгармоническая функция $u_{\mu}+\log |g|$ конечного типа при порядке $1$ и
для ее индикатора роста $h_{u_{\mu}+\log |g|}$ при всех $\theta  \in [-\pi , \pi ]$ выполнена оценка
\begin{equation}\label{ots_lu}
h_{u_{\mu}+\log |g|}(\theta ) <  \frac{12\pi (\pi+2) {\Delta}_{\mu}}
{{\pi}/2-\beta}\cdot \bigl(|\cos \theta | \ctg \beta + |\sin \theta |\tg \beta \bigr).
\end{equation}
\end{lemma} 
Для применения этой леммы \ref{l:vbpr} будем использовать поворот  плоскости $\CC$ 
на $\pi/2$ радиан <<против часовой стрелки>> с соответствующими заменами переменных для исходных функции $v$ и её меры Рисса $\nu$. По числу  $\e\in (0,1]$ выберем число $\beta \in (0,\pi/4]$ так, что $\tg \beta=\e$ и $\ctg \beta =1/\e$. По условию \eqref{c:nuvtp} <<повёрнутую>> меру $\nu$  можно представить в виде суммы мер $\nu :=\mu+\nu_0$, где мера $\mu:=\nu\bigm|_{U_{\beta}}$ удовлетворяет условию \eqref{ugoliR}, а $\nu_0\overset{\eqref{c:nuvtp}}{:=}\nu\bigm|_{\CC\setminus U_{\beta}}$ --- мера нулевого типа, или минимальной верхней плотности, при порядке $1$.  Если в обозначениях леммы \ref{l:vbpr} оказалось, что $\Delta_{\mu}=0$, то приходим к уже рассмотренному случаю $b=0$. Поэтому можем считать, что $0<\Delta_{\mu}\leq eb$. 
По лемме \ref{l:vbpr} найдётся целая функция $g_{\beta}\neq 0$, для которой выполнено  \eqref{ots_lu} с функцией $g_{\beta}$ вместо $g$. По доказанной части теоремы \ref{thmA}  с $b=0$ для  субгармонической функции $u_0:=v-u_{\mu}$ с мерой Рисса $\nu_0\in \mc M^+(\CC)$ нулевого типа найдётся 
целая функция $g_0\neq 0$, для которой $u_0+\log |g_0|$ --- субгармоническая функция 
 нулевого типа при порядке $1$. Для  функции $g:=g_{\beta}g_0\in \Hol_*(\CC)$ получаем требуемое 
 \eqref{hvge}, <<повёрнутое>> на $-\pi/2$.
\end{proof} 

\begin{remark} В теореме \ref{thmA} невозможно добиться ограниченности функции $u+\log |g|$ на $\RR$, так как в этом случае из субгармонического аналога легкой
необходимой части теоремы Берлинга\,--\,Мальявена о мультипликаторе функция $u\in \sbh_*(\CC) $ обязана быть функцией класса Картрайт \cite{MS}, \cite{BTKh}, , т.\,е.  должен сходиться интеграл $J_2^+[v]$ из  \eqref{J2+}, что строго сильнее  условия принадлежности функции $v$ классу A \cite[теорема 12]{Levin56}. 
\end{remark} 

Некоторое усиление \cite[следствие 3]{Kha01l} и теоремы A3 даёт 
 \begin{corollary}\label{corcomp} Пусть выполнены условия теоремы\/ {\rm A3}. Если 
 для любого $\beta \in (0,\pi/4]$ для пары вертикальных углов $U_{\beta}$
 из  \eqref{ang_e} сужение $\sf Z$ на $\CC \setminus U_{\beta}$ --- последовательность нулевой верхней плотности при порядке\/ $1$, то система $\Exp^{i\sf Z}$  не полна в $\Hol (\mathcal O)$. 
\end{corollary} 
Доказательство опускаем, поскольку оно выводится  из теоремы \ref{thmA} 
как \cite[доказательтво следствия 3]{Kha01l} и близк\'о 
к изложению  последних параграфов в \cite{Kha89}, \cite{kh91AA}
применительно к (полу)полосам. Часто употребляемые схемы для этого содержатся в \cite[1.1, 3.2]{Khsur}.

\end{document}